%% file: ay.tex
\documentclass[11pt]{article}
\usepackage[a4paper, margin={1.2in}]{geometry}
\usepackage[utf8]{inputenc}
\usepackage{authblk}

\usepackage{microtype}

\usepackage[T1]{fontenc}
\usepackage{amssymb,amsfonts,amsmath,amsthm,amscd}
\usepackage{bussproofs}
\usepackage[all]{xy}
\usepackage{comment}
\usepackage{tikz-cd}
\usetikzlibrary{arrows.meta}
\usepackage{stmaryrd}
\usepackage[shortlabels]{enumitem}
\usepackage{color,mathdots,setspace,framed,graphicx,natbib,latexsym,nicefrac,longtable,mathtools}
\usepackage[textwidth=3.5cm]{todonotes}
\usepackage{graphicx}
\usepackage{lipsum}
\usepackage{scalerel}
\usepackage{svg}
\usepackage{lipsum}
\usepackage{soul,xcolor}

\usepackage{tikz}
\usetikzlibrary{ipe}

\usepackage{chngcntr}
\counterwithin{figure}{section}

\usepackage[colorlinks=false]{hyperref}

\newtheorem{theorem}{Theorem}[section]
\newtheorem{lemma}[theorem]{Lemma}
\newtheorem{prop}[theorem]{Proposition}
\newtheorem{corollary}[theorem]{Corollary}

\theoremstyle{definition}
\newtheorem{definition}[theorem]{Definition}

\theoremstyle{remark}
\newtheorem*{remark}{Remark}

\newtheorem{question}{Question}
\DeclareMathOperator{\ACA}{ACA}

\newcommand{\BS}[1]{\operatorname{B\Sigma}_{#1}}
\DeclareMathOperator{\CA}{CA}
\newcommand{\qcr}[1]{\ulcorner #1\urcorner}

\newcommand{\Con}[1]{{\operatorname{Con}\!\left(#1\right)}}
\DeclareMathOperator{\CT}{CT}
\DeclareMathOperator{\DLO}{DLO}
\DeclareMathOperator{\EA}{I\Delta_{0}+{\exp}}
\newcommand{\form}{\operatorname{Form}}
\newcommand{\gn}[1]{\operatorname{\mathsf{N}}_{#1}}
\newcommand{\hn}[1]{\operatorname{\mathsf{K}}_{#1}}
\newcommand{\HH}{\mathcal{H}}
\DeclareMathOperator{\id}{id}
\newcommand{\ihn}[1]{\operatorname{\mathsf{J}}_{#1}}
\newcommand{\IS}[1]{\operatorname{I\Sigma}_{#1}}
\newcommand{\JJ}{\mathcal{J}}
\newcommand{\KK}{\mathcal{K}}

\DeclareMathOperator{\LPA}{\mathcal{L}_{\PA}}

\newcommand{\MM}{\mathcal{M}}
\DeclareMathOperator{\N}{\mathbb{N}}
\DeclareMathOperator{\name}{name}
\newcommand{\NN}{\mathcal{N}}
\newcommand{\num}[1]{\underline{#1}}

\DeclareMathOperator{\PA}{PA}
\DeclareMathOperator{\PAm}{PA^{--}}
\newcommand{\Proof}{\operatorname{Proof}}

\newcommand{\rfn}{\operatorname{RFN}}

\newcommand{\Sat}[1]{\operatorname{Sat}_{#1}}
\newcommand{\sn}[1]{\operatorname{IS}(#1)}
\DeclareMathOperator{\subst}{subst}

\newcommand{\Th}[1]{\operatorname{Th}(#1)}
\newcommand{\Thr}[2]{\operatorname{Th}_{#1}(#2)}

\newcommand{\tn}[1]{\operatorname{IT}(#1)}
\newcommand{\tnd}[1]{\operatorname{ITD}(#1)}
\newcommand{\tnf}[1]{\operatorname{ITF}(#1)}
\DeclareMathOperator{\tp}{tp}
\newcommand{\tuple}[1]{\langle #1 \rangle}

\DeclareMathOperator{\Z}{Z}
\DeclareMathOperator{\Zt}{Z_{2}}
\DeclareMathOperator{\ZF}{ZF}
\DeclareMathOperator{\ZFC}{ZFC}
\DeclareMathOperator{\Prov}{Prov}

\newcommand{\as}[2]{\forall #1 \! < \! #2 \,}

\tikzset{
  symbol/.style={
    draw=none,
    every to/.append style={
      edge node={node [sloped, allow upside down, auto=false]{$#1$}}}
  }
}

\mathchardef\mhyphen="2D

\title{
    Tightness and solidity\\ 
    in fragments of Peano Arithmetic\\
}
\author{Piotr Gruza\thanks{University of Warsaw, Doctoral School of Natural and Exact Sciences, \texttt{p.gruza3@uw.edu.pl}} , 
Leszek Aleksander Kołodziejczyk\thanks{University of Warsaw, Institute of Mathematics, \texttt{lak@mimuw.edu.pl}}
~and Mateusz Łełyk\thanks{University of Warsaw, Faculty of Philosophy, \texttt{mlelyk@uw.edu.pl}}}

\date{\today}

\begin{document}

\maketitle


\begin{abstract}
    It was shown by Visser that Peano Arithmetic has the property that any two bi-interpretable extensions of it (in the same language) are equivalent. Enayat proposed to refer to this property of a theory as \emph{tightness} and to carry out a more systematic study of tightness and its stronger variants that he called neatness and solidity.

    Enayat proved that not only $\PA$, but also $\ZF$ and $\Zt$
    are solid. On the other hand, it was shown in later work by a number of authors that many natural proper fragments of those theories are not even tight.

    Enayat asked whether there is a proper solid subtheory of the theories listed above. We answer that question in
    the case of $\PA$ by proving that for every $n$,
    there exist both a solid theory and a tight but not neat theory strictly between $\IS{n}$ and $\PA$. Moreover, the solid subtheories of $\PA$ can be required to be unable to interpret $\PA$.
    
    We also provide simple examples of proper solid subtheories of $\ZF$ and $\Zt$, as well as further separations between properties related to tightness, including an example of a sequential theory that is neat but not semantically tight in the sense of Freire and Hamkins.
\end{abstract}

\section{Introduction}

Our aim in this paper is to show that a potential very general characterization of Peano Arithmetic ($\PA$) as an axiomatic theory does not work.

To understand the background behind the potential characterization, recall the notion of interpretation (precise definitions of all relevant concepts will be provided in Section \ref{sec:prelim}). Intuitively speaking,
a structure $\MM$ interprets 
a structure $\NN$ if the universe, relations and operations of $\NN$ can be defined in $\MM$, where the universe of $\NN$ can consist of tuples of elements of $\MM$ rather than single elements, and equality in $\NN$ can be an equivalence relation on $\MM$ other than equality. Well-known examples include the interpretation of the field of rationals in the ring of integers, where rationals are given as pairs of integers $\tuple{k,\ell}$ with $\ell \neq 0$ and $\tuple{k,\ell}$ is identified with $\tuple{p,q}$ if $kq = p\ell$; and the interpretation of the field of complex numbers in the field of reals, where $a + bi$ is given as the pair of reals $\tuple{a,b}$.

An interpretation of an axiomatic theory $T$ in a theory $S$ is essentially a uniform recipe for interpreting a model of $T$ in a model of $S$. A pair of interpretations, of $T$ in $S$ and of $S$ in $T$, forms a bi-interpretation if the interpretations are provably mutually inverse, i.e.~each theory proves that composing the interpretations in the appropriate order gives rise to a structure
isomorphic to the original model of the theory. 
The concept of bi-interpretability was originally introduced
in model theory (see \cite{az:quasi-finitely}) but plays an increasingly meaningful role in foundational investigations (cf.~e.g. \cite{pakhomov:escape-tennenbaum, fv:synonymy}). 
Bi-interpretability between theories is much stronger than mutual interpretability: a well-known example is provided by the theories ZF and ZFC + GCH, which are mutually interpretable but not bi-interpretable. In contrast, $\PA$ is bi-interpretable with an appropriate formulation of finite set theory.  

In general, bi-interpretability and provability do not go hand in hand: for example one easily finds examples of theories that are bi-interpretable and yet mutually inconsistent. 
However, Visser \cite{visser:categories} showed that 
there are incomplete first-order theories, such as $\PA$, which make bi-interpretability collapse to logical equivalence
on the class of their extensions. 
More precisely, Visser proved that $\PA$ has the following curious property, later called \emph{tightness}: if two extensions $T_1$ and $T_2$ of $\PA$, still in the same language, are bi-interpretable, then in fact $T_1 \equiv T_2$.
Note that every complete theory will be tight in this sense. In fact, tightness is a kind of internal completeness property: a complete theory is one whose models are all elementarily equivalent, while if a theory $T$ is tight and a pair of interpretations $\mathsf{I}$ and $\mathsf{J}$ gives a bi-interpretation of $T$ with itself, then $T$ can prove that the structures described by $\mathsf{I}$ and by $\mathsf{J}$ are elementarily equivalent to one another, and in fact to the original model of $T$ in which the interpretations were applied.

Enayat \cite{enayat:variations} initiated a more systematic study of tightness and related concepts. In particular, he introduced semantical variants of tightness in which one considers interpretations between models of $T$ rather than between theories extending $T$. The strongest property that he considers, called \emph{solidity}, requires that any two models $\MM, \NN \vDash T$ have to be definably isomorphic as soon as they satisfy a weaker form of bi-interpretability (essentially, one of the two interpretations between $\MM$ and $\NN$ is only assumed to be a one-sided rather than two-sided inverse of the other). Thus, solidity is an internalized categoricity property rather than mere internalized completeness.

Enayat showed that $\PA$ is not just tight but also solid, and so are other foundationally important axiom schemes such as $\ZF$ set theory and second-order arithmetic $\Zt$.
On the other hand, it was gradually realized that natural proper fragments of those theories are not even tight. In particular:
\begin{enumerate}[(a)]
\item neither Zermelo set theory $\Z$ nor $\ZFC^-$ (i.e.~$\ZFC$ without Power Set and with collection instead of replacement) is tight \cite{fh:bi-interpretation-set};
\item for each $n$, the fragment $\Pi^1_n\text{-}\CA$ of $\Zt$ is not tight 
\cite{fw:non-tightness};
\item for each $n$, no $\Pi_n$-axiomatized
fragment of true arithmetic (thus, \emph{a fortiori}, of $\PA$) is tight \cite{el:categoricity-like}.
\end{enumerate}

Freire and Hamkins \cite{fh:bi-interpretation-set} noted that the proofs of tightness and solidity for set theory ``seem to use the full strength of ZF''. Similarly, Freire and Williams \cite{fw:non-tightness} referred to their results as ``evidence that tightness characterizes $\Zt$ (...) in a minimal way''. This situtation gave salience to a question asked already by Enayat \cite[Question 3.2]{enayat:variations}: do any any of $\PA$, $\ZF$, $\Zt$ have a proper solid subtheory?

In the context of 
set theory, it was stated in \cite{fw:non-tightness} that settling Enayat's question (presumably in the negative) would amount to ``a profound characterization of ZF''.
\emph{Mutatis mutandis}, this would be even more true in the case of $\PA$, due to both the very basic nature of first-order arithmetic and to the special role played by the induction scheme in tightness arguments: it was already observed in \cite[just before Question 3.1]{enayat:variations} that all solid theories known up to that point that interpret a minimal amount of arithmetic also imply the full induction scheme under that interpretation.

In our view, the ``right'' question to ask is not quite whether the theories mentioned above have \emph{arbitrary} solid subtheories. Taken literally, that question is vulnerable to trivial counterexamples: for instance, it is not difficult to show that ``either the axioms of ZF hold, or the universe has one element and $\in$ is the empty relation'' is a solid theory. Rather, one should ask whether there are solid proper subtheories that imply some reasonably strong axioms giving the scheme at hand its appropriate (arithmetical or set-theoretic) character: say
$\IS{1}$ or $\EA$ in the case of $\PA$ and Zermelo set theory in the case of ZF. (One not quite trivial solid subtheory of ZF that still seems to ``leave out too much'' is ZF without infinity but with the axiom that every set has a transitive closure; see \cite[Theorem 18]{el:categoricity-like}.) 

Here, we show that even after such an arguably natural modification, the answer to Enayat's Question 3.2 for $\PA$ is positive; as a consequence, we show that there
can be arithmetical solid theories that do not imply the full induction scheme. More precisely, we prove that for every $n$ there is a solid theory strictly between $\IS{n}$ and $\PA$. Our examples have a disjunctive nature like the trivial one above, but their construction is considerably more involved: essentially, they state ``either $\PA$ holds, or we are in a particular pointwise definable model of ${\IS{n}} + {\neg \IS{n+1}}$''.
To make this work, we have to ensure both that the theory
of those pointwise definable models is solid, and that it has a well-calibrated interpretability strength sufficiently different from (in our case: greater than) that of $\PA$.

The upshot of our result is that there is no hope of characterizing $\PA$ as a minimal solid theory, at least in the realm of theories ordered by logical implication. In fact, we are able to show that the same holds true for the coarser (pre-)order of interpretability as well. Nevertheless, our work points to a more subtle sense in which it remains open whether $\PA$ could be minimal solid; we discuss this briefly in Section \ref{sec:conclusion} of the paper.

Another very natural question related to tightness, solidity and their cousins is whether these concepts are actually distinct. Again, a naive version of this question admits some relatively trivial positive answers, due to the fact that syntactically defined properties like tightness apply to all complete theories, whereas semantically defined ones like solidity in general do not. However, separating two syntactically defined tightness-like properties, or getting any separating example at all that would be a computably axiomatized theory subject to G\"odel's theorems,
was a significant challenge. We provide the first separations ``of the nontrivial kind'', showing for example that there are arbitrarily strong proper subtheories of $\PA$ that distinguish tightness from a property that is likewise syntactically defined but has the bi-interpretability assumption weakened as in the case of solidity.

To prove our results, we need constructions that ensure the existence of some specific interpretations, isomorphisms etc.~but not of others. For such purposes, we rely on a wide variety of methods from the model and proof theory of arithmetic. The tools we make use of include, among other things: axiomatic truth theories, flexible formulas, 
pointwise definable models, and a very weak pigeonhole principle known as the cardinality scheme.

The remainder of the paper is organized as follows.
We review basic background 
concepts and facts in Section \ref{sec:prelim}. 
In Section \ref{sec:groundwork}, we discuss and develop 
some more advanced background material.
In Section \ref{sec:solidity-below}, we the prove our result
on solid proper subtheories of $\PA$ in its basic form, and in Section \ref{sec:refinements} we present some refinements.
In Section \ref{sec:separation}, we prove separations
between tightness, solidity, and other similar properties.
Section \ref{sec:soa} discusses some proper solid subtheories of $\ZF$ and $\Zt$, including a proper solid subtheory of $\Zt$ containing $\ACA$. We summarize our work and state some open problems in Section \ref{sec:conclusion}.

\section{Preliminaries}\label{sec:prelim}

A few general conventions: to avoid irrelevant complications, all languages considered in this paper are finite.
If a theory $T$ is fixed or clear from the context, then  $\mathcal{L}_T$ denotes the language of $T$. 
Unless explicitly stated otherwise, whenever we write $\MM \vDash T$ or say that $\MM$ is a model of $T$, we tacitly assume that $\MM$ is an $\mathcal{L}_T$-structure and not one in a strictly larger language. (Without this assumption, we would have to replace ``model of $T$'' with the cumbersome ``$\mathcal{L}_T$-reduct of a model of $T$'' in all sorts of statements about interpretability properties of structures -- see e.g.~Definitions \ref{def:sem-tight}--\ref{def_solidity} or Lemma \ref{lem:no-mix-ct-pa}.)
Given a formula $W(x)$, we may sometimes write $x\in W$ instead of $W(x)$,
mainly in order to be able to substitute $\forall x \! \in \! W \ldots$ for the more cumbersome $\forall x\,(x \in W \rightarrow {\ldots})$.

\subsection{Interpretability}

We expect that the reader has at least an intuitive understanding of what interpretations are in logic. 
As mentioned in the introduction, an interpretation of a structure $\NN$ in a structure $\MM$ is roughly a definition of the domain and relations of $\NN$ inside $\MM$, while an interpretation of a theory $T$ in a theory $S$ is a uniform recipe for interpreting models of $T$ in models of $S$. However, since we are going to prove theorems about interpretations and interpretability, we need to be somewhat precise, and thus we provide a more formal discussion of these concepts. Those formal details can be rather boring, so the reader may consider just skimming the present subsection at first and referring back to it as needed.

For the purpose of giving an official definition of interpretations, we pretend that all languages are purely relational. There is no harm in doing so, thanks to the usual transformation of arbitrary languages into relational ones that replaces each $n$-ary function symbol with an $(n+1)$-ary relational symbol. It is known that this transformation  applies not only to formulas, but also induces a (feasible) transformation of proofs in a theory $T$ into proofs
in its relational analogue $T^\mathrm{rel}$; for details, see e.g.~\cite[Section 7.3]{visser:inside-exp}. 

\paragraph{Translations.} The formal definition of interpretation starts with the
notion of \textit{translation}. If $\mathcal{L}_1$ and $\mathcal{L}_2$ are first-order languages,
then a
\emph{translation} $\mathsf{M}$ from $\mathcal{L}_1$ to $\mathcal{L}_2$ is determined by specifying:

    \begin{enumerate}[(i)]
        \item a unary $\mathcal{L}_2$-formula $\delta_{\mathsf{M}}(x)$, sometimes called 
        the \emph{domain formula};
        
        \item for every $n$-ary relation symbol $P$ of $\mathcal{L}_1$ (including equality), an $\mathcal{L}_2$-formula $P^{\mathsf{M}}(\bar{y})$ with exactly $n$ free variables, such that $\vdash 
        P^{\mathsf{M}}(y_1,\ldots,y_n)\rightarrow \bigwedge_{i\leq n}\delta_{\mathsf{M}}(y_i)$.
    \end{enumerate}

If $\Gamma$ is a class of $\mathcal{L}_2$-formulas, then we say that the translation $\mathsf{M}$ is $\Gamma$-\emph{restricted} if $\delta_{\mathsf{M}}$ and all the formulas $P^\mathsf{M}$ belong to $\Gamma$.

\begin{remark}
The intention behind the definition of translation is that
$\delta_\mathsf{M}$ defines a domain of $\mathcal{L}_1$-objects on which the formulas $P^\mathsf{M}$ define $\mathcal{L}_1$-relations. Our definition of translation (and the definition(s) of interpretation based on it, given below) is not the most general one possible: in particular, our translations are one-dimensional, in the sense that the domain formula always has just one free variable. For our purposes, this is inessential, because (almost) all the theories we consider support a pairing function. For an example (rather distant from our main topic) of a situation in which multi-dimensional interpretations would matter, see the Remark in Section \ref{subsec:tame}.

The requirement that $P^{\mathsf{M}}(y_1,\ldots,y_n)$  logically imply $\delta_\mathsf{M}(y_i)$ for each $i$ is a technical condition that is sometimes useful, and it can be assumed to hold without loss of generality in all contexts that will be relevant to us. So, we simplify things by including it in the definition of translation.  
\end{remark}

Given a {translation} $\mathsf{M}$ from $\mathcal{L}_1$ to $\mathcal{L}_2$, we define the translation $\varphi^\mathsf{M}$
of an $\mathcal{L}_1$-formula $\varphi$ as follows:
$(P(\bar{x}))^{\mathsf{M}}$ is $P^{\mathsf{M}}(\bar{x})$ for a relation symbol $P$ of $\mathcal{L}_1$; the translation commutes with propositional connectives; and $(\forall x\,\psi)^{\mathsf{M}}$ is $\forall x \bigl(\delta_{\mathsf{M}}(x)\rightarrow \psi^{\mathsf{M}}\bigr)$.

\paragraph{Interpretations of structures.} If $\mathsf{M}$ is a translation of $\mathcal{L}_1$ into $\mathcal{L}_2$ and $\NN$ is an $\mathcal{L}_2$-structure, then $\mathsf{M}$ is a \emph{parameter-free interpretation in $\NN$} if $(\delta_{\mathsf{M}})^\NN$ is nonempty and $(=^\mathsf{M})^\NN$ is an equivalence relation on $(\delta_{\mathsf{M}})^\NN$ that is a congruence w.r.t.~each relation $(P^\mathsf{M})^\NN$. In this case, $\mathsf{M}$ and $\NN$ uniquely determine an $\mathcal{L}_1$-structure
whose universe is the set of equivalence classes of $(=^\mathsf{M})^\NN$ and whose relations are as determined by 
$(P^\mathsf{M})^\NN$. We denote this structure by $\NN^\mathsf{M}$, and we say that a model $\mathcal{M}$ is \emph{interpreted without parameters in $\mathcal{N}$ via $\mathsf{M}$}, if $\mathcal{M} = \NN^\mathsf{M}$. 
In general, an \emph{interpretation $\mathsf{M}$ in $\NN$} can be based on formulas $\delta_\mathsf{M}$ and $P^\mathsf{M}$ that use parameters from $\NN$: in other words, an interpretation in $\NN$ is essentially the same thing as a parameter-free interpretation in $(\NN,\bar c)$ for some tuple of constants $\bar c$. We write $\mathsf{M}: \NN \rhd \MM$ to indicate that $\mathsf{M}$ is an interpretation of the structure $\MM$ in $\NN$ (so in particular $\MM = \NN^\mathsf{M}$). If the interpretation is clear from context or unimportant, we then omit the reference to it, writing simply $\NN \rhd \MM$ and saying that $\MM$ is interpreted in $\NN$. We say that a model $\mathcal{M}$ is \emph{interpretable} (as opposed to ``interpreted'') 
in $\mathcal{N}$ if there is an interpretation $\mathsf{M}$ in $\mathcal{N}$ such that $\mathcal{M}$ is isomorphic to $\mathcal{N}^{\mathsf{M}}$.

\begin{remark}
One easily notices that $\MM$ is interpretable (resp.~parameter-free interpretable) in $\NN$ if and only if there is a surjection from $\NN$ onto $\MM$ such that the preimage of every parameter-free $\MM$-definable set is definable 
(resp.~parameter-free definable) in $\NN$. Hence our definition  of interpretability is equivalent to the classical model-theoretic one, cf.~e.g. \cite{az:quasi-finitely, hodges:model-theory}. In our context it will be easier to work with the more syntactic approach presented above.
\end{remark}

\begin{remark}
Note that if a translation $\mathsf{M}:\mathcal{L}_1\rightarrow \mathcal{L}_2$ is identity preserving, which means that $x=^\mathsf{M}y$ is $x=y$, then $\mathsf{M}$ is an interpretation in $\NN$ for every $\mathcal{L}_2$-structure $\NN$ for which
$(\delta_{\mathsf{M}})^\NN$ is nonempty.
\end{remark}

\paragraph{Interpretations of theories.} Let $\mathsf{M}$ be a translation from $\mathcal{L}_1$ into $\mathcal{L}_2$, and for $i = 1,2$ let $T_i$ be an $\mathcal{L}_i$-theory. Then $\mathsf{M}$ is an \emph{interpretation of $T_1$ in $T_2$} if $T_2$ proves
that $=^\mathsf{M}$ is an equivalence relation on $\delta_{\mathsf{M}}$ that is a congruence w.r.t.~each relation $P^\mathsf{M}$, and $T_2$ also proves $\varphi^\mathsf{M}$ for each axiom $\varphi$ of $T_1$ (including logical axioms, in particular the non-emptiness of the universe).
We then say that $T_2$ \emph{interprets $T_1$ via $\mathsf{M}$}. Thus, an interpretation of $T_1$ in $T_2$
is given by a fixed set of formulas that provide a parameter-free interpretation of a model of $T_1$ in each 
$\NN \vDash T_2$. 
We write $T_2 \rhd T_1$ ({resp}.~$\mathsf{M}: T_2 \rhd T_1$) to indicate that $T_2$ interprets $T_1$ ({resp}.~via $\mathsf{M}$).
We say that a translation $\mathsf{M}$ is an \emph{interpretation in $T_2$} if it is an interpretation
of the empty theory over the appropriate language in $T_2$.

\paragraph{Composition and the identity interpretation.} Interpretations between structures, or between theories, can
be thought of as morphisms of a category, in that they can be composed and there is always an identity interpretation. 
For any language $\mathcal{L}$, the identity translation $\mathsf{id}_{\mathcal{L}}$ is given by letting  $\delta_{\mathsf{id}_{\mathcal{L}}}$ be $x=x$ and translating  each relation symbol of $\mathcal{L}$ to itself. 
For two translations 
$\mathsf{M} \colon \mathcal{L}_1 \to \mathcal{L}_2$ 
and $\mathsf{N}$ of $\mathcal{L}_2 \to \mathcal{L}_3$, 
we define $\mathsf{M}\mathsf{N} \colon \mathcal{L}_1 \to \mathcal{L}_3$
(note the order in which we write the composition of interpretations) as follows:
   \begin{itemize}
       \item $\delta_{\mathsf{MN}} :=\delta_{\mathsf{N}}\wedge (\delta_{\mathsf{M}})^{\mathsf{N}}.$
       \item $P^{\mathsf{MN}}(\bar{y}):= \bigwedge_{y\in \bar{y}} \delta_{\mathsf{N}}(y)\wedge (P^{\mathsf{M}}(\bar{y}))^{\mathsf{N}}$, for each
       $\mathcal{L}_1$-relation symbol $P$.
   \end{itemize}
 Then for every $\mathcal{L}_1$-formula $\varphi$, the formulas
 $\varphi^{\mathsf{MN}}$ and $(\varphi^{\mathsf{M}})^{\mathsf{N}}$ are logically equivalent. This is enough to define
 composition for interpretations in theories 
 and for parameter-free interpretations in structures.
 If $\MM$ is a model, $\mathsf{N}_1$ is an interpretation
 with parameters in $\MM$, and $\mathsf{N}_2$ is an interpretation with parameters in $\MM^{\mathsf{N}_1}$, 
 then the composition $\mathsf{N}_1\mathsf{N}_2$ is also routinely defined, and it is \emph{unique up to equivalence 
 in $\MM$}: since the parameters used by $\mathsf{N}_2$ correspond to equivalence classes of the $\MM$-definable equivalence relation $=^{\mathsf{N}_1}$, one should choose some representatives of these classes. Clearly any choice of these representatives is good, since $=^{\mathsf{N}_1}$ is a congruence w.r.t. all definable relations of $\MM^\mathsf{N}_1$.
 
 \paragraph{Isomorphism of interpretations.} To define the crucial notion of bi-interpretation and its weaker version, retraction, we need to say what it means for interpretations to be isomorphic.
 Given a language $\mathcal{L}$ and two interpretations $\mathsf{M}_1$, $\mathsf{M}_2$ of $\mathcal{L}$-structures 
 in a structure $\NN$, we say that $\mathsf{M}_1$, $\mathsf{M}_2$ are \emph{$\NN$-isomorphic} 
 (resp.~\emph{parameter-free $\NN$-isomorphic})
 if there is a definable (resp.~parameter-free definable) relation $\iota\subseteq N^2$ such that the domain of $\iota$ is $(\delta_{\mathsf{M}_1})^{\NN}$, the range is $ (\delta_{\mathsf{M}_2})^{\NN}$, 
 and $\iota$ preserves all $\mathcal{L}$-predicates 
 (including $=$). This means that $\iota$ canonically determines an isomorphism between $\NN^{\mathsf{M}_1}$ and $\NN^{\mathsf{M}_2}$. Note that for any interpretation
 $\mathsf{K}$ in $\NN^{\mathsf{M}_1}$,
 the isomorphism
 $\iota$ gives rise to a corresponding interpretation
 $\iota[\mathsf{K}]$ given by the same formulas as $\mathsf{K}$ with parameters shifted by $\iota$
 (so, if $\mathsf{K}$ involves no parameters,
 then $\iota[\mathsf{K}]$ is the same as $\mathsf{K}$).
 
 One can define the notion of embedding between interpretations  in a similar way. 

 \paragraph{Retractions and bi-interpretations of structures.}
    Let $\MM$ and $\NN$ be two structures,
    let $\mathsf{N}$ be an interpretation in $\MM$ and
    let $\mathsf{M}$ be an interpretation in $\MM^{\mathsf{N}}$.
    
    \begin{itemize}
        \item We say that the pair $(\mathsf{N},\mathsf{M})$
        forms a \emph{retraction in $\MM$} if $\mathsf{N}\mathsf{M}$ is $\MM$-isomorphic to the identity interpretation.
        
        \item We say that $\MM$ is \emph{a retract of} $\NN$ if there is a retraction $(\mathsf{N}$, $\mathsf{M})$ in $\MM$ such that $\NN$ is isomorphic to $\MM^{\mathsf{N}}$.
        \item We say that $\mathsf{M}$ and $\mathsf{N}$ form  a \emph{bi-interpretation in $\MM$} if $\mathsf{N}\mathsf{M}$ is $\MM$-isomorphic to the identity interpretation on $\MM$ 
        via isomorphism $\iota$, and $\mathsf{M}\iota[\mathsf{N}]$ is $\MM^{\mathsf{N}}$-isomorphic to the identity interpretation on $\MM^{\mathsf{N}}$.
        \item We say that $\MM$ and $\NN$ 
        are \emph{bi-interpretable} if 
        there is a bi-interpretation 
        $(\mathsf{N}$, $\mathsf{M})$ in $\MM$ 
        such that $\NN$ is isomorphic to $\MM^{\mathsf{N}}$.
    \end{itemize}
Figure \ref{fig:retract} 
below illustrates a retraction.
\begin{figure}
\centering
\begingroup
\tikzset{every picture/.style={scale=0.7, transform shape}}
\input{retraction.tex}
\endgroup
\caption{A retraction. 
The squares and rectangles comprising 
the interpreted structures $\MM^{\mathsf{N}}$ and $\MM^{\mathsf{NM}}$ correspond to the equivalence classes
of $=^{\mathsf{N}}$ and $=^{\mathsf{NM}}$.
The arrow indicates an $\MM$-definable isomorphism between $\MM$ and $\MM^{\mathsf{NM}}$.}
\label{fig:retract}
\end{figure}

\begin{remark}
By an easy argument one can see that bi-interpretability is actually symmetric, which need not be directly obvious from the definition. Similarly, it is equivalent to the standard notion studied e.g.~in \cite{az:quasi-finitely}.
\end{remark}

\begin{remark}\label{rem:biint_iso_aut}
We say that $\MM$ and $\NN$ are \emph{parameter-free bi-interpretable} if the interpretations and isomorphism needed for the bi-interpretation are given by parameter-free formulas. It is reasonably straightforward to show that if
$\MM$ and $\NN$ are {parameter-free bi-interpretable},
then the automorphism groups of $\MM$ and $\NN$ are isomorphic.  
Note that this is not true for general (parametric) bi-interpretability. 
\end{remark}

\paragraph{Retracts and bi-interpretability between theories.}
 Unsurprisingly, the concepts of bi-interpretability and being a retract have their analogues for theories as well, expressing that the appropriate compositions of interpretations are isomorphic to the identity provably
in the appropriate theories. 
Given two interpretations $\mathsf{M}_1, \mathsf{M}_2$ of a language $\mathcal{L}$ in a theory $T$, we say that 
$\mathsf{M}_1$ and $\mathsf{M}_2$ are \emph{isomorphic in $T$} if there is a binary $\mathcal{L}_T$-formula $\iota(x,y)$ which
$T$-provably determines an isomorphism between 
$\mathsf{M}_1$ and $\mathsf{M}_2$.
That is, provably in $T$ the relation 
defined by $\iota$ has $\delta_{\mathsf{M}_1}$ as its domain,
$\delta_{\mathsf{M}_2}$ as range, and preserves all the relations of $\mathcal{L}$, including equality.

Let $T_1$, $T_2$ be theories in languages $\mathcal{L}_1, \mathcal{L}_2$, respectively.
    \begin{itemize}
        \item We say that $T_1$ is a \emph{retract} of $T_2$ if there are translations $\mathsf{M}_1:\mathcal{L}_1\rightarrow \mathcal{L}_2$ and $\mathsf{M}_2:\mathcal{L}_2\rightarrow \mathcal{L}_1$ such that the theory $T_1$ interprets $T_2$ via $\mathsf{M}_2$, the theory $T_2$ interprets $T_1$ via $\mathsf{M}_1$,
        and $\mathsf{id}_{\mathcal{L}_1}$ is isomorphic to $\mathsf{M}_2\mathsf{M}_1$ in $T_1$.
        
        \item We say that $T_1$ and $T_2$ are \emph{bi-interpretable} if there are translations $\mathsf{M}_1:\mathcal{L}_1\rightarrow \mathcal{L}_2$ and $\mathsf{M}_2:\mathcal{L}_2\rightarrow \mathcal{L}_1$ witnessing that $T_1$ is a retract of $T_2$ and $T_2$ is a retract of $T_1$.
    \end{itemize}

\subsection{Categoricity-like notions for first-order theories}

Below we recall four categoricity- and completeness-like properties which emerged in the literature. The concepts of solidity, tightness and neatness were introduced in \cite{enayat:variations}, while semantical tightness was considered for the first time in \cite{fh:bi-interpretation-set}. These notions can be seen to arise by making independent choices with respect to two independent questions:
\begin{enumerate}
    \item Do we want a semantical or a syntactic property?
    \item Do we want to use the notion of retraction or the notion of bi-interpretation?
\end{enumerate}

Perhaps the most natural of the four properties is the one that was introduced last, semantical tightness. As the name suggests, this is a semantical property, and it is based on the notion of bi-interpretability between structures.

\begin{definition}[Semantical tightness]\label{def:sem-tight}
A theory $T$ is \emph{semantically tight} if whenever $\MM$ is a model of $T$ and $(\mathsf{N},\mathsf{M})$ is a bi-interpretation in $\MM$ such that $\MM^\mathsf{N}\vDash T$, then $\mathsf{N}$ is $\MM$-isomorphic to $\id_\MM$ 
(and, as a consequence, $\mathsf{M}$ is $\MM^{\mathsf{N}}$-isomorphic to $\id_{\MM^{\mathsf{N}}}$).
\end{definition}

\begin{remark}
We observe that the semantical tightness of a theory $T$ entails that the bi-interpretability relation between models of $T$ is trivial in the following sense: whenever $\MM$ and $\NN$ are models of $T$ and $\MM$ is bi-interpretable with $\NN$, then $\MM$ and $\NN$ are isomorphic.
\end{remark}

\begin{remark}
One can consider stronger and weaker notions of semantical tightness. For example, one could restrict the notion given above by insisting that the definition of the isomorphism between $\mathsf{N}$ and $\id_\MM$ do not use any parameters other than the ones involved in defining the interpretations and isomorphisms that give rise to the bi-interpretation.
In this paper, whenever we show the semantical tightness of a theory, it always holds in this more restrictive sense, and whenever we show failure of semantical tightness, it applies already in the weaker sense. So, our results do not depend on which of the two definitions was applied.

The original definition of semantical tightness in \cite{fh:bi-interpretation-set} is weaker still: 
the isomorphism between $\MM$ and $\NN$ need not be definable. We think that the definition proposed above is more in the spirit of the other notions considered in this paper (see the definition of solidity below). Moreover, our example of a theory which is neat but not semantically tight from Section \ref{subsec:neat-not-sem-tight} works also for 
definition used by \cite{fh:bi-interpretation-set}.
For an example of a situation in which the distinction
between isomorphism and definable isomorphism would matter, see the Remark in Section \ref{subsec:tame}.
\end{remark}

A stronger property, solidity, is also semantical but starts from notion of retraction.

\begin{definition}[Solidity]\label{def_solidity}
    We say that $T$ is \emph{solid} if whenever $\MM$ is a model of $T$ and $(\mathsf{N},\mathsf{M})$ is a retraction in $\mathcal{M}$ such that $\MM^{\mathsf{N}}\vDash T$, then $\mathsf{N}$ is $\MM$-isomorphic to $\id_{\MM}$.
\end{definition}

\begin{remark}
In analogy to the case of semantical tightness, we can observe that the solidity of a theory $T$ trivializes
the retraction relation between models of $T$: whenever $\MM$ and $\NN$ are models of $T$ and $\MM$ is a retract of $\NN$, then $\MM$ and $\NN$ are isomorphic.
\end{remark}

\begin{remark}
Our definition of solidity is equivalent to the original one given in \cite{enayat:variations}. 
However, the later paper \cite{el:categoricity-like}
used the more restrictive definition according to which the 
definition of the isomorphism between $\mathsf{N}$ and $\mathsf{id}_{\MM}$ can only use the parameters 
involved in defining $\mathsf{N}, \mathsf{M}$ and the isomorphism between $\mathsf{id}_{\MM}$ and $\mathsf{N}\mathsf{M}$. As in the case of semantical tightness, our main results do not depend on which 
of the two definitions is adopted (see, however, Lemma 
\ref{lem:u-tight} in Section \ref{subsec:pam}).
\end{remark}

\begin{remark}
By an easy L\"owenheim-Skolem argument, it is sufficient to verify the semantical tightness/solidity of $T$ only on countable models of $T$.
\end{remark}

Now we pass to two notions of syntactical character, which are properly speaking more ``completeness-like'' than ``categoricity-like'', as they do not imply that certain models are isomorphic but merely that they are elementarily equivalent. 

\begin{definition}[Tightness]
    A theory $T$ is \emph{tight} if whenever $U\subseteq \mathcal{L}_T$ and $V\subseteq\mathcal{L}_T$ are two extensions of $T$ which are bi-interpretable, then $U \equiv V$. 
\end{definition}

\begin{definition}[Neatness]
    A theory $T$ is \emph{neat} if whenever $U\subseteq\mathcal{L}_T$ and $V\subseteq\mathcal{L}_T$ are two extensions of $T$ and $U$ is a retract of $V$, then $U\vdash V$.
\end{definition}
\begin{remark}
By an easy argument, it is enough to verify the tightness/neatness of $T$ only for $U, V$ that are complete extensions of $T$. Hence, in particular, a theory is tight if and only if the bi-interpretability relation on its complete extensions coincides with the identity.
\end{remark}

It is quite easy to see that all the arrows in the diagram below correspond to implications between the four properties:
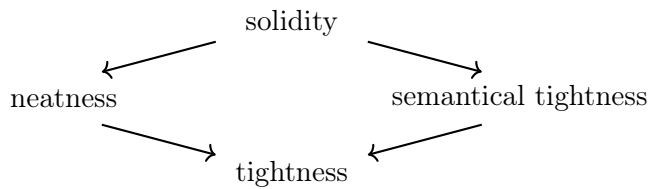
\begin{center}

\begin{tikzpicture}

\draw (0, 4.3) node {solidity};

\draw (-3, 3) node {neatness};

\draw (3, 3) node {semantical tightness};

\draw (0, 1.7) node {tightness};

\draw [thick, ->] (-2.5, 2.65) -- (-1,1.95);
\draw [thick, ->] (2.5, 2.65) -- (1,1.95);

\draw [thick, ->] (-1, 4.05) -- (-2.5,3.35);
\draw [thick, ->] (1, 4.05) -- (2.5,3.35);

\end{tikzpicture}
\end{center}

As for the question whether the implications are strict, see some results in Section \ref{sec:separation} and the discussion in Section \ref{sec:conclusion}.

\subsection{Basic first-order arithmetic}

In this paper we focus mainly
on theories in the language of arithmetic $\LPA$ (i.e.~the usual language of ordered rings), though occasionally we also consider extensions of $\LPA$ 
by finitely many predicates (see e.g.~Section \ref{sect::truth_theories}) or other languages (see Sections \ref{sec:separation} and \ref{sec:soa}).
Below, we review some basic properties of arithmetical theories that we will need later on, and we fix some notational conventions. A few more advanced topics related to first-order arithmetic are discussed in Section \ref{sec:groundwork}.

Standard notions and classical results in the model and proof theory of first-order arithmetic that we rely on can be found in the monographs \cite{hp:metamathem} and \cite{kaye:models}.

\paragraph{Arithmetical theories.} All the arithmetical theories that we consider extend $\PA^-$, the finitely axiomatized theory of non-negative parts of discretely ordered rings. 
It was shown by Je\v{r}\'abek \cite{jerabek:pam_sequential} that $\PA^-$ is a \emph{sequential} theory, which means roughly that it supports a reasonably behaved theory
of finite sequences of arbitrary elements.
(For a precise definition of sequentiality, 
an important general concept that was in fact discovered
in the study of interpretability \cite{pudlak:some-prime}, 
see e.g.~\cite[Section 2.4]{visser:small_is_very_small}.)
Sequentiality of $\PA^-$ implies that there is an interpretation $\mathsf{M}$ of $\PA^-$ in $\PA^-$, 
with the domain forming a \emph{definable cut} 
in $\PA^-$ (i.e., provably closed downwards under $\le$ under and successor) and the arithmetical operations translated identically,  
and there is a formula $y = x_z$ (intended to mean ``$y$ is the $z$-th element of the sequence $x$'') such that $\PA^-$ proves the statement:
\[\forall s,x,k\, \exists s' \, \forall i,y\, \bigl(\delta_{\mathsf{M}}(k)\wedge i \le k \rightarrow 
(y = s'_i \leftrightarrow (i<k \wedge y=s_i)  \vee (i=k \wedge y=x))\bigr).\]
This says that a given sequence $s$ can be extended/modified by inserting an arbitrary element $x$ in the position
indexed by an arbitrary number $k$ from the domain of $\mathsf{M}$.

Most of the systems we study extend $\EA$, 
also known as elementary arithmetic EA or elementary function arithmetic EFA, a well-known theory whose provably total functions are exactly the
elementary computable functions. This theory extends $\PA^-$ by the induction scheme for all $\Delta_0$ formulas, $\mathrm{I}\Delta_0$, and a single axiom stating that the exponential function is total. In general,
if $\Gamma$ is a class of formulas, then $\mathrm{I}\Gamma$ denotes the extension of $\PAm$ by all
instances of the induction scheme 
for formulas from $\Gamma$, 
while $B\Gamma$ denotes the extension of $\mathrm{I}\Delta_0$ by all
instances of the collection scheme for formulas from $\Gamma$.

\paragraph{Encoding of syntax and set theory.} The theory $\EA$ allows for a convenient encoding of finite sets via the Ackermann membership relation $\in_A$ (``the $x$-th bit in the binary notation for $y$ is 1''). Except for Section \ref{subsec:pam}, when a more fancy coding is required by the weaker setting, all encoding of set-theoretic and syntactical notions is done using $\in_A$. For a syntactical object $o$ (a formula, a term, a proof), $\qcr{o}$ denotes the G\"odel code of $o$.  We let $\mathrm{seq}(x)$ and $\mathrm{len}(s) = x$ be some fixed arithmetical formulas expressing (in terms of $\in_A$) that $s$ is a sequence and that the length of the sequence $s$ is $x$, respectively, and we let $y_x$ or $(y)_x$ stand for the $x$-th element of the sequence $y$. Given an arithmetization of some
language $\mathcal{L}$, we let $\form_{\mathcal{L}}(x)$, $\form^{\leq 1}_{\mathcal{L}}$, $\mathrm{Sent}_{\mathcal{L}}(x)$, $\mathrm{Term}_{\mathcal{L}}(x)$, $\mathrm{Var}(x)$ denote  the arithmetical formulas expressing respectively that $x$ is (the G\"odel code of) an $\mathcal{L}$-formula, 
an $\mathcal{L}$-formula with at most one free variable, an $\mathcal{L}$-sentence, an $\mathcal{L}$-term, and a variable.
Omitting the subscript $\mathcal{L}$ indicates that the intended language is $\LPA$. If $\Gamma$ is class of formulas, then $\form_{\Gamma}(x)$ denotes the arithmetical definition of $\Gamma$. 

We use the notation $\name(x)$ to denote
the arithmetical naming function, which given a number $x$ returns (the code of) a canonical numeral naming $x$, say of
\[\underbrace{(\ldots ((0+1)+1)\ldots +1)}_{x \textnormal{ additions }}.\]
We use $\mathrm{val}(t)$ for the function which given (the code of) a closed term $t$ outputs its value. Given (the codes of) a formula $\varphi$ a variable $v$ and a term $t$, $\subst(\varphi,t)$ denotes the result of substituting the term $t$ for all occurrences of the unique free variable of $\varphi$ (preceded by the renaming of bound variables so as to avoid clashes). The notation $\subst(\varphi,\name(x))$ is often simplified
using the dot convention:
instead of $\subst(\varphi,\name(x))$, we write $\qcr{\varphi(\dot{x})}$. More generally,
$\qcr{\cdots}$ often indicates the application of some syntactical function on codes of formulas: for example, 
$\qcr{\varphi\wedge \psi}$ denotes (the code of) the conjunction of given formulas $\varphi$ and $\psi$. 

\paragraph{Provability, reflection and partial truth predicates.} If $T$ is (an arithmetical definition of) a theory, then $\Prov_T(x)$ stands for the canonical provability predicate, i.e. provability in first-order logic with sentences from $T$
as additional axioms, and $\Proof_T(z,x)$ stands for the formula stating that $z$ is a proof of $x$ in $T$. The sentence $\Con{T}$ is $\neg\Prov_T(\qcr{0\! \neq \!0}).$ If $\Gamma$ is a class of formulas, then $\Gamma\mhyphen\rfn(T)$ denotes the theory extending $\EA$ by the uniform $\Gamma$-reflection scheme for $T$, that is, by all axioms of the form
\[\forall x \,\bigl(\Prov_{T}(\qcr{\varphi(\dot{x})})\rightarrow \varphi(x)\bigr),\]
where $\varphi\in\Gamma$ (we assume that $\varphi$ has at most
one free variable).
We write $\Gamma\mhyphen\Con{T}$ for the extension of $\EA$ by all sentences of the form
\[\forall x \,\bigl(\varphi(x)\rightarrow \Con{T+\qcr{\varphi(\dot{x})}}\bigr),\]
where $\varphi\in\Gamma$. 

\begin{remark}
Let $\Sigma_n$, $\Pi_n$ be the usual formula classes of the arithmetical hierarchy. 
It is easy to prove that $\Pi_n\mhyphen\rfn(T)$ is equivalent to $\Sigma_n\mhyphen\Con{T}$, and vice versa.
\end{remark}

For each $n \ge 1$
and $\Gamma\in\{\Sigma_n,\Pi_n\}$ there is a partial satisfaction predicate $\Sat{\Gamma}(\varphi,x)$ which satisfies the usual inductive Tarskian truth conditions for formulas from $\Gamma$ provably in $\EA$. 
As a consequence, for each $\varphi(x)\in \Gamma$,
\[\EA\vdash \forall x\, \bigl(\Sat{\Gamma}(\qcr{\varphi},x)
\leftrightarrow \varphi(x)\bigr).\]
$\mathrm{Tr}_{\Gamma}(\varphi)$ denotes the canonical truth predicate based on $\Sat{\Gamma}$, that is the formula $\Sat{\Gamma}(\varphi,0)$ applied to sentences $\varphi$ 
(so that the second argument, in this case fixed to be 0, does not matter).

Thanks to the partial truth predicates, for each $n \ge 1$ and $\Gamma\in\{\Sigma_n,\Pi_n\colon\, n\in\mathbb{N}\}$
the theory $\Gamma\mhyphen\rfn(T)$ can be finitely axiomatized, using a fixed finite axiomatization of $\EA$ and the sentence
\[\forall \varphi\,\bigl(\form_{\Gamma}(\varphi)\wedge \Prov_{T}(\varphi)\rightarrow \mathrm{Tr}_\Gamma(\varphi)\bigr).\]
See \cite[Lemma 2.7]{beklemiszew:survey} for details.

\paragraph{Ehrenfeucht's Lemma.} We recall a classical fact about models of $\PA$ (originally due to \cite{ehrenfeucht:discernible}, a proof can also be found e.g.~in \cite[Theorem 1.7.2]{kossak-schmerl}).

Assume that $\MM \vDash \PA$ and $a,b$ are distinct elements of $\MM$ such that $b$ is definable from $a$ -- in other words, $b$ is unique such that $\MM \vDash \varphi(b,a)$, where $\varphi(x,y)$ is a formula  with no free variables other than $x,y$. Then $\tp^\MM(a) \neq \tp^\MM(b)$, where the notation $\tp^\MM(\cdot)$ refers to the complete type of an element in $\MM$.

\paragraph{Models of fragments of PA.} 
We briefly summarize some well-known constructions of models of ${\IS{n}} + {\exp} + \neg\BS{n+1}$ and ${\BS{n}} + {\exp} + \neg \IS{n}$. A detailed presentation can be found in \cite[Chapter IV.1(d)]{hp:metamathem}. 
In our main arguments in Sections \ref{sec:solidity-below}--\ref{sec:separation}, we will rely heavily on the arithmetization of these constructions. 

A typical method of building a model of ${\IS{n}} + {\neg\BS{n+1}}$ is to use \emph{pointwise definable} structures. Given $\MM\vDash \PAm$, the substructure $\mathcal{K}_{n+1}(\MM)$ consists of those elements of $\MM$ which are definable in $\MM$ by a $\Sigma_{n+1}$ formula. 
If $\MM\vDash \IS{n}$, then $\mathcal{K}_{n+1}(\MM) \preccurlyeq_{n+1} \MM$ (that is, the extension is elementary with respect to $\Sigma_{n+1}$ formulas), so $\mathcal{K}_{n+1}(\MM)$ is a model of $\IS{n}$, 
and it satisfies $\exp$ if $\MM$ does.
Assuming $\mathcal{K}_{n+1}(\MM) \vDash \exp$, we also have $\mathcal{K}_{n+1}(\MM) \vDash {\neg\BS{n+1}}$ unless 
$\mathcal{K}_{n+1}(\MM)$ coincides with the standard model.

A typical method of building a model of ${\BS{n}} + {\neg\IS{n}}$
for $n \ge 1$ is to use a sufficiently elementary proper initial segment of a model of enough induction.
If $\MM \vDash \IS{n-1}$, where $n \ge 1$, and $\mathcal{J} \preccurlyeq_{n-1} \MM$ is a proper initial segment of $\MM$, then $\mathcal{J} \vDash \BS{n}$. 

To get $\mathcal{J} \vDash \neg \IS{n}$, we can for instance ensure that $\mathcal{J}$ is nonstandard but has a $\Sigma_{n}$-definition of the standard cut. 
One way of doing that is to let $\mathcal{J}$ be the closure 
of $[0,a]$, where $a$ is a nonstandard element of $\MM$, under
the witness-bounding function for the universal $\Sigma_{n-1}$ formula -- that is, the function that on input $x$ outputs the smallest $y$ such that all $\Sigma_{n-1}$ sentences (whose codes are) smaller than $x$ are witnessed either below $y$ or not at all. If $n = 1$, we may additionally want to close under $\exp$ to have $\mathcal{J} \vDash \BS{1} + \exp$.

If $n \ge 2$ and $a$ is a nonstandard $\Sigma_{n-1}$-definable element of $\MM$, then the segment $\mathcal{J}$ thus obtained coincides with the structure called $\mathcal{H}_{n-1}(\MM)$ in \cite{hp:metamathem}, but we will reserve the letter $\mathcal{H}$ for structures of a different kind (Henkin models). 

\section{Groundwork}\label{sec:groundwork}

In this section, we discuss three topics in first-order arithmetic which are still of essentially preliminary character but require more extensive treatment.
In some cases, this is because we need to refine standard formulations of presumably rather familiar results, 
in others because the results themselves and the concepts underlying them might not be very widely known.

First, we give a hierarchical version of the well-known 
argument showing that a model of PA is an initial segment of any model of arithmetic that it interprets. Then, we recall the notion of flexible formulas and prove the existence of particular variants of flexible formulas that we will later make use of. Finally, we discuss the topic of axiomatic truth theories with multiple nested truth predicates.

\subsection{The formalized categoricity argument}\label{subsec:formalized-categoricity}

It is well-known that by formalizing the classical argument
used to prove that the (second-order) Dedekind-Peano axiomatization of the standard natural numbers is categorical -- or to prove that the standard numbers form an initial segment of any nonstandard model of arithmetic -- one can show that every model of PA embeds as an initial segment into any model of arithmetic that it can interpret. In \cite{enayat:variations}, this observation is attributed to Feferman \cite{feferman:arithmetization}. 

We need a hierarchical version of that result, in which the ground model might not satisfy full PA, but at the same time we have control over the complexity of the interpretation. This version is proved by mimicking the usual argument.

\begin{definition}
An interpretation $\mathsf{M}$ 
(in an $\LPA$-theory or in a structure for $\LPA$) 
is \emph{$\Sigma_n$-restricted} if the formula $\delta_\mathsf{M}$
and all the formulas $P^\mathsf{M}$, for $P$ a symbol
of the interpreted language, are $\Sigma_n$.

\end{definition}

\begin{lemma}\label{ultimate_dedekind}\label{it:indfact-1}
Let $n \ge 1$.
Suppose that $\MM\vDash \mathrm{I}\Delta_0 + \exp$ and that $I\subseteq_e \MM$ is a cut in $\MM$ such that for every $\Sigma_{n}$ formula $\varphi(x)$ (possibly with parameters from $\MM$) we have:
\begin{equation}\tag{\textdagger} \MM\vDash \varphi(0)\wedge\forall x\,\bigl( \varphi(x)\rightarrow \varphi(x+1)\bigr) \, \Rightarrow \, \text{for each } a \in I \textrm{ it holds that }\MM\vDash \varphi(a).
\end{equation}
Suppose further that $\NN\vDash \PA^-$ is interpreted in $\MM$
via a $\Sigma_n$-restricted interpretation $\mathsf{N}$. 

Then there exists an $\MM$-definable relation $\iota(x,y)$ such that $\iota\cap (I\times\mathcal{N})$ is an embedding of $\mathcal{L}_{\PA}$-structures $I\hookrightarrow \NN$. Moreover, $\iota[I]$ is an initial segment of $\NN$, and the definition of $\iota$ refers only to the parameters used by $\mathsf{N}$.
\end{lemma}

\begin{proof}

Fix $n$, $\MM$, $I$ and $\mathsf{N}$ as above. In particular $\MM^\mathsf{N} = \NN\vDash \PA^-$ and
\[\delta_{\mathsf{N}}, +^{\mathsf{N}}, \times^{\mathsf{N}}, 0^{\mathsf{N}}, 1^{\mathsf{N}}, <^{\mathsf{N}}, =^{\mathsf{N}}\]
are given by $\Sigma_n$ formulas.      
For the purpose of this proof, we introduce the following abbreviation: if $\varphi(x)$ is a $\Sigma_n$ formula, then $s \colon \varphi(x)$ denotes the formula in two free variables $s$ and $x$ resulting from $\varphi(x)$ by deleting the leftmost existential quantifier (or quantifier block) and substituting the variable $s$ for the variable bound by that quantifier. 
Hence if $\varphi(x)$ is $\exists y\, \psi(x,y)$, then $s \colon\varphi(x)$ is $\psi(x,s)$.

We define $\iota(x,y)$ to hold if:
\begin{multline*}
    \exists s,t \,\big[\mathrm{seq}(s,t)
    \wedge \mathrm{len}(s)=\mathrm{len}(t)=x+1
    \\ \wedge s_0 =^{\mathsf{N}} 0^{\mathsf{N}}   
    \wedge \as{i}{x}(t_i \colon (s_{i+1} =^{\mathsf{N}} s_i+^{\mathsf{N}} 1_{\mathsf{N}})) 
    \wedge s_x =^{\mathsf{N}} y\big].
\end{multline*}
Since the formula 
$t_i \colon (s_{i+1}=^{\mathsf{N}}s_i+_{\mathsf{N}}1_{\mathsf{N}})$ 
is $\Pi_{n-1}$, 
the relation $\iota$ is $\Sigma_n$-definable.
We claim that for every $a,a'\in I$ the following holds in $\MM$:
     \begin{align}
        & \exists y\,\iota(a,y),\label{claim:embedding-1}\\
         & \forall y \, \forall y'\,\big(\iota(a,y) \wedge \iota(a,y')\rightarrow y=^{\mathsf{N}} y'\big),\label{claim:embedding-2}\\
        &  \forall y \, \forall y'\,\big(\iota(a,y) \wedge \iota(a',y') \wedge y=^{\mathsf{N}}y'\rightarrow a=a'\big).\label{claim:embedding-3}
     \end{align}

The proof of \eqref{claim:embedding-1} is a straightforward application of $(\dagger)$, 
since $\exists y\, \iota(x,y)$ is a $\Sigma_n$ formula, and the subset of $\MM$ it defines is closed under successor. The latter follows from the fact that the successor operation is provably total in $\PA^-$ and that we can always extend a given sequence by one element.
    
To prove \eqref{claim:embedding-2}, consider $a \in I$ 
and $s, s'$ such that $s\colon\iota(a,y)$ and $s'\colon\iota(a,y')$, and apply $(\dagger)$ to the formula $(s_x =^{\mathsf{N}} s'_x) \vee x > a$ to prove $s_i =^{\mathsf{N}} s'_i$ for all $i \le a$. 
    
Finally, \eqref{claim:embedding-3} can be proved by
using $(\dagger)$ to simulate the $\Sigma_n$ least number principle up to elements of $I$:  if $a \in I$ is the smallest number for which there is $a'$ witnessing that $\iota$ is not injective with respect to $=^{\mathsf{N}}$, we reach an easy contradiction by considering $a-1$ and $a'-1$.
     
This completes the proof that $\iota$ is an embedding from $I$ into $\MM$. Clearly, the definition of $\iota$ does not make use of any parameters beyond the ones used in $\mathsf{N}$. It remains to check that the range of $\iota$ is an initial segment of $\NN$. To this end, consider $a \in I$ and $y,s,t,z$ such that that $s,t \colon \iota(a,y)$
and $z <^{\mathsf{N}} y$. Apply $(\dagger)$
to find $i < a$ such that $s_i \le^{\mathsf{N}} z$ and
$s_{i+1} >^{\mathsf{N}} z$. Then it is easy to 
check that in fact $s_i =^{\mathsf{N}} z$.
\end{proof}

\begin{remark}\label{ultimate_dedekind_pam}
Lemma \ref{ultimate_dedekind} is stated for models of $\EA$, because the proof of the lemma officially relies on
sequence coding by means of the Ackermann interpretation. However, using the sequentiality of $\PA^-$
and essentially the same proof as above but with appropriately modified sequence coding, one can obtain the following variant of the lemma:

\medskip

\noindent{Suppose that $\MM\vDash \PA^-$ and that there is a shortest $\MM$-definable cut $I\subseteq_e \MM$. Suppose further that $\NN\vDash \PA^-$ is interpreted in $\MM$ via an interpretation $\mathsf{N}$. {Then} there exists an $\MM$-definable relation $\iota(x,y)$ such that $\iota\cap (I\times\mathcal{N})$ is an embedding $I\hookrightarrow \NN$. Moreover, the $\iota[I]$ is an initial segment of $\NN$, and the definition of $\iota$ refers only to the parameters used by $\mathsf{N}$.}
\end{remark}

\begin{corollary}\label{lem_n_minimality}
Let $n \ge 1$.
Suppose that $T \supseteq {\IS{n}}$ and that $\mathsf{N}$ is a $\Sigma_n$-restricted interpretation of $\PA^-$ in $T$. Then, provably in $T$, there is an embedding of 
$\mathsf{id}_T$ into $\mathsf{N}$ whose range is an $\le^\mathsf{N}$-initial segment of $\delta_\mathsf{N}$.
\end{corollary}
\begin{proof}
Fix any model $\MM\vDash T$ and apply the proof of Lemma \ref{ultimate_dedekind} to $I = M$.
\end{proof}

The corollary below is a syntactical incarnation of the well-known fact that the truth of $\Pi_1$ formulas is preserved in initial substructures.

\begin{corollary}\label{cor_downward_prec_pi1}
Let $n \ge 1$.
Suppose that $T \supseteq \IS{n}$ and that $\mathsf{N}$ is a $\Sigma_n$-restricted interpretation of $\PA^-$ in $T$. Then for every $\Pi_1$ sentence $\varphi$
\[T\vdash \varphi^{\mathsf{N}}\rightarrow \varphi.\]
\end{corollary}

\subsection{Flexible formulas}

Flexible formulas, or formulas whose truth values that are ``as undetermined as possible'', were introduced in their basic form by Mostowski \cite{mostowski:generalization} and 
Kripke \cite{kripke:flexible}.
The theory of flexible formulas was then developed by a number of authors, in recent years \emph{inter alios} by Woodin, Hamkins, Blanck and Enayat.
Flexible formulas of various kinds play an important technical role in our arguments. The version that we need is almost the same as in \cite[Corollary 4.28]{rsl::blanck},
with the exception that we find it convenient to isolate one more parameter. We describe the construction in some detail in order to keep track where it can be formalized.

\begin{definition}
Let $n,k\in \omega$ and $T\subseteq \LPA$ be a theory. 
We say that $\xi$ is \emph{$(\Pi_n,\Sigma_k)$-flexible over $T$} if for every formula $\varphi(x)$ in $\Sigma_k$, \[T+\forall x \, (\xi(x)\leftrightarrow \phi(x)) + \Pi_n\mhyphen{\Th{\mathbb{N}}}\] is consistent. We say that $\xi$ is \emph{$\Sigma_k$-flexible over $T$} if it is  
$(\Pi_0, \Sigma_k)$-flexible over $T$.
\end{definition}

The notion of ``$\Sigma_k$-flexible'' defined above
corresponds to the standard concept of flexibility as used e.g.~in \cite{rsl::blanck}). Our proof of the existence of $(\Pi_n,\Sigma_k)$-flexible formulas is based on \cite[Chapter 2.3, Theorem 11]{lindstrombook}.

\begin{theorem}\label{tw_flex_main}
Let $T$ be a computable $\LPA$-theory extending $\mathrm{I}\Delta_0+\exp$. For all $n \ge 0$ and $k \ge 1$, there is a $\Sigma_{\max(n+1,k)}$ formula $\xi$ such that $\IS{n}+\exp$ proves the implication:
\begin{equation*}
\Pi_n\mhyphen\mathrm{Con}(T)\quad\rightarrow\quad\forall\varphi\!\in\!\mathrm{Form}_{\Sigma_k} \left(\Pi_n\mhyphen\mathrm{Con}(T+\ulcorner\forall v\,(\xi(v)\leftrightarrow\varphi(v))\urcorner)\right)\text{.}
\end{equation*}
Moreover, there is a computable function that outputs the formula $\xi$
given $n$, $k$, and an index of $T$.
\end{theorem}
\begin{proof}
Let $\Gamma\left(\theta,\langle\varphi,\sigma,p\rangle\right)$ be
\begin{equation*}
\form_{\Sigma_k}(\varphi)\ \land\
\mathrm{Tr}_{\Pi_n}(\sigma)\ \land\ \mathrm{Proof}_{T\cup\{\sigma\}}\big(\,p,\; \ulcorner\lnot\forall v\,(\theta(v)\leftrightarrow\varphi(v))\urcorner\,\big)\text{.}
\end{equation*}
This formula is $\Delta_0(\exp)$ if $n=0$, and $\Pi_n$ if $n>0$. Hence, $\lnot\Gamma$ is $\Delta_0(\exp)$ if $n=0$, and otherwise it is logically equivalent to a formula $\exists z\,\delta(\theta,\langle\varphi,\sigma,p\rangle,z)$, where $\delta$ is $\Pi_{n-1}$.

\noindent If $n=0$, we let
\begin{align*}
&\Gamma'\left(\theta,\langle\varphi,\sigma,p\rangle\right):=\quad\Gamma\left(\theta,\langle\varphi,\sigma,p\rangle\right)\ \land\ \forall\langle\tau,\psi,q\rangle\!<\!\langle\varphi,\sigma,p\rangle\ \, \lnot\Gamma(\theta,\langle\psi,\tau,q\rangle)\text{.}\\
\intertext{For $n>0$, we set}
&\Gamma'\left(\theta,\langle\varphi,\sigma,p\rangle\right):=\quad\Gamma\left(\theta,\langle\varphi,\sigma,p\rangle\right)\ \land\ \exists w\,\forall\langle\tau,\psi,q\rangle\!<\!\langle\varphi,\sigma,p\rangle\ \,\delta\!\left(\theta,\langle\psi,\tau,q\rangle,(w)_{\langle\psi,\tau,q\rangle}\right)\text{.}
\end{align*}
In each case, $\Gamma'$ is $\Sigma_{n+1}$. Let $\xi$ be a $\Sigma_{\max(n+1,k)}$ formula such that $\mathrm{I}\Delta_0+\exp$ proves the universal closure of the biconditional
\begin{equation*}
\xi(x)\ \leftrightarrow\ \exists\langle\varphi,\sigma,p\rangle\ \big(\,\Gamma'\left(\ulcorner\xi\urcorner,\langle\varphi,\sigma,p\rangle\right)\ \land\ \Sat{\Sigma_k}(\varphi,x)\,\big)\text{.}
\end{equation*}
Such a formula exists by the standard proof of the diagonal lemma. Note that the function returning a code of $\xi$ given $n$, $k$, and an index of $T$ can be chosen to be computable.

From now on, we work in $\mathrm{I}\Sigma_n+\exp$. We show the following implication:
\begin{equation*}
\Big(\exists\varphi\!\in\!\mathrm{Form}_{\Sigma_k}\, \lnot\Pi_n\mhyphen\mathrm{Con}(T+\ulcorner\forall v\,(\xi(v)\leftrightarrow\varphi(v))\urcorner)\Big)\quad\rightarrow\quad\lnot\Pi_n\mhyphen\mathrm{Con}(T)\text{.}
\end{equation*}
Assuming the antecedent holds, let $\langle\varphi,\sigma,p\rangle$ be the least triple satisfying
\begin{equation*}
\form_{\Sigma_k}(\varphi)\ \land\
\mathrm{Tr}_{\Pi_n}(\sigma)\ \land\ \mathrm{Proof}_{T\cup\{\sigma\}}\big(\,p,\; \ulcorner\lnot\forall v\,(\xi(v)\leftrightarrow\varphi(v))\urcorner\,\big)\text{.}
\end{equation*}
The existence of such a triple is guaranteed by $\Sigma_n$-induction. By the (formalized) $\Sigma_{n+1}$-completeness of the set of all true $\Pi_n$ sentences, it follows that $\mathrm{I}\Delta_0+\exp$ augmented with all true $\Pi_n$ sentences proves that $\langle\varphi,\sigma,p\rangle$ is the unique triple satisfying $\Gamma'(\ulcorner\xi\urcorner,\langle\varphi,\sigma,p\rangle)$. Consequently, by the definition of $\xi$ (and the $\Sigma_1$-completeness of $\mathrm{I}\Sigma_n+\exp$), we have that this augmented theory proves
\begin{equation*}
\forall x\,(\xi(x)\,\leftrightarrow\,\mathrm{Sat}_{\Sigma_k}(\varphi,x))\text{.}
\end{equation*}
Since, by assumption, $T$ extended by all true $\Pi_n$ sentences proves $\lnot\forall v\,(\xi(v)\leftrightarrow\varphi(v))$, and $T$ contains $\mathrm{I}\Delta_0+\exp$, we obtain $\lnot\Pi_n\mhyphen\mathrm{Con}(T)$.
\end{proof}

In the case $n = 0$, Theorem \ref{tw_flex_main} gives us
a provably $\Sigma_1$-flexible formula over $T$ which is itself a $\Sigma_1$ formula. In the corollary below, we note a well-known property of such formulas.

\begin{corollary}\label{cor:flexible-con}
Suppose $T$ is an r.e. theory which extends $\EA$ and assume that $\xi(x)$ is a $\Sigma_1$ formula which is $(\EA)$-provably $\Sigma_1$-flexible over $T$ in the sense that $\EA$ proves
\[\Con{T}  \quad\rightarrow\quad\forall\varphi\!\in\!\mathrm{Form}_{\Sigma_1} \left(\mathrm{Con}(T+\ulcorner\forall v\,(\xi(v)\leftrightarrow\varphi(v))\urcorner)\right)\text{.} \] 
Then $\EA\vdash \Con{T}\rightarrow \forall x \, \neg\xi(x).$
\end{corollary}
\begin{proof}
Fix $T$ and $\xi(x)$ as in the assumptions. Working in $\EA$, assume that $\exists x \, \xi(x)$. Then, by provable $\Sigma_1$-completeness, $\Prov_{T}(\ulcorner\exists x \,  \xi(x)\urcorner)$. However, that implies 
\[\neg \Con{T + \forall x \, (\xi(x)\leftrightarrow x\neq x)},\]
and $x\neq x$ is a $\Sigma_1$ formula. 
By 
Theorem \ref{tw_flex_main}, we obtain $\neg\Con{T}$.
\end{proof}

\subsection{Theories of iterated Tarskian truth}\label{sect::truth_theories}

In the proofs of our main results, we will need to have
access to a pair of theories which are themselves solid but additionally do not interpret models of each other in a ``nice'' way: specifically, no model of one of the theories should be a retract of a model of the other.
It turns out that one way of securing such a property
is to use Tarski's undefinability of truth theorem.
Consequently, one example of not just a pair, but a whole
family of such theories is supplied by the following canonical theories of truth over arithmetic with varying numbers of hierarchically nested truth predicates.  

\begin{definition}\label{defn:ct-n}
    For $n\in\omega$, the theory $\CT^n[\PA]$ is
    formulated in the language $\mathcal{L}_n$ which
    extends $\mathcal{L}_{\PA}$ with fresh predicates $P_1,\ldots, P_n$ (we assume that $\mathcal{L}_0=\LPA$). The theories are defined inductively: $\CT^0[\PA] = \PA$,
    and $\CT^{n+1}[\PA]$ extends $\CT^n[\PA]$
    by the induction scheme for all $\mathcal{L}_{n+1}$-formulas and the following axioms:
    \begin{enumerate}[(i)]
        \item $\forall t\!\in\!\mathrm{Term}\, \big(P_{n+1}(\subst(\qcr{P_{i}(x)},t)) \leftrightarrow P_i(\mathrm{val}(t))\big)$, for each $i = 1, \ldots, n$. 
        \item $\forall s,t \!\in\!\mathrm{Term}\, \big(P_{n+1}(\qcr{s=t}) \leftrightarrow \mathrm{val}(s)=\mathrm{val}(t)\big).$
        \item $\forall \varphi \!\in \! \mathrm{Sent}_{\mathcal{L}_n}\, \big(P_{n+1}(\qcr{\neg\varphi})\leftrightarrow \neg P_{n+1}(\varphi)\big).$
        \item $\forall \varphi,\psi\!\in \! \mathrm{Sent}_{\mathcal{L}_n}\, \big(P_{n+1}(\qcr{\varphi\wedge \psi})\leftrightarrow (P_{n+1}(\varphi)\wedge P_{n+1}(\psi))\bigr).$
        \item $\forall \varphi \!\in \! \mathrm{Form}^{\leq 1}_{\mathcal{L}_n} \,
        \forall v \!\in \! \mathrm{Var} \,
        \big(P_{n+1}(\qcr{\forall v\, \varphi})\leftrightarrow \forall y\, P_{n+1}(\subst(\varphi,\name(y)))\bigr).$
        \item $\forall x \, \bigl(P_{n+1}(x)\rightarrow \mathrm{Sent}_{\mathcal{L}_n}(x)\bigr)$. 
    \end{enumerate}
\end{definition}
\noindent One usually writes $\CT[\PA]$ instead of $\CT^1[\PA]$, and we will occasionally write $P$ instead of $P_1$. The theories $\CT^n[\PA]$ are sometimes also called $\mathrm{RT}^{<n+1}$ (for example in \cite{halbach_book}).

\begin{remark}
By induction on formula complexity inside $\CT^n[\PA]$, we can show that for all $1\leq i\leq j \leq n$, $P_j$ agrees with $P_i$ on $\mathcal{L}_{i-1}$, provably in $\CT^n[\PA]$. More precisely, for every $1\leq i \leq j\leq n$ the following is provable in $\CT^n[\PA]$:
\[\forall \varphi\,\big(\mathrm{Sent}_{\mathcal{L}_{i-1}}(\varphi)\rightarrow (P_{j}(\varphi)\leftrightarrow P_i (\varphi)\big).\]
\end{remark}

The lemma below can be seen as a generalization
of the result 
that all the theories $\CT^n[\PA]$ are solid \cite{el:categoricity-like}.
The proof of the lemma 
combines a few simple but important
observations concerning models of the full induction scheme and theories of iterated Tarskian truth. 
The statement and proof of the lemma are illustrated in Figure \ref{fig:indfact}.

\begin{lemma}\label{lem:indfact}
Fix $m\in\omega$. Let $\MM_1$, $\MM_2$, $\MM_3$ be models in a language extending $\mathcal{L}_{\PA}$, possibly strictly, such that 
\begin{enumerate}[(i)]
    \item $\MM_i\vDash\PAm$, for $i=1,2,3$, 
    \item $\mathsf{M}_2:\MM_1\rhd\MM_{2}$ and $\mathsf{M}_3: \MM_2\rhd \MM_3$.
\end{enumerate} 
Suppose further that for each $i=1,2,3$ we are given an interpretation $\mathsf{N}_i$ of $\CT^m[\PA]$ in $\MM_i$ such that
\begin{enumerate}[(a)]
    \item the domain of $\mathsf{N}_i$ is the shortest definable cut in $\MM_i$, and 
    \item  there exists an $\MM_1$-definable isomorphism from $\MM_1^{\mathsf{N}_1}$ onto $\MM_3^{\mathsf{N}_3}$.
\end{enumerate}
Then, there are $\MM_i$-definable isomorphisms $h_i:\MM_i^{\mathsf{N}_{i}}\rightarrow \MM_{i+1}^{\mathsf{N}_{i+1}}$.
\end{lemma}
\begin{remark}
It will be clear from the proof of Lemma \ref{lem:indfact} that assumption (b) could be weakened, as in fact we only
need an $\MM_1$-definable isomorphism between the $\LPA$-reducts of $\MM_1^{\mathsf{N}_1}$ and $\MM_3^{\mathsf{N}_3}$.
\end{remark}
\begin{figure}[htbp]
\centering
\begingroup
\input{three-models.tex}
\caption{Lemma \ref{lem:indfact} and its proof. The map $j$ is the isomorphism between (the $\LPA$-reducts of) $\mathsf{N}_1$ and $\mathsf{N}_3$ provided by the assumption.
For $i =1, 2$, the map $h_i$ is the canonical embedding of
$\mathsf{N}_i{\upharpoonright}_{\LPA}$
into $\mathsf{N}_{i+1}{\upharpoonright}_{\LPA}$.}
\label{fig:indfact}
\endgroup
\end{figure}
\begin{proof}
 Since we are going to apply $\mathsf{N}_i$ and $\mathsf{M}_{i+1}$ only in $\MM_i$, we shall abbreviate $\MM_i^{\mathsf{N}_i}, \MM_i^{\mathsf{M}_{i+1}}$ as $\mathsf{N}_i$, $\mathsf{M}_{i+1}$, respectively.  
 Observe that since the domain of $\mathsf{N}_i$ is the shortest cut in $\MM_i$, we know that $\mathsf{N}_i$ satisfies induction with respect to all $\MM_i$-definable properties. We shall refer to this feature as
 the $\MM_i$-inductiveness of $\mathsf{N}_i$. 
 Thanks to $\MM_i$-inductiveness, we can repeat
 the argument from Section \ref{subsec:formalized-categoricity}
 so as to conclude that for each $i\leq 2$ there is an $\MM_i$-definable embedding $h_i$ of the reduct $\MM_i^{\mathsf{N}_i}{\upharpoonright}_{\mathcal{L}_{\PA}}$ onto an initial segment of the reduct $\MM_{i+1}^{\mathsf{N}_{i+1}}{\upharpoonright}_{\mathcal{L}_{\PA}}$. 
 
 The reasoning thus far was independent of $m$.
 To prove the lemma,  we use induction on $m$. 
 Assume first that $m = 0$. Let $j$ be the
 $\MM_1$-definable isomorphism from $\mathsf{N}_1$ onto
 $\mathsf{N}_{3}$ (or, more precisely from the point of view of $\MM_1$, onto
 $\mathsf{M}_2\mathsf{M}_3\mathsf{N}_{3}$). 
 Since $m = 0$, in order to prove that 
 $\mathsf{N}_i$ and $\mathsf{N}_{i+1}$ 
 are $\MM_i$-definably isomorphic we only need to show that
 both the embeddings $h_i$ are onto $\mathsf{N}_{i+1}$
 for their respective $i$.
 Suppose otherwise: then $(h_2\circ h_1)(\mathsf{N}_1)$ is an $\MM_1$-definable proper initial segment of $\mathsf{M}_2\mathsf{M}_3\mathsf{N}_3$, and thus $(j^{-1}\circ h_2\circ h_1)(\mathsf{N}_1)$ is an $\MM_1$-definable proper initial segment of $\mathsf{N}_1$, contradicting the $\MM_1$-inductiveness of $\mathsf{N}_1$.

 Now fix $m>0$ and assume that the lemma holds
 for $m\!-\!1$.
 The inductive assumption tells us that for $i = 1,2$, the map $h_i$ is an isomorphism between $\mathsf{N}_i{\upharpoonright_{\mathcal{L}_{m-1}}}$ and 
 $\mathsf{N}_{i+1}{\upharpoonright_{\mathcal{L}_{m-1}}}$.
 Let $P^i_m$ be the $m$-th truth predicate of $\mathsf{N}_i$. We argue that $h_i[P^i_m] = P^{i+1}_m$, which will complete the proof. Consider $h_i^{-1}[P^{i+1}_m]$. This is an $\mathcal{M}_{i}$-definable subset of $\mathsf{N}_{i}$. Since $h_i$ is an isomorphism between $\mathsf{N}_i{\upharpoonright_{\mathcal{L}_{m-1}}}$ and 
 $\mathsf{N}_{i+1}{\upharpoonright_{\mathcal{L}_{m-1}}}$,
 we see that $h_{i}^{-1}[P^{i+1}_m]$ actually satisfies the axioms of $\CT^m[\PA]$,
 cf.~Definition \ref{defn:ct-n}. In $\mathsf{N}_i$,
 define $I$ be the set of logical depths of those sentences for which $P^i_m$ coincides with $h_i^{-1}[P^{i+1}_m]$. In other words, $I$ consists of those $x \in \mathsf{N}_i$ such that
 $\mathcal{M}_i$ satisfies
 \[
 \forall \varphi\!\in\!\mathsf{N}_i
 \left(
 ({\varphi \in \form_{\mathcal{L}_{m-1}}} \land {\mathsf{dpt}(\varphi)\leq x})^{\mathsf{N}_i}
 \rightarrow 
 \left(P^i_m(\varphi)\leftrightarrow \varphi\in h_i^{-1}[P^{i+1}_m]\right)\right).
 \]    
Since both $P^i_m$ and $h_{i}^{-1}[P^{i+1}_m]$ satisfy conditions (i)--(ii) of Definition \ref{defn:ct-n},
we know that $0 \in I$, and since both satisfy the inductive conditions (iii)--(v), we also know that $I$ is closed under successor. Thus, $I$ is a $\mathcal{M}_i$-definable cut contained in $\mathsf{N}_i$, and since the latter is the shortest cut in $\mathcal{M}_i$, we conclude that $P^i_m$ actually coincides with $h_i^{-1}[P^{i+1}_m]$.
\end{proof}

As a corollary, we obtain a theorem originally proved by Ali Enayat in \cite[Theorem 88 a)]{el:categoricity-like}.

\begin{corollary}\label{cor:solid_ctn}
For every $m$, the theory $\CT^m[\PA]$ is solid.
\end{corollary}
\begin{proof}
This is a special case of Lemma \ref{lem:indfact} 
in which 
$\mathsf{N}_i$ is the identity interpretation on $\MM_i$.
\end{proof}

\section{Solidity below PA: the basic construction}\label{sec:solidity-below}\label{sec:proper-solid}

This section contains the proof of a no-frills version of our main result, which simply says that there are arbitrarily strong solid proper subtheories of $\PA$. The techniques we use are generalized and refined in the next section, yielding two natural strengthenings of our main theorem and providing background for constructions from Section \ref{sec:separation}.

\begin{theorem}\label{thm:proper-solid}
    For every $n \in \N$, there exists an r.e. solid subtheory $T_n$
    of $\PA$ that contains $\IS{n}+\exp$ but not $\BS{n+1}$.
\end{theorem}

To prove Theorem \ref{thm:proper-solid}, we will gradually introduce the main concepts involved in our argument 
and derive a series of lemmas about those concepts.

\paragraph{Construction of $T_n$.} Fix $n \in \N$. We want to construct a theory $T_n \supseteq \IS{n}+\exp$ that is a proper subtheory of $\PA$ but is nevertheless solid. To that end, we will define an auxiliary theory $\tn{n} \supseteq \IS{n}+\exp$ that is inconsistent with $\BS{n+1}$ but rather strong from the perspective of interpretability.
Furthermore, we will define an interpretation $\hn{n}$ of a model of ${\tn{n}}$ in the standard model $(\N,\Th{\N})$
of ${\CT[\PA]}$, and an interpretation $\gn{n}$ in the other direction. 
Our eventual theory $T_n$ will essentially say that we are either in a universe satisfying $\PA$ or in one satisfying
$\tn{n}$.

\begin{remark}The interpretations $\hn{n}$ and $\gn{n}$ will in fact witness the bi-interpretability of $\tn{n}$ and $\CT[\PA]$, which we will show as a separate proposition in Section \ref{subsec:formalizing-biint}.
\end{remark}

We now proceed to define our concepts more precisely,
beginning with $\hn{n}$. Let $\xi_n(x)$
be a $\Sigma_1$ formula 
that is provably $\Sigma_1$-flexible over $\IS{n+1}$ in the sense of Corollary \ref{cor:flexible-con}.
The interpretation $\hn{n}$ describes the following process, as carried out in the standard model $(\N,\Th{\N})$ of $\CT[\PA]$:

\begin{itemize}
    \item Consider a canonically defined binary tree whose
    paths correspond to complete consistent henkinized extensions of the theory
        \begin{equation}\label{eq:theory-1}
            {\IS{n+1}}  
            \cup \{\xi_{n}(\num{k})\colon\, k \in \Th{\N}\}
            \cup \{\neg \xi_{n}(\num{k})\colon\, k \in \N \setminus \Th{\N}\}.
        \end{equation}
    (One could work with $\IS{n} + \exp$ instead of $\IS{n+1}$, but the latter allows us to avoid having to write the annoying $\exp$ more than needed.)    
    \item Take the Henkin model, say $\mathcal{H}$, given by
    the leftmost path through that tree.
    \item Take $\KK_{n+1}(\mathcal{H})$, that is the submodel of $\mathcal{H}$ consisting of the $\Sigma_{n+1}$-definable elements.
\end{itemize}

Since $\xi_n(x)$ is a $\Sigma_1$-flexible formula over ${\IS{n+1}}$, the compactness theorem implies that the theory in (\ref{eq:theory-1}) is consistent. 
Thus, when applied in $(\N,\Th{\N})$, the interpretation
$\hn{n}$ indeed produces a structure $\KK : = \KK_{n+1}(\mathcal{H})$. The construction of $\KK$, that is, of $(\N,\Th{\N})^{\hn{n}}$, is schematically presented in Figure \ref{fig:structure-k}.

\begin{figure}[htbp]

\begin{tikzpicture}
\draw[very thick, ] (3,0) -- (4.65,0);

\draw[very thick, dashed] (4.65,0) -- (5.6,0);
\draw[very thick, ] (5.6,0) -- (7,0);
\draw[very thick, dashed] (7,0) -- (8.5,0);
\draw[very thick, ] (8.5,0) -- (10.8,0);

\draw[very thick, dashed] (10.8,0) -- (12.95,0);

\draw[very thick, ] (3, -0.1) -- (3, 0.1);
\draw[very thick, ] (3.3, -0.1) -- (3.3, 0.1);
\draw[very thick, ] (3.6, -0.1) -- (3.6, 0.1);

\draw (3, -0.4) node {\small{$\mathsf 0$}};
\draw (3.3, -0.4) node {\small{$\mathsf 1$}};
\draw (3.6, -0.4) node {\small{$\mathsf 2$}};
\draw (4.1, -0.4) node {\small{$\ldots$}};

\draw[very thick, ] (4.55, -0.25) .. controls (4.7,-0.15)  and (4.7,0.15) ..  (4.55, 0.25);

\draw (4.8, 0.5) node {$\mathbb{N}$};

\draw[very thick, ] (6.3, -0.1) -- (6.3, 0.1);

\draw[very thick, ] (10.65, -0.25) .. controls (10.8,-0.15)  and (10.8,0.15) ..  (10.65, 0.25);

\draw (11, 0.5) node {$\KK_{n+1}(\mathcal{H})$};

\draw[very thick] (12.85, -0.25) .. controls (13,-0.15)  and (13,0.15) ..  (12.85, 0.25);

\draw (13.2, 0.5) node {$\mathcal{H}$};

\end{tikzpicture}

\caption{Construction of the model $(\N,\Th{\N})^{\hn{n}}$ of $\tn{n}$. The solid horizontal lines represent $(\N,\Th{\N})^{\hn{n}}$, 
which is the pointwise $\Sigma_{n+1}$-definable substructure of the Henkin structure $\mathcal{H}$. The dashed horizontal lines represent the rest of $\mathcal{H}$.}

\label{fig:structure-k}

\end{figure}

\begin{lemma}\label{lem:facts-about-k}
$\KK$ is a $\Sigma_{n+1}$-elementary substructure of
$\mathcal{H}$ satisfying ${\IS{n}} + {\exp} + {\neg\BS{n+1}}$.
\end{lemma}

\begin{proof}
It follows directly from the well-known
properties of pointwise $\Sigma_{n+1}$-definable structures
discussed at the end of Section \ref{sec:prelim} that
$\KK \preccurlyeq_{n+1} \mathcal{H}$ and, as a consequence, that $\KK \vDash {\IS{n}} +\exp$. 
We can also conclude that 
$\KK \vDash {\neg\BS{n+1}}$ unless $\KK$ is the standard
model. 

However, since $\mathcal{H} \vDash \exists x\,\xi_n(x)$, by $\Sigma_{n+1}$-elementarity we get 
$\KK \vDash \exists x\,\xi_n(x)$, which ensures nonstandardness by Corollary \ref{cor:flexible-con}.
\end{proof}

The standard cut $\N$ is definable in $\KK$ as the set of those $x$ that satisfy the formula $\delta_n(x)$: $``$there exists an element without a $\Sigma_{n+1}$ definition smaller than $x"$. Clearly then, $\N$ is the smallest definable cut of $\KK$. Moreover, since $\xi_n(x)$ is a $\Sigma_1$ formula and $\KK \preccurlyeq_{n+1} \mathcal{H}$, for each standard $k$ it holds that $\KK \vDash \xi_n(k)$ if and only if $k$ is (the code of) an arithmetical sentence true in $\N$.

Thus, $(\N, \Th{\N})$ can be interpreted in $\KK$ by the interpretation $\gn{n}$ in which the domain is defined by $\delta_n$, the arithmetical operations are unchanged, and $P$ is defined by $\delta_n(x) \land \xi_n(x)$. 

\begin{lemma}\label{lem:in-and-jn}
There is a $\KK$-definable isomorphism $i_n$ between  $\gn{n}\!\hn{n}$ and the identity interpretation of $\KK$ in itself, and there is an $(\N, \Th{\N})$-definable isomorphism $j_n$ between $\hn{n}\!\gn{n}$ and the identity interpretation of $(\N,\Th{\N})$ in itself.
\end{lemma}

\begin{proof}
The isomorphism $i_n$ takes an element $y$ of $\KK$, finds the least $\Sigma_{n+1}$ definition $x$ of $y$ (where $x$ is necessarily a standard number,
because $\KK$ is in fact pointwise $\Sigma_{n+1}$-definable), and maps $y$ to the element defined by $x$ in the structure obtained according to $\hn{n}$. 

The isomorphism $j_n$ is a special case of the map appearing in Lemma \ref{ultimate_dedekind}: 
it takes $x \in \N$ to the $x$-th smallest element of $\KK$. By the construction of $\KK$ and the definition of $\gn{n}$, the range of $j_n$ is exactly $\KK^{\gn{n}} = (\N, \Th{\N})^{\hn{n}\!\gn{n}}$
and the isomorphism of arithmetical structures extends to the truth predicate $P$.
\end{proof}
 
Note that $\gn{n}, \hn{n}, i_n, j_n$ are all definable without parameters in the respective structures. 
(Which is in any case obvious since both $(\N,\Th{\N})$ and $\KK$ are pointwise definable.)
 
We let $\tn{n}$ be a theory axiomatizing some salient properties of $\KK$. Namely, the axioms of $\tn{n}$ are:
\begin{enumerate}[(i)]
    \item\label{it:t0n-1} ${\IS{n}} + {\exp} + {\neg\BS{n+1}} $,
    \item\label{it:t0n-2} $``\delta_n$ defines a cut which is the shortest definable cut$"$,
    \item\label{it:t0n-3} $\gn{n}\vDash\CT[\PA]$, 
    \item\label{it:t0n-4} $``i_n \colon \id \to \gn{n}\!\hn{n}$ is an isomorphism$"$,
    \item\label{it:t0n-5} $\gn{n}\vDash ``j_n \colon \mathsf{id} \to \hn{n}\!\gn{n}$ is an isomorphism$"$.
\end{enumerate}
We note that \ref{it:t0n-2} and \ref{it:t0n-3} are infinite collections of sentences. 

Finally, we let $T_n$ be the following theory: 
\[{\mathrm{I}\Delta_0 + \exp}  \cup \{\BS{n+1}\rightarrow \IS{k}\colon\,k\in \mathbb{N}\} \cup \{\neg\BS{n+1}\rightarrow \varphi\colon\, \varphi\in \tn{n}\}.\]
In other words, $T_n$ is defined by cases: if $\BS{n+1}$ holds, then $\PA$ holds, and if $\BS{n+1}$ fails, then $\tn{n}$ holds. 

\paragraph{Verification of properties of $T_n$.}

\begin{lemma}\label{lem:tn-proper-subtheory}
 $T_n$ contains $\IS{n} + \exp$ but not $\BS{n+1}$. Thus, it is a proper subtheory of $\PA$.  
\end{lemma}
\begin{proof}
By the construction of the model $\KK$ described above,
and the facts summarized in Lemmas \ref{lem:facts-about-k}
and \ref{lem:in-and-jn},
$\tn{n}$ is a consistent theory. 

By definition, $\tn{n}$ contains ${\IS{n}} + {\exp} + {\neg \BS{n+1}}$, and each model of $T_n$ is either a model of 
$\PA$ or one of $\tn{n}$.
\end{proof}

To prove Theorem \ref{thm:proper-solid}, we need to show
the solidity of $T_n$. This requires analyzing a number
of cases dependent on the theories satisfied by models forming a potential counterexample to solidity. 
We prove two more lemmas, the first of which rules out 
a counterexample consisting of models of $\tn{n}$, while
the other will be helpful in ruling out counterexamples
in which models of $\tn{n}$ and of $\PA$ alternate.

\begin{lemma}\label{lem:t0n-solid}
 $\tn{n}$ is solid.  
\end{lemma}
\begin{proof}
Let $\MM_1 \rhd \MM_2 \rhd \MM_3$ be models of $\tn{n}$
such that there is an $\MM_1$-definable isomorphism from $\MM_1$ onto $\MM_3$.
For each $i \in \{1,2,3\}$, let $\NN_i$ be the model
of $\CT[\PA]$ obtained by applying the interpretation $\gn{n}$
in $\MM_i$, and let $\MM'_i$ be the model of $\tn{n}$
obtained by applying $\hn{n}$ in $\NN_i$. Note that
the domain of each $\NN_i$ is the smallest definable
cut of $\MM_i$, by axioms \ref{it:t0n-2} of $\tn{n}$,  and that there is an $\MM_1$-definable isomorphism from $\NN_1$ onto $\NN_3$. See Figure \ref{fig:t0n-solid}.

\begin{figure}[htbp]
\begin{center}
\begin{tikzcd}[column sep = 4.5em, row sep = 4.5 em ]
 \MM_1 
 \arrow[r, symbol = \rhd]
 \arrow[rr, bend left = 30]
 \arrow[d, symbol = \rhd] 
 \arrow[dd, bend right = 40, "i_n"] & 
 \MM_2 
 \arrow[r, symbol = \rhd]
 \arrow[d, symbol = \rhd] 
 \arrow[dd, bend right = 40, "i_n"] & 
 \MM_3 
 \arrow[d, symbol = \rhd]
 \\
 \NN_1
 \arrow[r, bend right = 40, dashed]
 \arrow[rr, bend left = 20, dashed]
 \arrow[d, symbol = \rhd] & 
 \NN_2
 \arrow[d, symbol = \rhd] & 
 \NN_3
 \arrow[d, symbol = \rhd] \\
 \MM'_1 
 \arrow[r, bend right = 40, dashed] & 
 \MM'_2 &
 \MM'_3
\end{tikzcd}
\end{center}

\caption{The proof of Lemma \ref{lem:t0n-solid}. 
The solid arrows represent isomorphisms given directly by the assumptions about $\MM_1, \MM_2, \MM_3$, and the dashed arrows represent isomorphisms shown to exist during the argument. Composing the arrows gives an isomorphism
from $\MM_1$ onto $\MM_2$.} 

\label{fig:t0n-solid}
\end{figure}

We can use Lemma \ref{lem:indfact} for $m=1$ to infer that there is an $\MM_1$-definable isomorphism between $\NN_1$ and $\NN_2$. This isomorphism in turn clearly gives rise to an $\MM_1$-definable isomorphism between $\MM'_1$ and $\MM'_2$.

By axioms \ref{it:t0n-4} of $\tn{n}$, for each $i$ there 
is an $\MM_i$-definable (hence $\MM_1$-definable)
isomorphism between $\MM_i$ and $\MM'_i$.
Combining this with the isomorphism between
$\MM'_1$ and $\MM'_2$, we obtain
an $\MM_1$-definable isomorphism
between $\MM_1$ and $\MM_2$, which completes the proof.
\end{proof}

\begin{lemma}\label{lem:no-mix-ct-pa}
No model of $\PA$
is a retract of a model of $\CT[\PA]$.
In other words, if $\MM$, $\MM'$ are models of $\PA$, 
$\KK$ is a model of $\CT[\PA]$, and \mbox{$\MM\rhd\KK\rhd\MM'$}, then there is no $\MM$-definable isomorphism from $\MM$ onto $\MM'$.

Similarly, no model of $\CT[\PA]$ is a retract of a model of $\PA$.
\end{lemma}
\begin{proof}
We prove only the first part of the statement,
as the proof of the other part is very similar. Suppose
that $\MM, \MM' \vDash \PA$ and $\KK \vDash \CT[\PA]$
with $\MM\rhd\KK\rhd\MM'$,
but there is an $\MM$-definable isomorphism from $\MM$ onto $\MM'$. Then, by Lemma \ref{lem:indfact} for $m = 0$, there exists an $\MM$-definable isomorphism from $\MM$ onto the $\mathcal{L}_{\PA}$-reduct of $\KK$.

We claim that such an isomorphism would make it possible to define a satisfaction predicate for $\MM$ in $\MM$, contradicting Tarski's theorem on undefinability of truth.
Indeed, let $\mathsf{K}: \MM \rhd \KK$, and let $j$ be the isomorphism from $\MM$ onto $\KK{\upharpoonright}_{\LPA}$. 
Then, since $\KK \vDash \CT[\PA]$,
the formula
\[\sigma(x,y) := \exists x'\, \exists y' \left[j(x) =^{\mathsf{K}} x' \land j(y) =^{\mathsf{K}} y') \land \left(\mathsf{K} \vDash P(\subst(x',\name(y')) \right)\right],\]
evaluated in $\MM$, correctly determines whether a
(standard) $\LPA$-formula $x$ is satisfied in $\MM$ by
an (arbitrary) element $y \in \MM$. 

Note that the formula $\sigma(x,y)$ may involve parameters from $\MM$ required to define the interpretation $\mathsf{K}$ or the isomorphism $j$. However, by Tarski's theorem,
not even a formula with parameters can be a definition of \emph{satisfaction} for formulas with free variables,
in contrast to merely being a definition of \emph{truth} for sentences.
\end{proof}

We can now complete the proof of Theorem \ref{thm:proper-solid}.

\begin{proof}[Proof of Theorem \ref{thm:proper-solid}]
By the definition of $T_n$ and Lemma \ref{lem:tn-proper-subtheory}, we already know
that $T_n$ is an r.e.~subtheory of $\PA$ containing
${\IS{n}} + {\exp} + {\neg\BS{n+1}}$. It remains
to show that $T_n$ is solid.

So, let $\MM_1 \rhd \MM_2 \rhd \MM_3$ be models of $T_n$
such that there is an $\MM_1$-definable isomorphism
between $\MM_1$ and $\MM_3$. We need to prove
that there is an $\MM_1$-definable
isomorphism between $\MM_1$ and $\MM_2$.

By the definition of $T_n$,
each $\MM_i$ satisfies either $\PA$ or $\tn{n}$. Moreover, clearly $\MM_1 \equiv \MM_3$. This leaves four cases to consider.

$1^\circ$ Each $\MM_i$ satisfies $\PA$. Then an $\MM_1$-definable isomorphism between $\MM_1$ and $\MM_2$ exists by the solidity of $\PA$.

$2^\circ$ Each $\MM_i$ satisfies $\tn{n}$. 
Then the isomorphism exists by Lemma \ref{lem:t0n-solid}.

$3^\circ$ $\MM_1 \vDash \PA$ and $\MM_2 \vDash \tn{n}$.
Let $\NN_2$ be the model of $\CT[\PA]$ obtained by applying the interpretation $\gn{n}$ in $\MM_2$, and let $\MM'_2$ be the model of $\tn{n}$ obtained by applying $\hn{n}$ in $\NN_2$.
Note that $i_n$ applied in $\MM_2$ is an isomorphism
between $\MM_2$ and $\MM'_2$, by axiom \ref{it:t0n-4} of $\tn{n}$. Let $\MM'_3$ be the model of $\PA$ obtained by applying in $\MM'_2$ the interpretation provided by the formulas defining the interpretation of $\MM_3$ in $\MM_2$,
but with all parameters of the latter replaced by their $(i_n)^{\MM_2}$-images.

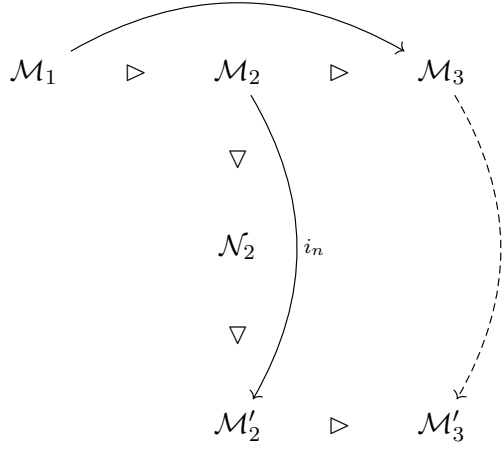
\begin{center}
\begin{tikzcd}[column sep = 4.5em, row sep = 4.5 em ]
 \MM_1 
 \arrow[r, symbol = \rhd]
 \arrow[rr, bend left = 30]
 & 
 \MM_2 
 \arrow[r, symbol = \rhd]
 \arrow[d, symbol = \rhd] 
 \arrow[dd, bend left = 30, "i_n"] & 
 \MM_3 
 \arrow[dd, bend left = 30, dashed ] 
 \\
 & 
 \NN_2
 \arrow[d, symbol = \rhd] & 
 \\
 & 
 \MM'_2 
 \arrow[r, symbol = \rhd] &
 \MM'_3
\end{tikzcd}
\end{center}

Composing interpretations, 
we see that $\MM_1 \rhd \NN_2 \rhd \MM'_3$. 
Moreover, by assumption there is 
an $\MM_1$-definable isomorphism between $\MM_1$ and $\MM_3$,
and the isomorphism $(i_n)^{\MM_2}$ between $\MM_2$ and $\MM'_2$ induces an $\MM_2$-definable (thus, $\MM_1$-definable) isomorphism between $\MM_3$ and $\MM'_3$.
Composing the isomorphisms from $\MM_1$ onto $\MM_3$ and
from $\MM_3$ onto $\MM'_3$, we obtain an $\MM_1$-definable
isomorphism between $\MM_1$ and $\MM'_3$. However,
then the triple of models $\MM_1 \rhd \NN_2 \rhd \MM'_3$ witnesses that $\MM_1$ is a retract of $\NN_2$, contradicting
Lemma \ref{lem:no-mix-ct-pa}.

$4^\circ$ $\MM_1 \vDash \tn{n}$ and $\MM_2 \vDash \PA$. This case is similar to the previous one but slightly more subtle. 
Let $\NN_1 \vDash \CT[\PA]$ be obtained by applying $\gn{n}$ in $\MM_1$,
let $\MM'_1 \vDash \tn{n}$ be obtained by applying $\hn{n}$ in $\NN_1$, 
and let $\NN'_1 \vDash \CT[\PA]$ be obtained by applying $\gn{n}$ in $\MM'_1$.
Note that 
$(i_n)^{\MM_1}$ is an isomorphism
between $\MM_1$ and $\MM'_1$, by axiom \ref{it:t0n-4} of $\tn{n}$,
while $(j_n)^{\NN_1}$ is an isomorphism
between $\NN_1$ and $\NN'_1$, by axiom \ref{it:t0n-5}. 
Let $\MM'_2 \vDash \PA$ 
be obtained by applying in $\MM'_1$ the interpretation 
of $\MM_2$ in $\MM_1$,
but with parameters moved by $(i_n)^{\MM_1}$.
Let $\MM'_3 \vDash \tn{n}$ be obtained by applying in $\MM'_2$ the interpretation of $\MM_3$ in $\MM_2$,
again with parameters moved by $(i_n)^{\MM_1}$.
Finally, let $\NN'_3 \vDash \CT[\PA]$ be obtained by
applying $\gn{n}$ in $\MM'_3$.

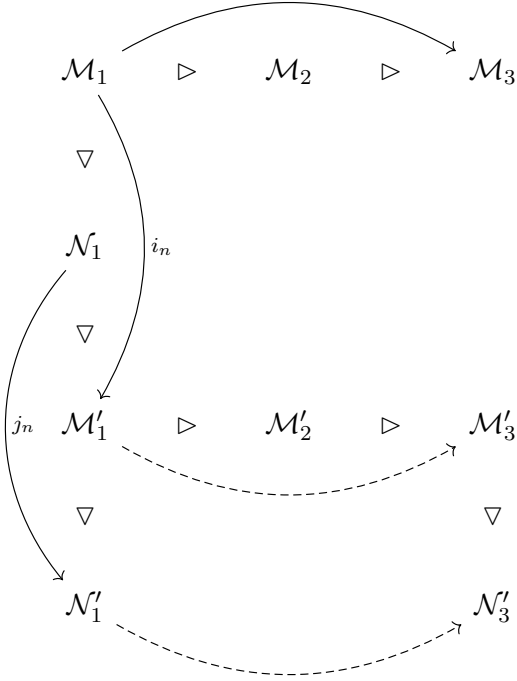
\begin{center}
\begin{tikzcd}[column sep = 4.5em, row sep = 4.5 em ]
 \MM_1 
 \arrow[r, symbol = \rhd]
 \arrow[rr, bend left = 30]
 \arrow[d, symbol = \rhd] 
 \arrow[dd, bend left, "i_n"] 
 & 
 \MM_2 
 \arrow[r, symbol = \rhd]
 & 
 \MM_3 
 \\
 \NN_1
 \arrow[dd, bend right = 40, "j_n"] 
 \arrow[d, symbol = \rhd] 
 & 
 & 
 \\
 \MM'_1 
 \arrow[r, symbol = \rhd]
 \arrow[rr, bend right, dashed]
 \arrow[d, symbol = \rhd] & 
 \MM'_2 
 \arrow[r, symbol = \rhd] &
 \MM'_3 
 \arrow[d, symbol = \rhd] \\
 \NN'_1 
 \arrow[rr, bend right, dashed] &
 &
 \NN'_3
\end{tikzcd}
\end{center}

Composing interpretations, 
we see that $\NN_1 \rhd \MM'_2 \rhd \NN'_3$. 
Moreover, there is an $\MM'_1$-definable, and thus
$\NN_1$-definable, isomorphism from $\MM'_1$ onto $\MM'_3$,
which induces an $\NN_1$-definable isomorphism from
$\NN'_1$ onto $\NN'_3$.
This can be composed with the isomorphism $(j_n)^{\NN_1}$ to give an $\NN_1$-definable isomorphism between $\NN_1$ and $\NN'_3$. 
However, then the triple of models $\NN_1 \rhd \MM'_2 \rhd \NN'_3$ witnesses that $\NN_1$ is a retract of $\MM'_2$, contradicting Lemma \ref{lem:no-mix-ct-pa}.

This concludes the proof that $T_n$ is solid, and thus also the proof of Theorem \ref{thm:proper-solid}.
\end{proof}

\section{Solidity below PA: refinements}\label{sec:refinements}

The current section serves two main purposes. 
Initially, we analyze some
key aspects of the construction from Section \ref{sec:proper-solid}. This is covered by the first two subsections. Then, using the tools developed in this analysis, we obtain two improvements of Theorem \ref{thm:proper-solid}: 
firstly, that the solid theories can be weaker than $\PA$ in terms of not just provability, but also interpretability; secondly, that they can be made weak enough that extending them by any single true sentence, or even by all true $\Pi_n$ sentences for a fixed $n$, will still be insufficient to derive $\PA$.
These improvements are presented in Sections
\ref{subsec:solid-non-interpret} and \ref{subsec:Ali's question}, respectively. 

\subsection{Formalizing the bi-interpretability result}\label{subsec:formalizing-biint}

We start by proving a result that was
already announced above: the interpretations $\gn{n}$ and $\hn{n}$ defined in Section \ref{sec:proper-solid} 
witness not only the bi-interpretability of the specific models $(\N,\Th{\N})$ and $\KK$,
but work in a more general axiomatic context.

\begin{prop}\label{prop:biint_truth}
For each $n \ge 1$, $(\gn{n},\hn{n})$ is a bi-interpretation of
$\tn{n}$ with $\CT[\PA]$.
\end{prop}
\begin{proof}
    Axioms \ref{it:t0n-3} and \ref{it:t0n-4} of $\tn{n}$ explicitly state that 
    $\gn{n}$ is an interpretation of $\CT[\PA]$
    and that $\gn{n}\!\hn{n}$ is isomorphic to the identity 
    interpretation. It remains to show that $\hn{n}$ is really an interpretation of $\tn{n}$ in $\CT[\PA]$,
    not merely in $(\N,\Th{\N})$, and that $\hn{n}\!\gn{n}$ is isomorphic to the identity interpretation provably in $\CT[\PA]$. This is a somewhat routine verification
    that $\CT[\PA]$ is strong enough to carry out various constructions involved in the definitions of
    $\gn{n}, \hn{n}, i_n, j_n$. We provide some details.
    
    The first step is to check that the construction
    of the models $\HH$ and $\KK$ prescribed by $\hn{n}$
    formalizes in $\CT[\PA]$. We begin by verifying that 
    $\CT[\PA]$ proves the consistency of the $\LPA\cup\{P\}$-definable theory
        \[U:= {\IS{n+1}} \cup 
        \{\xi_n(\num{k})\colon\, P(k)\} \cup
        \{\neg\xi_n(\num{\ell}) \colon\, \neg P(\ell)\}.\]
    (note that this is the formalized version of the theory appearing in (\ref{eq:theory-1}) 
    in the definition of $\hn{n}$). 
    Indeed, work in $\CT[\PA]$ and assume that $U$ 
    is inconsistent.
    Then for some disjoint finite sets $c,d$ of numbers
    we have
    \[\neg \Con{\IS{n+1} + \bigwedge_{k\in c} \xi_n(\num{k})\wedge \bigwedge_{\ell\in d} \neg\xi_n(\num{\ell})}\]
    However, then also
    \[\neg \Con{\IS{n+1} + 
    \forall x \left(\xi_n(x)\equiv \bigvee_{k\in c} x = \num{k}\right)}.\]
    Since $\xi_n$ is provably $\Sigma_1$-flexible over $\IS{n+1}$ in the sense of 
    Corollary \ref{cor:flexible-con},
    and $\bigvee_{k\in c} x = \num{k}$ is a $\Sigma_1$ formula (quantifier-free, in fact), 
    Corollary \ref{cor:flexible-con}
    gives us $\neg \Con{\IS{n+1}}$.
    But we are working in $\CT[\PA]$, so we do have
    $\Con{\IS{n+1}}$, and thus we reach a contradiction.
    This concludes the proof of $\Con{U}$ in $\CT[\PA]$.

    From now on it will be convenient to assume that we are working with
    a given model $\MM\vDash\CT[\PA]$. We already know that $\MM\vDash\Con{U}$, so we can apply the usual construction associated with the  arithmetized completeness theorem (see e.g.~\cite[Theorem 13.13]{kaye:models})
    to $U$ in $\MM$.
    We only need the construction in its most basic form:
    produce a ($\Delta_1(P)$-definable) henkinized consistent extension of $T$ and take the ($\Delta_2(P)$-definable) leftmost path through the infinite binary tree whose paths correspond to complete consistent extensions of the henkinization. Thus we obtain a model of $U$,
    say $\HH^\MM$, definable in $\MM$ by a formula that by abuse of notation we could call $\HH$. 
    In fact, the structure $\HH^\MM$ has an $\MM$-definable satisfaction relation (or, in other words, $\MM$-definable elementary diagram). 
    Thus, we can easily define the structure $\MM^{\hn{n}}:= (\mathcal{K}^{n+1})^\MM(\mathcal{H}^\MM)$ consisting of those elements of $\HH^\MM$ that are definable in $\HH^\MM$ by $\Sigma_{n+1}$ formulas in the sense of $\MM$. 
    In contrast to $\HH^\MM$, from the point of view of $\MM$ the structure $\MM^{\hn{n}}$ is only a partial model in the sense that $\MM$ cannot define the full satisfaction relation for $\MM^{\hn{n}}$ but only its universe and operations. Of course, that is already enough to define 
    satisfaction for $\Sigma_{m}$ formulas for any fixed $m$. In other words, we have $\MM$-definable predicates $\HH\vDash \varphi(x)$ and $\hn{n}\vDash_m \varphi(x)$, for any $m \in \omega$, which agree with satisfaction 
    in $\HH^\MM$ resp.~$\MM^{\hn{n}}$ for atomic formulas and satisfy the usual inductive clauses of a definition of satisfaction for all $\MM$-formulas resp.~for all $\MM$-formulas that belong to the class $\Sigma_m$.
    
    We still have to check that $\hn{n}$ provides an interpretation of $\tn{n}$, or in other words, that $\mathcal{M}^{\hn{n}}\vDash \tn{n}$. In the process,
    we will also check that $\hn{n}$ and $\gn{n}$ give
    rise to a bi-interpretation. 
    
    The verification that $\mathcal{M}^{\hn{n}}$ is a model of $\IS{n} + \exp$ and that $\Sigma_{n+1}$-elementarity holds between
    $\MM^{\hn{n}}$ and $\mathcal{H}^\MM$, i.e.~that 
    \[\mathcal{M}\vDash\forall \varphi \! \in \! \form_{\Sigma_{n+1}} \forall x \! \in \! \hn{n} \,
    \bigl(\hn{n}\vDash_{n+1}\varphi(x)\leftrightarrow \mathcal{H}\vDash\varphi(x)\bigr),\]
    is straightforward. 
    
    Recall the map named $j_n$ in Lemma \ref{lem:in-and-jn} 
    and first introduced in the proof of Lemma \ref{ultimate_dedekind}, namely the one taking $k \in \MM$ to the $k$-th smallest element of $\MM^{\hn{n}}$. 
    By the argument from Section \ref{subsec:formalized-categoricity},
    the map $j_n$ is an $\MM$-definable embedding 
    of $\mathcal{M}{\upharpoonright}_{\LPA}$ onto an initial segment of $\mathcal{M}^{\hn{n}}$. Let $\mathcal{J}$ be the ($\MM$-definable) image of $j_n$. 
    We know that $\mathcal{J}$ 
    is also an initial segment of $\HH^\MM$,
    since every element of $\HH^\MM$ that
    is below $j_n(k)$ is named by a numeral from $\MM$,
    so it is $\Sigma_{n+1}$-definable in the sense of $\MM$. 
    Moreover, $\mathcal{J}$ is a proper cut in $\mathcal{M}^{\hn{n}}$: otherwise, $\mathcal{M}^{\hn{n}}$ would be isomorphic to $\MM{\upharpoonright_{\LPA}}$, 
    which cannot happen by Corollary \ref{cor:flexible-con}, 
    because $\MM \vDash \Con{\IS{n+1}}$, 
    while $\HH^\MM$ and as a consequence $\MM^{\hn{n}}$ both
    satisfy $\exists x\, \xi_n(x)$. 
    
    By the definition of $\MM^{\hn{n}}$
    and $\Sigma_{n+1}$-elementarity,
    each element of $\MM^{\hn{n}}$ can be $\Sigma_{n+1}$-defined by a formula in $\mathcal{J}$.
    This lets us carry out the usual argument
    showing that $\MM^{\hn{n}} \vDash \neg \BS{n+1}$,
    so $\MM^{\hn{n}}$ validates axiom \ref{it:t0n-1} of
    $\tn{n}$.
    
    Clearly, $\mathcal{J}$ is the smallest $\MM$-definable cut in $\mathcal{M}^{\hn{n}}$ (and thus, also the
    smallest $\mathcal{M}^{\hn{n}}$-definable cut), because
    otherwise $\MM$ would define its own proper cut.
    This implies in particular that
    $\mathcal{J} \subseteq \delta_n^{\MM^{\hn{n}}}$.
    But we also have $\delta_n^{\MM^{\hn{n}}} \subseteq \mathcal{J}$, because each element of $\MM^{\hn{n}}$ 
    has a $\Sigma_{n+1}$ definition in $\mathcal{J}$.
    So, $\delta_n^{\MM^{\hn{n}}} = \mathcal{J}$,
    and thus $\MM^{\hn{n}}$ satisfies axioms \ref{it:t0n-2} of $\tn{n}$. 
    
    For each $k \in \MM$, we have that $P(k)$ holds in $\MM$ exactly if $\xi_n(j_n(k))$ holds in $\HH^\MM$. This is because $\MM \vDash (\HH \vDash U)$ and 
    $j_n(k)$ is the element named by the numeral $\num{k}$
    in $\HH^\MM$. The truth values of $\xi_n$ are the same in $\MM^{\hn{n}}$ as in $\HH^\MM$, by $\Sigma_{n+1}$-elementarity.
    So, by the definition of $\gn{n}$, we 
    indeed have $\MM^{\hn{n}\!\gn{n}}\vDash \CT[\PA]$,
    which means that $\MM^{\hn{n}}$ satisfies axioms \ref{it:t0n-3}. Moreover, we have just shown that $j_n$ is an isomorphism between $\MM$ and $\MM^{\hn{n}\!\gn{n}}$. This also implies that (the map defined by the same formula as) $j_n$ is an isomorphism between $\MM^{\hn{n}\!\gn{n}}$
    and $\MM^{\hn{n}\!\gn{n}\!\hn{n}\!\gn{n}}$, so
    $\MM^{\hn{n}}$ satisfies axiom \ref{it:t0n-5}.

    Finally we argue that $\MM^{\hn{n}}$ satisfies (iv). 
    Since $\mathcal{J} = \delta_n^{\MM^{\hn{n}}}$ is the shortest cut in $\MM^{\hn{n}}$, and each element of $\MM^{\hn{n}}$ is definable via a $\Sigma_{n+1}$-definition from $\mathcal{J}$, the definition of $i_n$ makes sense: each element of $\MM^{\hn{n}}$ has a least $\Sigma_{n+1}$-definition. To verify that $i_n$ is indeed an isomorphism 
    between $\MM^{\hn{n}}$ and $\MM^{\hn{n}\!\gn{n}\!\hn{n}}$,
    one uses the fact that $\MM$ and $\MM^{\hn{n}\!\gn{n}}$ are isomorphic via $j_n$.
\end{proof}

\subsection{Modularizing the construction}\label{subsec:modularizing}

In the proof of Theorem \ref{thm:proper-solid}, 
we made use 
of Lemma \ref{lem:no-mix-ct-pa}: no model of $\PA$ can
be a retract of a model of $\CT[\PA]$, and vice versa.
We now carry out a more general study of families of theories
with this extreme form of non-bi-interpretability property.
Infinite families of this kind will be needed 
in our proofs of refinements of Theorem \ref{thm:proper-solid}.

\begin{definition}
    We say that a family $\{U_k\}_{k\in\omega}$ of theories is \emph{retract-disjoint} if for any $k,n\in\omega$ the following holds: if $\MM\vDash U_k$ and $\NN\vDash U_n$ and $\MM$ is a retract of $\NN$, then $k=n$.
\end{definition}

\begin{definition}
    Let $\{U_k\}_{k\in\omega}$ be a sequence of theories and $\{\varphi_k\}_{k\in\omega}$ be a sequence of sentences. The symbol $\bigoplus_k (U_k|\varphi_k)$ denotes the theory $\{\varphi_k\rightarrow \psi \colon\, k \in \omega, \psi\in U_k\}$.
\end{definition}

\begin{prop}\label{pres_solid_disj_sum}
    Suppose that $\{U_k\}_{k \in \omega}$ is a family of solid theories, $V$ is a theory, and $\{\varphi_k\}_{k\in\omega}$ is a sequence of sentences 
    with the following properties:
    \begin{enumerate}
        \item the sentences $\varphi_k$ are pairwise inconsistent, 
        \item the theory $V \cup \{\neg\varphi_k \colon\, k\in\omega\}$ is solid, 
        \item the family $\{V \cup \{\neg\varphi_k \colon\, k\in\omega\}\} \cup \{U_k\}_{k \in \omega}$ is retract-disjoint.
    \end{enumerate}
    Then the theory $V \cup \bigoplus_k(U_k|\varphi_k)$ is solid.
\end{prop}
\begin{proof}
    Assume that $\MM \vDash V \cup \bigoplus_k(U_k|\varphi_k)$ and there is a retraction $(\mathsf{N},\mathsf{M})$ such that $\MM^{\mathsf{N}}\vDash  V \cup \bigoplus_k(U_k|\varphi_k)$. Put $\NN:= \MM^{\mathsf{N}}$. By our assumptions, exactly one of the following two cases holds:
    \begin{enumerate}[1$^\circ$]
        \item there is $k\in\omega$ such that both $\MM$ and $\NN$ are models of $U_k$,
        \item both $\MM$ and $\NN$ are models of 
        $V + \{\neg\varphi_k \colon\, k\in\omega\}$.
    \end{enumerate}
    By solidity of each of the relevant theories,
    in either case there is 
    an $\MM$-definable isomorphism from 
    $\MM$ onto $\NN$.
\end{proof}

\begin{remark}\label{uwaga:skonczone_sumy}
Officially, Proposition {\ref{pres_solid_disj_sum}} and the definitions preceding it deal with an infinite family of theories. However, it can be easily checked that the finite case is covered by it as well: 
for a fixed $n$, a family $\{V\} \cup \{U_i\}_{i < n}$ of theories and a family $\{\varphi_i\}_{i < n}$ of sentences satisfying the assumptions of Proposition {\ref{pres_solid_disj_sum}} with $n$ in place of $\omega$, 
the theory  
${V} \cup {\bigoplus_{i < n}(U_i|\varphi_i)}$ (defined in the obvious way)
is solid.
\end{remark}

\begin{lemma}\label{lem:truthin}
The family $\{\CT^k[\PA]\}_{k\in\omega}$ is retract-disjoint.
\end{lemma}
\begin{proof}
The argument is a slight generalization of the proof of
Lemma \ref{lem:no-mix-ct-pa}.
Fix $\mathcal{M}_1\vDash \CT^k$ and $\mathcal{M}_2\vDash \CT^m$ and assume that $\mathcal{M}_1$ is a retract of $\mathcal{M}_2$ as witnessed by the retraction $(\mathsf{M}_2, \mathsf{M}_1)$. Without loss of generality assume that $\mathcal{M}_2 = \mathcal{M}_1^{\mathsf{M}_2}$. Aiming at a contradiction assume further that $k\neq m$. Let $\ell:=\min\{k,m\}$. 
By Lemma \ref{lem:indfact}, 
we conclude that for $i\in \{1, 2\}$ the 
$\mathcal{L}_\ell$-reducts
of $\mathcal{M}_i$ and $\mathcal{M}_{3-i}$ are isomorphic via an $\mathcal{M}_i$-definable isomorphism.  

Assume 
that $\ell=k < m$ (the other case is treated analogously). 
Arguing as in the proof
of Lemma \ref{lem:no-mix-ct-pa}, we can use the $(k+1)$-th truth predicate of $\MM_2$ to define satisfaction for 
$\mathcal{L}_k$-formulas in $\MM_1$ within $\MM_1$,
contradicting the undefinability of truth.
\end{proof}

The next proposition and its consequence, Corollary \ref{cor_biint_pres_retract}, will play a key role in verifying retract-disjointness of various families of theories.
It is probably most convenient to follow the statement of the proposition and its proof while looking at
Figure \ref{fig:lem_o_wezu} below.

\begin{prop}\label{prop_lem_o_wezu}
    Suppose that $\MM, \NN, \KK, \mathcal{G}$ are structures such that:
    \begin{itemize}
        \item $\MM$ is a retract of $\NN$ via the retraction $(\mathsf{N}, \mathsf{M})$, 
         \item $\mathcal{N}$ is a retract of $\mathcal{G}$ via the retraction $(\mathsf{G},\mathsf{N_\star})$,
        \item $\MM$ is bi-interpretable with $\KK$ via the bi-interpretation $(\mathsf{K},\mathsf{M_\circ})$.
        \end{itemize}
    Then $\KK$ is a retract of $\mathcal{G}$. Moreover, the witnessing retraction can be chosen such that the isomorphism between $\MM^{\mathsf{NG}}$ and the copy of $\mathcal{G}$ in $\MM^{\mathsf{K}}$ is $\MM$-definable.
\end{prop}
\begin{figure}
    \centering
    \begin{tikzcd}[column sep=3em, row sep=6ex,
execute at end picture={
\node[anchor=center,xshift=-5em,yshift=-5em] at (current bounding box.center)
{
\begin{tabular}{l}
$\textcolor{darkgray}{\NN\simeq\MM^\mathsf{N}}$\\[1.5ex]
$\textcolor{darkgray}{\mathcal{G}\simeq\MM^\mathsf{NG}}$\\[1.5ex]
$\textcolor{darkgray}{\KK\simeq\MM^\mathsf{K}}$
\end{tabular}
};
}
]
\mathcal{M} \arrow[rr, bend left = 35, "i"] \arrow[dd, bend right = 50] \arrow[d,symbol=\rhd,"\;\mathsf{K}"] \arrow[r, symbol=\rhd^\mathsf{N}] &
\MM\mathrlap{^\mathsf{N}} \arrow[dd, bend right = 50, "\,j"] \arrow[r,symbol={\phantom{^\mathsf{N}}\rhd^\mathsf{M}}] \arrow[d,symbol=\rhd,"\;\mathsf{G}"]  &
\mathcal{M}\mathrlap{^\mathsf{NM}} &
{} \arrow[ddddd,shorten <=-2em,shorten >=2em,xshift=-0.5em,dash,dashed] &
\mathcal{M} \arrow[rr, bend left = 35]\arrow[d,symbol=\rhd, "\;{\mathsf{K}}"] \arrow[r,symbol=\rhd^{\mathsf{N}}] \arrow[dd, bend right = 50] &
\mathcal{N} \arrow[r,symbol=\rhd^{\mathsf{M}}] &
\widetilde{\MM}
\\
\MM\mathrlap{^\mathsf{K}} \arrow[dd, bend right = 50] \arrow[d,symbol=\rhd,"\;\mathsf{M_\circ}"] &
\MM\mathrlap{^\mathsf{NG}} \arrow[d,symbol=\rhd,"\;\mathsf{N_\star}"] &
&
{} &
\KK \arrow[dd, bend right = 50] \arrow[d,symbol=\rhd,"\;{\mathsf{M_\circ}}"] &
&
\\
\MM\mathrlap{^\mathsf{KM_\circ}} \arrow[d, symbol = \rhd,"\;\mathsf{K}"] &
\MM\mathrlap{^\mathsf{NGN_\star}} &
&
{} &
\MM_1\arrow[rr, bend left  = 35, "i_1", dashed] \arrow[d, symbol = \rhd, "\;{\mathsf{K}}"] \arrow[r,symbol=\rhd^{\mathsf{N}}] &
\NN_1 \arrow[dd, bend right = 50, "\,j_1", dashed] \arrow[d, symbol = \rhd, "\;{\mathsf{G}}"] \arrow[r,symbol=\rhd^{\mathsf{M}}] &
\widetilde{\MM}_1 \arrow[dd, "\widehat{j_1}", dashed]
\\
\MM\mathrlap{^\mathsf{KM_\circ K}} &
&
&
{} &
\KK_1 \arrow[rrdd, bend right = 40, dashed] &
\mathcal{G}_1 \arrow[d, symbol = \rhd, "\;{\mathsf{N_\star}}"] &
\\
&
&
&
{} &
&
\NN_2 \arrow[r,symbol=\rhd^{\mathsf{M}}] &
\MM_2 \arrow[d, symbol = \rhd, "\;{\mathsf{K}}"]
\\
&
&
&
{} &
&
&
\KK_2
\end{tikzcd}
    \caption{Proposition \ref{prop_lem_o_wezu} and its proof. The diagram on the left illustrates the situation described by the assumptions of the proposition. The diagram on the right illustrates the proof, using the notational conventions introduced at the beginning 
    of the proof.}
    \label{fig:lem_o_wezu}
\end{figure}
In the statement of the proposition, the mere fact 
that $\KK$ is a retract of $\mathcal{G}$ can be obtained from
the transitivity of retracts, which is a special case of the proposition with a considerably simpler proof. However,
a more involved argument seems needed to obtain 
the definability relation expressed in
the ``moreover'' part of the statement.

\begin{proof}
    Note that, by our convention, $\NN$ is isomorphic to $\MM^{\mathsf{N}}$
    and $\KK$ is isomorphic to $\MM^{\mathsf{K}}$; the interpretations
    $\mathsf{N}$ and $\mathsf{M}$ witness
    that $\MM$ is a retract of $\NN$; etc. Below, we will identify $\NN$ with $\MM^{\mathsf{N}}$, $\KK$ with $\MM^{\mathsf{K}}$, and $\mathcal{G}$ with $\NN^{\mathsf{G}}$ to simplify the notation.
    Also for the sake of notational simplicity, we assume that all the interpretations and isomorphisms involved are definable without parameters; otherwise, nothing substantial would change but we would have to keep track of how the parameters are mapped by various isomorphisms.

Let
\begin{gather*}
    \widetilde{\MM}:=\NN^\mathsf{M},\\
    \MM_1:=\mathcal{K}^{\mathsf{M_\circ}},\quad\mathcal{N}_1:= \MM_1^{\mathsf{N}},\quad\widetilde{\MM}_1:=\NN_1^\mathsf{M},\quad\KK_1:=\MM_1^{\mathsf{K}},\quad\mathcal{G}_1:=\NN_1^{\mathsf{G}},\\
    \NN_2:=\mathcal{G}_1^{\mathsf{N_\star}},\quad\MM_2:=\NN_2^{\mathsf{M}},\quad\KK_2:=\MM_2^{\mathsf{K}}\text{.}
\end{gather*}
By our assumptions, $\mathsf{K}$ and $\mathsf{M_\circ}$ witness that $\KK$ is a retract of $\MM$, so there is an $\MM$-definable isomorphism $\MM\rightarrow \MM_1$. This isomorphism induces further $\MM$-definable isomorphisms  $\NN\rightarrow \NN_1$ 
and $\widetilde{\MM}\rightarrow\widetilde{M}_1$. 
Thus $\mathcal{G}$ is $\MM$-definably isomorphic to $\mathcal{G}_1$,
and the interpretations 
$\mathsf{G}$ and $\mathsf{N_\star}$ witness that $\mathcal{N}_1$ is a retract of $\mathcal{G}_1$. 
Moreover, since $\NN$ and $\mathcal{G}^{\mathsf{N_\star}}$ are $\NN$-definably isomorphic, say via $j$, 
there is an $\NN_1$-definable isomorphism $j_1:\NN_1\rightarrow\NN_2$. Composing the interpretations $\mathsf{M_\circ}$, $\mathsf{N}$, and $\mathsf{G}$ witnesses that $\mathcal{G}_1$ is interpretable in $\mathcal{K}$. Also, $j_1$ induces an $\mathcal{N}_1$-definable isomorphism $\widehat{j_1}$ between $\widetilde{M}_1$ and $\MM_2$; and
$\KK_2$ is interpretable in $\mathcal{G}_1$ via the composition 
of $\mathsf{N_\star}$, $\mathsf{M}$,
and $\mathsf{K}$.

So, we have interpretations 
of $\mathcal{G}_1$ in $\KK$ and 
of $\KK_2$ in $\mathcal{G}_1$.
We want to show that there is 
a $\KK$-definable isomorphism between 
$\KK$ and $\KK_2$.
Note firstly that, by the choice of $\mathsf{M}$ and $\mathsf{N}$, there is an $\MM$-definable isomorphism $i: \MM\rightarrow\widetilde{\MM}$.
Since $\MM$ and $\MM_1$ are isomorphic, the formula defining $i$ in $\MM$ determines an $\MM_1$-definable isomorphism $i_1\colon\MM_1\rightarrow\widetilde{\MM}_1$.
Composing with $\widehat{j_1}$ yields
an $\MM_1$-definable isomorphism $\MM_1\rightarrow \MM_2$,
which induces 
an $\MM_1$-definable (and hence $\KK$-definable) isomorphism $\KK_1\rightarrow \KK_2$. 
But $\mathsf{K}$ and $\mathsf{M_\circ}$ are chosen
so as to witness the bi-interpretability
of $\MM$ and $\KK$, so there is a $\KK$-definable
isomorphism $\KK \rightarrow \KK_1$.
Composition with the isomorphism 
$\KK_1\rightarrow \KK_2$ just constructed gives the desired $\KK$-definable isomorphism $\KK \rightarrow \KK_2$. 
 
This completes both the proof that $\KK$ is a retract of $\mathcal{G}$ and of the ``moreover'' clause, because we have shown that $\KK$ is a retract of $\mathcal{G}_1$ and that $\mathcal{G}_1$ is $\MM$-definably isomorphic to $\mathcal{G}$.
\end{proof}

\begin{corollary}\label{cor_biint_pres_retract}
Suppose that $\{U_k\}_{k\in\omega}$ and $\{V_k\}_{k\in\omega}$ are two sequences of theories such that for each $k$, the theory $U_k$ is bi-interpretable with $V_k$. Then if $\{U_k\}_{k\in\omega}$ is retract-disjoint, then $\{V_k\}_{k\in\omega}$ is retract-disjoint.
\end{corollary}
\begin{proof}
Assume that $m \neq n$ and that $\MM \vDash V_m$ is a retract of $\NN \vDash V_n$. By the assumption on bi-interpretability between theories, there are $\KK \vDash U_m$ and
$\mathcal{G} \vDash U_n$ which are bi-interpretable with
$\MM$ and $\NN$, respectively. By Proposition
\ref{prop_lem_o_wezu}, $\KK$ is a retract 
of $\mathcal{G}$,
so the family $\{U_k\}_{k\in\omega}$ is not retract-disjoint.
\end{proof}

The fact expressed in the corollary below 
was first observed in \cite{enayat:variations}. Here we derive it using Proposition \ref{prop_lem_o_wezu}.

\begin{corollary}\label{cor_biint_pres_solid}
    If $U$ and $V$ are bi-interpretable theories and $U$ is solid, then $V$ is solid.
\end{corollary}
\begin{proof}
Let $U$ and $V$ be as above, and let $\mathsf{V}: U\rhd V$ and $\mathsf{U}: V\rhd U$ witness the bi-interpretability.
Assume that $\MM\vDash V$ and let $(\mathsf{N},\mathsf{M})$ be a retraction in $\MM$ such that $\MM^{\mathsf{N}} \vDash V$. Put $\NN = \MM^{\mathsf{N}}$. 

Applying Proposition \ref{prop_lem_o_wezu} to $\MM$, $\NN$, $\KK := \MM^{\mathsf{U}}$, and $\mathcal{G} := \NN^{\mathsf{U}}$,
we obtain a retraction $(\mathsf{G},\mathsf{K})$ witnessing that $\mathcal{K}$ is a retract of $\mathcal{G}$, as well as an $\MM$-definable isomorphism $\iota$ between $\MM^{\mathsf{N}\mathsf{U}}$ ($=\mathcal{G}$) and $\MM^{\mathsf{U}\mathsf{G}}$  ($=\KK^{\mathsf{G}}$) . 

By the solidity of $U$, there is a $\KK$-definable isomorphism between 
$\KK$ and $\KK^{\mathsf{G}}$, which
induces an $\MM$-definable isomorphism between 
$\MM^{\mathsf{UV}}$ ($=\KK^{\mathsf{V}}$) and $\KK^{\mathsf{G}\mathsf{V}}$. Composing this isomorphism with $\iota$, we obtain an $\MM$-definable isomorphism between $\MM^{\mathsf{UV}}$ and $\NN^{\mathsf{UV}}$ ($=\mathcal{G}^{\mathsf{V}}$).
However, $\MM^{\mathsf{U}\mathsf{V}}$ is $\MM$-definably isomorphic to $\MM$, and $\NN^{\mathsf{U}\mathsf{V}}$ is $\NN$-definably (hence, $\MM$-definably) isomorphic to $\NN$, so there is an $\MM$-definable
isomorphism between $\MM$ and $\NN$. 
\end{proof}

\subsection{A solid theory not interpreting PA}\label{subsec:solid-non-interpret}

The solid theories $T_n$ constructed in Section \ref{sec:proper-solid} are weak in the sense that they
do not prove $\PA$. However, they are relatively strong in the sense of interpretability, and in particular they clearly interpret $\PA$ -- for each $n$ the $\LPA$-part of $\gn{n}$ interprets $\PA$ in $\tn{n}$, and then an interpretation of $\PA$ in $T_n$ is obtained by an obvious case distinction. 

So, it is quite natural to ask whether there are solid subtheories of $\PA$ that are also strictly weaker than $\PA$ in terms of interpretability. This subsection contains a proof of the following theorem, which provides a positive answer to that question.

\begin{theorem}\label{main_thm_solid_nointerpret_PA}
For every $n\in\mathbb{N}$, there exists an r.e. solid subtheory of $\PA$ that contains $\IS{n}$ but not $\BS{n+1}$ and that does not interpret $\PA$.
\end{theorem}

Let $\tn{n}$ be defined as in Section \ref{sec:proper-solid}. We observe that 

\begin{lemma}\label{cor_tn_nonresninterpret_pa}
For each $n\geq 1$, there is no $\Sigma_{n}$-restricted interpretation of $\PA$ in $\tn{n}$.
\end{lemma}
\begin{proof}
    Assume that $\mathsf{M}$ is a $\Sigma_n$-restricted interpretation of $\EA$ in $\tn{n}$. 
    By definition, $\tn{n}$ contains
    $\IS{n}$ and proves $\exists x\,\xi_n(x)$ for the the $\Sigma_1$-flexible formula
    $\xi_n$. So, by Corollary \ref{cor_downward_prec_pi1}, we have $\tn{n}\vdash (\exists x \, \xi_n(x))^{\mathsf{M}}$. But $\PA\vdash \Con{\IS{n}}$, whereas ${\EA} + {\exists x \, \xi_n(x)} \vdash \neg\Con{\IS{n}}$ by Corollary 
    \ref{cor:flexible-con}. Hence, $\mathsf{M}$ is not an interpretation of $\PA$.
\end{proof}
Now, the intuition is that the theory we are looking for is roughly the following one (``p'' stands for ``proto-''):
\[pT_n:=\IS{n}+\exp+\bigoplus_{k\geq n}(\tn{k}|{\IS{k}}\land{\lnot\IS{k+1}}).\]
\begin{lemma}\label{lem_noninterpret_PA}
    The theory $pT_n$ does not interpret $\PA$.
\end{lemma}
\begin{proof}
    Suppose that $\mathsf{M}$ is an interpretation of $\PA$ in $pT_n$. Let $k \geq n$ be such that $\mathsf{M}$ is $\Sigma_k$-restricted. Then $\mathsf{M}$ is also a $\Sigma_k$-restricted interpretation of $\PA$ in the consistent theory $pT_n+{\IS{k}}+{\neg \IS{k+1}}$. However, the latter theory coincides with $\tn{k}$, which contradicts Lemma \ref{cor_tn_nonresninterpret_pa}.
\end{proof}

The problem with $pT_n$ is that the family $\{\tn{k}\}_{k\in\omega}$ is not retract-disjoint. 
As a result, there are models $\MM, \NN\vDash pT_n$ such that $\MM$ is a retract of $\NN$ but the two structures satisfy $pT_n$ ``for different reasons'' and thus cannot
be isomorphic. 
For example, consider the models of $\tn{2}$ and $\tn{3}$,
respectively, obtained by applying the interpretations $\hn{2}$ and $\hn{3}$ from Section \ref{sec:proper-solid}
in $(\N, \Th{\N})$. These models are not even elementarily equivalent, but they are both bi-interpretable
with $(\N, \Th{\N})$, so they are not only retracts
but in fact bi-interpretable with one another.

We will now show how to improve the theories $\tn{n}$ 
so as to eliminate such patterns. Below we use the same interpretations $\hn{n}$ and $\gn{n}$ as before, in Section \ref{sec:proper-solid}.

Let $\zeta(x)$ be a $\Sigma_{1}$-flexible formula over the arithmetical consequences of $\CT[\PA]$. Since, by Proposition \ref{prop:biint_truth}, $\mathsf{N}_n$ is part of a pair of interpretations witnessing the bi-interpretability of $\tn{n}$ and $\CT[\PA]$, it follows that $\zeta$ is $\Sigma_{1}$-flexible over
\[\{\sigma\in\LPA\colon\, \tn{n}\vdash\sigma^{\mathsf{N}_n}\}.\] 
Let $\tnd{n}:= \tn{n}+(\forall x\,(\zeta(x)\leftrightarrow x=\num{n}))^{\mathsf{N}_n}$ (here ``D'' stands for ``disjoint''). By flexibility of $\zeta$, the theory $\tnd{n}$ is consistent for each $n$. 
By Proposition \ref{prop:biint_truth}, we obtain:
\begin{lemma}\label{lem_biint_tnd_truth}
    For each $n\in\omega$, $(\gn{n}, \hn{n})$ is a bi-interpretation of 
    $\tnd{n}$ with 
    $\CT[\PA]+\forall x\, (\zeta(x)\leftrightarrow x=n)$.
\end{lemma}

Since bi-interpretability preserves solidity (Corollary \ref{cor_biint_pres_solid}), and
since each extension of $\CT[\PA]$ in the same language is solid (a consequence of Corollary \ref{cor:solid_ctn}), we obtain:

\begin{corollary}\label{cor_tnd_solid}
    For each $n$, the theory $\tnd{n}$ is solid.
\end{corollary}

One last missing piece is the retract-disjointness of
the theories $\tnd{n}$.

\begin{corollary}\label{lem_retract_disj}
The family $\{\PA\}\cup\{\tnd{n}\}_{n\in\omega}$ is retract-disjoint.
\end{corollary}
\begin{proof}
    By Lemma \ref{lem:no-mix-ct-pa}, Corollary \ref{cor_biint_pres_retract} and Lemma \ref{lem_biint_tnd_truth}, it is enough to check that the family $\{V_n\}_{n\in\omega}$ defined by
    \[V_n:= \CT[\PA] + \forall x\, (\zeta(x)\leftrightarrow x=\num{n})\]
    is retract-disjoint. However, this easily follows from the solidity of $\CT[\PA]$.
\end{proof}

Finally, we define $TD_{n}$ to be
\[\IS{n}+\exp+\bigoplus_{k \ge n}(\tnd{k}|{\IS{k}} \wedge {\neg\IS{k+1}}).\]

We observe that the axioms of $\tnd{k}$ can be effectively generated given $k$, so $TD_n$ really is an r.e.~subtheory of $\PA$. By repeating the argument from the proof of Lemma \ref{lem_noninterpret_PA} (which requires using
the obvious variant of Corollary \ref{cor_tn_nonresninterpret_pa} for $\tnd{k}$ instead
of $\tn{k}$), we obtain: 

 \begin{corollary}\label{cor:tdn-not-interpret-pa}
     The theory $TD_n$ does not interpret $\PA$.
 \end{corollary}

\begin{proof}[Proof of Theorem \ref{main_thm_solid_nointerpret_PA}]
We have already observed that $TD_n$ is an r.e.~subtheory
of $\PA$, and in Corollary \ref{cor:tdn-not-interpret-pa}
we have shown that it does not interpret $\PA$.
Clearly, $TD_n$ contains $\IS{n}$ by definition.
It is also consistent with $\neg\BS{n+1}$, because
$\tnd{n}$ is consistent and implies $\neg\BS{n+1}$.
So, it remains to show that $TD_n$ is solid.

To this end, we want to invoke Proposition \ref{pres_solid_disj_sum} with $U_k := \tnd{k}$, 
$V:= \IS{n}+\exp$,
and $\varphi_k:= {\IS{k}} \wedge {\neg\IS{k+1}}$
(where in $\varphi_0$ we use a finite
fragment of $\mathrm{I}\Delta_0$ that is
equivalent to $\mathrm{I}\Delta_0$ assuming $\exp$).
By Corollary \ref{cor_tnd_solid},
each theory $U_k$ is solid.
The sentences $\varphi_k$ are obviously pairwise inconsistent. It is easy to see that 
$V \cup \{\neg\varphi_k\colon\, k\in\omega\}$ is simply $\PA$, hence it is also a solid theory. The last thing we need to verify is that
the family consisting of PA
 and of $\tnd{0},\tnd{1},\ldots$ is retract-disjoint. This was established in Lemma \ref{lem_retract_disj}.
\end{proof}

\begin{remark}
Note that each model of the theory $TD_n$ actually interprets a model of $\PA$. Intuitively, the reason why these interpretations cannot be merged into a single interpretation of $\PA$ in $TD_n$ is that $TD_n$ is defined by an ``infinite case distinction'', and a finite formula is unable to decide which of the cases applies. We return to this topic in the final section of the paper.  
\end{remark}

\subsection{A solid theory infinitely below PA}\label{subsec:Ali's question}

This section improves on our main result from Section \ref{sec:proper-solid} by providing an example of 
a solid subtheory of $\PA$ which does not imply $\PA$ even after being strengthened by an arbitrary true sentence, 
or in fact by all true $\Pi_k$ sentences where $k$ is an arbitrary fixed number\footnote{We are grateful to Ali Enayat for asking the question which led to this improvement.}.

\begin{theorem}\label{main_thm_solid_inf-below_PA}
    For each $n \in \omega$, there is a solid r.e.~subtheory $TF_n$ of $\PA$ such that $TF_n\vdash \IS{n} + \exp$,
    but for each $k$ it holds that $TF_n + \Thr{\Pi_k}{\mathbb{N}}\nvdash \PA$. 
\end{theorem}

One of the motivations for this theorem is a
result of Wilkie's that we will now explain. An $\LPA$ scheme template is an $\LPA$-formula $\theta(P)$ with a marked fresh predicate letter $P$. For a first-order formula $\psi(x)$, the notation $\theta[\psi/P]$ 
stands for the formula obtained by replacing each occurrence of a subformula $P(t)$ with $\psi(t)$ (preceded by renaming the bound variables if necessary). The scheme generated by the template $\theta(P)$,
denoted by $\theta[\LPA]$, 
is the set $\{\theta[\psi/P]\colon\, \psi(x)\in \LPA\}$.
A scheme template $\theta(P)$ is \emph{restricted} if it is of the form
\[Q_0x_0 \!\in\! P\, Q_1x_1 \! \in \! P \ldots Q_n x_n \!\in \! P\,\varphi(x_0,\ldots,x_n),\]
where $P$ does not occur in $\varphi(x_0,\ldots,x_n)$ and $Q_i x_i \! \in \! P\, \psi$ denotes either $\exists x_i\, (P(x_i)\wedge \psi)$ or $\forall x_i\, (P(x_i)\rightarrow \psi)$. We say that $\theta(P)$ is \emph{second-order categorical} if $(\mathbb{N}, \mathcal{P}(\mathbb{N}))\vDash \forall P\, \theta(P)$ and for every $\LPA$-structure $\mathcal{M}$, if $(\mathcal{M}, \mathcal{P}(M))\vDash \forall P\,\theta(P)$, then $\mathcal{M}$ is isomorphic to $\mathbb{N}$.

\begin{theorem}[Wilkie \cite{wilkie:schemes}]
    Let $\theta(P)$ be a restricted $\LPA$ scheme template which is second-order categorical. Then there is a true $\LPA$-sentence $\varphi$ such that \[\theta[\LPA]+\varphi\vdash \PA.\]
\end{theorem}

Recall that a theory being solid is, in a loose sense,
a categoricity property.
Theorem \ref{main_thm_solid_inf-below_PA} shows that
a natural variant of Wilkie's result with
``solid theory'' replacing ``restricted second-order categorical scheme'' is false.

Our strategy for the proof of Theorem \ref{main_thm_solid_inf-below_PA}
is similar to the one used to prove
Theorem \ref{main_thm_solid_nointerpret_PA}
in the previous subsection. 
We will define a retract-disjoint family of theories $\tnf{n}$ with the following additional property: for every $n$, $\tnf{n}$ is consistent with $\Thr{\Pi_{n-1}}{\mathbb{N}} + \neg \IS{n+1}$. 
We observe that the theories $\tnd{n}$ do not have this consistency property because they all imply
the false $\Sigma_1$ statement $\neg \Con{\PA}$
(additionally, each $\tnd{n}$ implies a $\Sigma_1$ statement 
inconsistent with $\PA$, namely $\exists x\, \xi_n(x)$ where $\xi_n$ is the flexible formula from Section \ref{sec:proper-solid}, which is inconsistent with $\PA$). 
Our way of making the family $\{\tnf{n}\}_{n\in\omega}$ retract-disjoint will also be different from
the one in Section \ref{subsec:solid-non-interpret}. (Plausibly, an adaptation of the previous method could also be used here. However we find the current mode of presentation more convenient and it also allows us to present another tool for building retract-disjoint families.) Finally, the target theory $TF_n$ is constructed as an infinite definition by cases using the theories $\tnf{n}$.

\paragraph{Construction of $TF_n$.}
As in the case of $\tn{n}$, we first describe the 
construction of the ``intended model'' of $\tnf{n}$ and then extract its relevant properties in the form of axioms. 

Let $\zeta_n$ be a $\Sigma_{n+1}$ formula that is $(\Pi_n,\Sigma_1)$-flexible
over $\IS{n+1}$ and witnesses Theorem~\ref{tw_flex_main}.
The construction starts in the standard model of $\CT^{n}[\PA]$, 
i.e. the $n$-th element of the sequence defined recursively by:
\begin{align*}
    \mathcal{CT}_0&:= \mathbb{N}\\
    \mathcal{CT}_{j+1}&:= (\mathbb{N}, \Th{\mathcal{CT}_0},\ldots, \Th{\mathcal{CT}_j}).
\end{align*} In $\mathcal{CT}_n$, we perform a procedure of finding a Henkin structure $\mathcal{H}$ and its pointwise $\Sigma_{n+1}$-definable substructure $\KK$ similar to the one described by the interpretation $\hn{n}$ from Section \ref{sec:proper-solid}, 
except that now instead of the theory from (\ref{eq:theory-1}) we start with
\begin{equation}\label{eq:theory-for-KK-with-iterated-CT}
    \IS{n+1}\,\cup\,     \{\zeta_{n}(\num{k}), \neg\zeta_{n}(\num{\ell})\colon\, P_{n}(k), \neg P_{n}(\ell)\}\,\cup\, \{\sigma\in \LPA\colon\, \Sat{\Pi_{n}}(\ulcorner\sigma\urcorner)\}
    \text{.}
\end{equation}
Here, $P_n$ is the highest-level truth predicate of $\mathcal{CT}_n$. Hence the second component of the theory in (\ref{eq:theory-for-KK-with-iterated-CT}) ensures that $\zeta_{n}$ encodes the theory of $\mathcal{CT}_{n-1}$, and the third one -- that $\mathcal{H}$
satisfies $\Thr{\Pi_n}{\mathbb{N}}$.
The theory in (\ref{eq:theory-for-KK-with-iterated-CT})
is consistent by the flexibility properties of $\zeta_n$.

We use the notation $\hn{n}', \gn{n}'$ 
for analogues of the interpretations $\hn{n}, \gn{n}$ from Section \ref{sec:proper-solid} adapted to the current setting.
So, $\hn{n}'$ describes the
procedure of constructing the model $\KK_n:=\mathcal{K}^{n+1}(\mathcal{H})$,
where $\mathcal{H}$ is the model of the theory obtained by taking the leftmost completion of the henkinization of the theory in {(\ref{eq:theory-for-KK-with-iterated-CT})}.
As in the case of our previous constructions, we have a canonical definition $\delta_n$ that isolates the standard cut in $\KK_n$. 
The interpretation $\gn{n}'$ of $\mathcal{CT}_n$ in $\KK_n$ has
the universe defined by $\delta_n$, the arithmetical operations given by restricting $+$ and $\times$ to $\delta_n$,
and the $n$-th truth predicate $P_n$ given by $\zeta_n$. 
(Note that $\zeta_n$ is $\Sigma_{n+1}$, so its 
truth values on elements of $\KK_n$ are the same 
in $\KK_n$ as in $\mathcal{H}$ by partial elementarity.)
The predicates $P_1, \ldots, P_{n-1}$ are determined by $P_n$.

Now, 
$\tnf{n}$ is defined as $\tn{n}$ is Section {\ref{sec:proper-solid}}, except that 
$\hn{n}, \gn{n}$ are changed to $\hn{n}', \gn{n}'$, 
and $``\gn{n}\vDash \CT[\PA]"$ in (iii) is replaced by

\begin{itemize}
    \item[(iii)$'$] $\gn{n}'\vDash \CT^n[\PA]$.
\end{itemize}

Finally, we let $TF_n$ be ${\IS{n}} + {\exp} + {\bigoplus_{k \ge n}(\tnf{k}|{\IS{k}}\wedge {\neg \IS{k+1}})}$.

\paragraph{Verification of the properties of $TF_n$.}
By the above construction of $\tnf{n}$ and the fact that for each $n$ and each model $\MM\vDash \IS{n}+\exp$, we have $\mathcal{K}^{n+1}(\MM)\vDash\lnot\BS{n+1}$ and $\mathcal{K}^{n+1}(\MM)\preceq_{n+1}\MM$, we obtain: 

\begin{lemma}\label{lem:tnf-consistent-with-true-and-false}
    For each $n\geq 1$, $\tnf{n}$ is consistent with $\Thr{\Pi_{n}}{\mathbb{N}}+\neg\BS{n+1}$.
\end{lemma}

The proof of the next lemma is also fully very similar to its 
analogue for $\tn{n}$ (Proposition \ref{prop:biint_truth}), 
so we only sketch it.
\begin{lemma}\label{lem:tnf-k-biint-ct-k}
    For each $n\geq 1$, $\tnf{n}$ is bi-interpretable with $\CT^{n}[\PA]$.
\end{lemma}
\begin{proof}
    The main change compared to the proof of Proposition \ref{prop:biint_truth} is that instead of a formalization of the arguments from Section \ref{sec:proper-solid} in $\CT[\PA]$, we use a formalization of the proof of Lemma \ref{lem:tnf-consistent-with-true-and-false}
    in $\CT^n[\PA]$. The verification that $\zeta_n$ 
    has the required $(\Pi_n,\Sigma_1)$-flexibility
    property over $\IS{n+1}$ can be carried out in $\CT^{n}[\PA]$ thanks to
    Theorem \ref{tw_flex_main}
    and the fact that $\CT^{n}[\PA]\vdash \Pi_{n}\mhyphen\mathrm{Con}(\IS{n+1})$. 
\end{proof}

\begin{corollary}\label{cor_tnf_solid}
    For each $n\geq 1$, the theory $\tnf{n}$ is solid.
\end{corollary}
\begin{proof}
By Lemma \ref{lem:tnf-k-biint-ct-k},
Corollary \ref{cor:solid_ctn},  
and Corollary \ref{cor_biint_pres_solid}.
\end{proof}

\begin{corollary}\label{cor:pa-tnf-retract-disjoint}
    The family $\{\PA\} \cup \{\tnf{n}\}_{n\in\omega}$ is retract-disjoint.
\end{corollary}
\begin{proof}
    By Lemma \ref{lem:tnf-k-biint-ct-k}
    and Corollary \ref{cor_biint_pres_retract}.
\end{proof}

\begin{proof}[Proof of Theorem \ref{main_thm_solid_inf-below_PA}]

Clearly, $TF_n$ is an r.e.~subtheory
of $\PA$, and it contains $\IS{n}$ by definition.

The fact that $TF_n + \Thr{\Pi_k}{\N}$ does not imply
$\PA$ for any $k$ follows immediately from 
Lemma \ref{lem:tnf-consistent-with-true-and-false}.

To prove solidity of $TF_n$, we invoke Proposition \ref{pres_solid_disj_sum} with $U_k := \tnf{k}$, 
$V:= \IS{n}+\exp$,
and $\varphi_k:= {\IS{k}}\wedge\neg{\IS{k+1}}$
(where $\mathrm{I}\Delta_0$ is replaced by 
a sufficiently large finite fragment, as before). 
By Corollary \ref{cor_tnf_solid},
each $U_k$ is solid, and $V \cup \{\neg\varphi_k\colon\, k\in\omega\}$ is solid because it is equivalent to $\PA$.
The sentences $\varphi_k$ are pairwise inconsistent. 
Finally, the family ${\PA} \cup {\{U_k\}_{k\in\omega}}$ is retract-disjoint
by Corollary \ref{cor:pa-tnf-retract-disjoint}.
\end{proof}

\begin{remark}
We can combine Theorems \ref{main_thm_solid_nointerpret_PA}
and \ref{main_thm_solid_inf-below_PA} in the following sense. 
If we let $U_{2n+2} := \tnf{2n+2}$ and $U_{2n+1}:= \tnd{2n+1}$, 
one can check that for every $n \geq 1$ the theory 
\[{\IS{n}} + {\exp} + \bigoplus_{k \ge n}(U_k |{\IS{k}}\wedge{\neg\IS{k+1}})\] is a proper solid subtheory of $\PA$ extending $\IS{n}$ that is both unable to interpret $\PA$ and
``infinitely below'' $\PA$ 
in the sense of Theorem \ref{main_thm_solid_inf-below_PA}. 
\end{remark}

\section{Separation theorems}\label{sec:separation}

We now focus on separating the categoricity-like properties 
considered in this paper. In particular, we obtain a separation of tightness from neatness, and a nontrivial
(e.g., not based on a complete theory) example separating neatness from semantical tighness and solidity. 

In most of our constructions, we exploit the fact that an actually existing isomorphism that would be needed to witness one of the properties we study or the failure of another is somehow difficult to express. In some cases, the isomorphism cannot be defined internally in a structure, in others, its definition requires parameters from the structure.  

Some of our nontrivial examples take the form 
${(T_1|\neg \psi)}\bigoplus{(T_2|\psi)}$ for a sentence $\psi$. To avoid using such cumbersome notation, we will write
${T_1} \oplus_\psi {T_2}$ in its stead.

\subsection{Tame separators}\label{subsec:tame}
We first discuss two simple examples of separations between the syntactic and the semantical notions, based on the fact that any complete theory has to be neat. The theories in these examples have the virtue of being r.e., but they do not interpret any arithmetic at all; in particular, they are not sequential.

\begin{prop}
$\DLO$ is neat and therefore tight, but it is not semantically tight.
\end{prop}
\begin{proof}
Note first that $\DLO$ is clearly neat because it is complete. To show that it is not semantically tight, let $\MM = \tuple{\mathbb{Q},\le} + \tuple{\mathbb{R},\le}$ and let $\NN = \tuple{\mathbb{R}, \ge} +\tuple{\mathbb{Q},\ge}$, where $+$ stands for the disjoint sum of linear orders. Since the order on $\NN$ is just the inverse of the order on $\MM$, the two structures are clearly bi-interpretable. 
Yet, $\MM$ is not isomorphic to $\NN$, so $\DLO$ fails to be semantically tight.
\end{proof}

\begin{remark}
As another example in this spirit that additionally illustrates the role of multi-dimensional interpretations and the subtleties involved in defining semantical tightness, consider the theory of infinite sets in the empty language. Just like $\DLO$, this theory is complete and hence trivially neat.

We show that this theory is not semantically tight if one is willing to allow multi-dimensional interpretations. Let
$\MM = \mathbb{N},$ $\NN = \mathbb{N}\setminus\{0\}$, $\MM^* = \{\langle n,n\rangle\colon\, n\geq 1\} \cup \{ \langle 1,2\rangle\}$, $\NN^* = \{\langle n,n\rangle\colon\, n\geq 1\}$. Then we have
three interpretations witnessing $\MM \rhd \NN \rhd \MM^* \rhd \NN^*$. The interpretations are defined in the obvious way, though the one of $\MM^*$ in $\NN$ is two-dimensional, and they all use parameters, namely $0$; $1,2$; and $\tuple{1,2}$, respectively, in order to exclude the appropriate elements from the domain. There is an $\MM$-definable isomorphism between $\MM$ and $\MM^*$ (map $0$ to $\tuple{1,2}$, and any other $n$ to $\tuple{n,n}$), as well as an $\NN$-definable isomorphism between $\NN$ and $\NN^*$ (map $n$ to $\tuple{n,n}$). Thus, we get a bi-interpretation between $\MM$ and $\NN$. But there is no $\MM$-definable bijection between $\MM$ and $\NN$, because, by quantifier elimination, any definable injection from $\MM$ to $\MM$ has the following form: an arbitrary permutation of a finite set (the set of parameters involved in the definition), and the identity on all other elements. 

The argument from the previous paragraph breaks down if in the definition of semantical tightness we only require an isomorphism instead of an $\MM$-definable isomorphism: indeed, any two bi-interpretable infinite sets must have the same cardinality and thus be isomorphic. It also breaks down if we only allow one-dimensional interpretations as in the rest of this paper.
\end{remark}

\subsection{Tight but not neat}\label{subsec:tight-not-neat}

This subsection is devoted to the proof of a less trivial separation, between the two syntactic notions of tightness and neatness. In fact, we prove that there are arbitrarily strong subtheories of $\PA$ that are tight but not neat. We do not know whether the theories we construct are semantically tight.

\begin{theorem}\label{thm:tight-not-neat}
    For every $n \ge 1$, there exists an r.e.~subtheory
    of $\PA$ that contains $\BS{n}+\exp$ but not $\IS{n}$
    and is tight but not neat.
\end{theorem}

The overall structure of the argument is similar to that
of Section \ref{sec:proper-solid}, though the role of $\CT[\PA]$ is played by $\PA$, and pointwise definable
models of ${\IS{n}} + {\neg \BS{n+1}}$ are replaced by models
of ${\BS{n}} + {\neg \IS{n}}$. 

Fix $n \ge 1$. In analogy to the theories $T_n$ and $\tn{n}$
from the proof of Theorem \ref{thm:proper-solid} in Section \ref{sec:proper-solid}, we will use the symbol $S_n$ to denote the theory that will eventually witness Theorem \ref{thm:tight-not-neat}, and we will define an auxiliary theory $\sn{n}$.

In Section \ref{sec:proper-solid}, we had an interpretation
$\hn{n}$ of a model of ${\IS{n}} + {\neg \BS{n+1}}$ in $(\N, \Th{\N})$, and an inverse interpretation $\gn{n}$ of $(\N, \Th{\N})$ in that model. 
The theory $\tn{n}$ axiomatized many features of our particular model of ${\IS{n}} + {\neg \BS{n+1}}$, 
and $T_n$ said that we are either in a model of $\tn{n}$ or
in one of $\PA$. 
This time, $\hn{n}$ will be replaced by a parameter-free interpretation $\ihn{n}$ of a model of ${\BS{n}} + {\neg \IS{n}}$ in $\N$, and again there will be an inverse interpretation of $\N$ in our model. In fact, we will reuse the name $\gn{n}$ for that inverse interpretation, because it will essentially do the same job as before -- pick out the smallest definable cut of our model -- except that there will be no need to define the truth predicate. 
As before, $\sn{n}$ will axiomatize some properties of our model, 
and $S_n$ will say that we are either in a model of $\sn{n}$ or in one of $\PA$.

The interpretation $\ihn{n}$ describes the following process, as carried out in $\N$ (and routinely formalizable in $\PA$):

\begin{itemize}
    \item Consider a canonically defined binary tree whose
    paths correspond to complete consistent henkinized extensions of the theory ${\IS{n}} + {\neg \Con{\IS{n}}}$. \item Take the Henkin model $\mathcal{H}$ given by
    the leftmost path through that tree.
    \item Take the initial segment of $\mathcal{H}$
    generated by the first $\N$ iterations of the witness-bounding function for the universal $\Sigma_{n-1}$ formula (see discussion at the end of Section \ref{sec:prelim}) on the smallest proof of inconsistency in $\IS{n}$. For $n=1$, instead of the witness-bounding function consider the first $\N$ iterations of $\exp$.
\end{itemize}

\begin{figure}[htbp]

\begin{tikzpicture}
\draw[very thick, ] (3,0) -- (4.65,0);

\draw[very thick] (4.65,0) -- (10.8,0);

\draw[very thick, dashed] (10.8,0) -- (12.95,0);

\draw[very thick, ] (3, -0.1) -- (3, 0.1);
\draw[very thick, ] (3.3, -0.1) -- (3.3, 0.1);
\draw[very thick, ] (3.6, -0.1) -- (3.6, 0.1);

\draw (3, -0.4) node {\small{$\mathsf 0$}};
\draw (3.3, -0.4) node {\small{$\mathsf 1$}};
\draw (3.6, -0.4) node {\small{$\mathsf 2$}};
\draw (4.1, -0.4) node {\small{$\ldots$}};

\draw[very thick, ] (4.55, -0.25) .. controls (4.7,-0.15)  and (4.7,0.15) ..  (4.55, 0.25);

\draw (4.8, 0.5) node {$\mathbb{N}$};

\draw[very thick, ] (6.3, -0.1) -- (6.3, 0.1);

\draw[very thick, ] (10.65, -0.25) .. controls (10.8,-0.15)  and (10.8,0.15) ..  (10.65, 0.25);

\draw (11, 0.5) node {$\JJ$};

\draw[very thick] (12.85, -0.25) .. controls (13,-0.15)  and (13,0.15) ..  (12.85, 0.25);

\draw (13.2, 0.5) node {$\mathcal{H}$};

\end{tikzpicture}

\caption{Construction of the model $\N^{\ihn{n}}$ of $\sn{n}$. The solid horizontal line represents $\N^{\ihn{n}}$, which is a nonstandard $\Sigma_{n-1}$-elementary
initial segment $\JJ$ of the Henkin structure $\mathcal{H}$.
The dashed horizontal line represents the rest of $\mathcal{H}$.}

\label{fig:structure-j}

\end{figure}

This process is illustrated in Figure \ref{fig:structure-j}. As discussed at the end of Section \ref{sec:prelim},
it produces a $\Sigma_{n-1}$-elementary cut 
$\JJ$ of $\mathcal{H}$ that is necessarily a proper cut,
because $\mathcal{H}$ is nonstandard, and
in a model of $\IS{n}$ the witness-bounding function for 
a $\Sigma_{n-1}$ formula can be iterated an arbitrary number of times (and so can $\exp$ in a model of $\IS{1}$).
Thus, $\JJ$ is a model of ${\BS{n}}$.
Moreover, $\JJ$ is a model of $\neg \IS{n}$, because 
it is a nonstandard structure
in which the standard cut $\N$ is $\Sigma_n$-definable, say 
by the formula $\delta_n(x)$ expressing ``there exists an inconsistency proof for $\IS{n}$, and on that proof the witness-bounding function for the universal $\Sigma_{n-1}$ formula can be iterated $x$ times''. 

Thus, $\N$ is the smallest definable cut of $\JJ$, and it can be interpreted in $\JJ$ by the interpretation $\gn{n}$ in which the domain is defined by $\delta_n$
and the arithmetical operations are unchanged. 
As in the proof of Theorem \ref{thm:proper-solid}, there is an $\N$-definable isomorphism $j_n$ between $\ihn{n}\!\gn{n}$ and the identity interpretation of $\N$ in itself, namely the map from Lemma \ref{ultimate_dedekind}:
take $x \in \N$ to the $x$-th smallest element of $\JJ$.
Each of $\gn{n}, \ihn{n}, j_n$ is definable without parameters.

However, we no longer have a $\JJ$-definable isomorphism between  $\gn{n}\!\ihn{n}$ and the identity interpretation of $\JJ$ in itself. The reason is that $\JJ$ is a model
of ${\BS{n}} + {\exp} + {\neg\IS{n}}$, and the domain
of $\gn{n}$ is a proper cut in it. It is known that 
models of ${\BS{n}} + {\exp} + {\neg\IS{n}}$ cannot have 
a definable injective multifunction into a proper initial segment:
\begin{theorem}[\cite{kaye:theory-kappa-like}]\label{thm:gwphp}
    Let the cardinality scheme $\mathop{CARD}$ say that  
    no formula defines an injective multifunction from
    the universe into $[0,x]$ for any number $x$.
    Then, for each $n \ge 1$, it holds that ${\BS{n}} + {\exp} + {\neg\IS{n}} \vdash CARD$.
\end{theorem}
On the other hand, the structure produced by $\gn{n}\!\ihn{n}$
is in fact isomorphic to $\JJ$, even though $\JJ$ does not see the isomorphism. In particular, the two structures
are elementarily equivalent.

We let $\sn{n}$ axiomatize the properties of $\JJ$ discussed above. The axioms of $\sn{n}$ are:
\begin{enumerate}[(i)]
    \item\label{it:s0n-1} ${\BS{n}} + {\exp} + {\neg\IS{n}}$,
    \item\label{it:s0n-2} ``$\delta_n$ defines a cut which is the smallest definable cut'',
    \item\label{it:s0n-4} $\psi^{\gn{n}\!\ihn{n}} \leftrightarrow \psi$, for each sentence $\psi$,
    \item\label{it:s0n-5} $\gn{n}\vDash ``j_n \colon \mathsf{id} \to \ihn{n}\!\gn{n}$ is an isomorphism$"$.
\end{enumerate}

We let $S_n$ be
${\sn{n}} \oplus_{\IS{n}} {\PA}$. 
Note that it follows from axioms \ref{it:s0n-2} of $\sn{n}$ 
that $\gn{n}$ is an
interpretation of $\PA$ in $\sn{n}$.

\begin{lemma}\label{lem:sn-proper-subtheory}
 The theory $S_n$ contains $\BS{n} + \exp$ but not $\IS{n}$. Thus, it is a proper subtheory of $\PA$.  
\end{lemma}
\begin{proof}
The argument is analogous to the one for Lemma \ref{lem:tn-proper-subtheory}: by the construction of the model $\JJ$ described above, $\sn{n}$ is consistent and contains $\BS{n} + \exp$ but contradicts $\IS{n}$.
\end{proof}

\begin{lemma}\label{lem:sn-not-neat}
     The theory $S_n$ is not neat.
\end{lemma}
\begin{proof}
Consider $U = \PA$ and $V = \sn{n}$. 
Note that both these theories extend $S_n$. 
Moreover, 
$\ihn{n}$ is an interpretation of $\sn{n}$ in $\PA$, 
and $\ihn{n}\!\gn{n}$ is an interpretation of $\PA$ in $\PA$. By design, the latter interpretation is $\PA$-provably isomorphic to the identity interpretation: 
$\ihn{n}$ is the shortest
initial segment of $\mathcal{H}$ which contains all the finite (in the sense of the ground model) iterations of the 
witness-bounding function for the universal $\Sigma_{n-1}$ formula on the smallest witness to $\neg \Con{\IS{n}}$, 
and $\gn{n}$ isolates precisely the set of those numbers $a$ for which that function
can be iterated $a$-times. 
Thus, $\PA$ is a retract 
of $\sn{n}$.
Clearly, however, $\sn{n}$ is not a subtheory of $\PA$.
\end{proof}

\begin{lemma}\label{lem:s0n-neat}
 The theory $\sn{n}$ is neat.  
\end{lemma}
\begin{proof}
It is enough to prove that if $\MM_1 \rhd \MM_2 \rhd \MM_3$ are models of $\sn{n}$ and there is an $\MM_1$-definable isomorphism from $\MM_1$ onto $\MM_3$, then $\MM_1 \equiv \MM_2$. So, let $\MM_1, \MM_2, \MM_3$ be as described.

For each $i \in \{1,2,3\}$, let $\NN_i$ be the model
of $\PA$ obtained by applying the interpretation $\gn{n}$
in $\MM_i$, and let $\MM'_i$ be the model of $\sn{n}$
obtained by applying $\ihn{n}$ in $\NN_i$. 
See Figure \ref{fig:s0n-neat}.
Note that
the domain of each $\NN_i$ is the smallest definable
cut of $\MM_i$ and that there is an $\MM_1$-definable isomorphism from $\NN_1$ onto $\NN_3$ (induced by the $\MM_1$-definable isomorphism between $\MM_1$ and $\MM_3$).

\begin{figure}[htbp]\label{fig:s0n-neat}

\begin{center}
\begin{tikzcd}[column sep = 4.5em, row sep = 4.5 em ]
 \MM_1 
 \arrow[r, symbol = \rhd]
 \arrow[rr, bend left = 30]
 \arrow[d, symbol = \rhd] 
 \arrow[dd, bend right = 40, "\rotatebox{90}{$\equiv$}"] & 
 \MM_2 
 \arrow[r, symbol = \rhd]
 \arrow[d, symbol = \rhd] 
 \arrow[dd, bend right = 40, "\rotatebox{90}{$\equiv$}"] & 
 \MM_3 
 \arrow[d, symbol = \rhd]
 \\
 \NN_1
 \arrow[r, bend right = 40, dashed]
 \arrow[rr, bend left = 20, dashed]
 \arrow[d, symbol = \rhd] & 
 \NN_2
 \arrow[d, symbol = \rhd] & 
 \NN_3
 \arrow[d, symbol = \rhd] \\
 \MM'_1 
 \arrow[r, bend right = 40, dashed] & 
 \MM'_2 &
 \MM'_3
\end{tikzcd}
\end{center}

\caption{The proof of Lemma \ref{lem:s0n-neat}. 
The horizontal solid arrow represents an isomorphism given directly by the assumptions about $\MM_1, \MM_2, \MM_3$, and the dashed arrows represent isomorphisms shown to exist during the argument. The vertical solid arrows stand for elementary equivalences, which together with the isomorphisms
let us conclude that $\MM_1$ is elementarily equivalent to $\MM_2$.}

\end{figure}

By axioms \ref{it:s0n-2} of $\sn{n}$ and Lemma \ref{lem:indfact} we have an ($\MM_1$-definable) isomorphism between $\NN_1$ and $\NN_2$. This clearly gives rise to an ($\MM_1$-definable) isomorphism between $\MM'_1$ and $\MM'_2$.

By axioms \ref{it:s0n-4} of $\sn{n}$, we know that
$\MM_1 \equiv \MM'_1$ and $\MM_2 \equiv \MM'_2$.
Since $\MM'_1$ and $\MM'_2$ are isomorphic,
this implies that $\MM_1 \equiv \MM_2$.
\end{proof}

\begin{proof}[Proof of Theorem \ref{thm:tight-not-neat}]

We have already shown in Lemmas \ref{lem:sn-proper-subtheory}
and \ref{lem:sn-not-neat} that $S_n$ is a subtheory
of $\PA$ containing $\BS{n} + \exp$ but not $\IS{n}$,
and that $S_n$ is not neat. Clearly, $S_n$ is an r.e.~theory,
so it remains to prove that it is tight.

It is enough to show that if $\MM_1$ is a model of $S_n$ and $(\mathsf{M}_2,\mathsf{M}_1)$ is a bi-interpretation in $\MM_1$,
then $\MM_1 \equiv \MM_1^{\mathsf{M}_2}$. Put $\MM_2 = \MM_1^{\mathsf{M}_2}.$ 

By the definition of $S_n$,
each $\MM_i$ satisfies either $\PA$ or $\sn{n}$. 
If $\MM_1$ and $\MM_2$ both satisfy $\PA$ or both satisfy $\sn{n}$, then $\MM_1 \equiv \MM_2$ follows from
the solidity of $\PA$ or the proof of Lemma \ref{lem:s0n-neat}, respectively.

The remaining case is that exactly one of $\MM_1, \MM_2$ satisfies $\PA$. Assume w.l.o.g.~that $\MM_1 \vDash \sn{n}$
and $\MM_2 \vDash \PA$.
We will show that this leads to a contradiction,
which will complete the proof of the theorem.

Let $\MM_3 = \MM_2^{\mathsf{M}_1}$ be the structure interpreted in $\MM_2$
which is $\MM_1$-definably isomorphic to $\MM_1$.
That isomorphism between $\MM_1$ and $\MM_3$,
say $f_1$, may also be viewed as an
$\MM_1$-definable injective multifunction $\widehat{f_1}$
from $\MM_1$ into $\MM_2$.

Since $\MM_2$ satisfies all of $\PA$ and interprets the arithmetical structure on the initial segment $(\delta_n)^{\MM_3}$ of $\MM_3$, 
Lemma \ref{ultimate_dedekind} tells us that 
there is an $\MM_2$-definable (hence $\MM_1$-definable) embedding $f_2$ from $\MM_2$ into $(\delta_n)^{\MM_3}$. However, $(f_1)^{-1}$ restricted to $(\delta_n)^{\MM_3}$ is
an $\MM_1$-definable isomorphism between $(\delta_n)^{\MM_3}$ and $(\delta_n)^{\MM_1}$. So, $(f_1)^{-1} \circ f_2 \circ \widehat{f_1}$ is an $\MM_1$-definable injective multifunction from $\MM_1$ into $(\delta_n)^{\MM_1}$, where $(\delta_n)^{\MM_1}$ is a proper initial segment of $\MM_1$.
This contradicts Theorem \ref{thm:gwphp}.
\end{proof}

\subsection{Tight but neither neat nor semantically tight}\label{subsec:pam}
In this subsection, we aim to define a theory that separates tightness from neatness and for which we also know that it is not semantically tight. We are able to find a sequential theory of this kind, but we do not know whether such theories can have arbitrary arithmetical strength.

Below, $\mathbb{Z}[X]$ denotes the ring of polynomials over $\mathbb{Z}$, which we see as a model for $\LPA$, 
with the ordering determined by making $X$ greater than all the integers. 
We write $(\mathbb{Z}[X])_{\ge 0}$ for the nonnegative part
of $\mathbb{Z}[X]$, which is a model of $\PAm$.

\begin{lemma}[
\cite{elv:completions}]\label{lem:visser_pam}
The structure $((\mathbb{Z}[X])_{\ge 0}, X)$ is 
parameter-free bi-interpretable  with $\mathbb{N}$. 
As a consequence, $(\mathbb{Z}[X])_{\ge 0}$ 
is bi-interpretable with $\mathbb{N}$; but it is not  parameter-free bi-interpretable with $\mathbb{N}$.
\end{lemma}
\begin{proof}
$(\mathbb{Z}[X])_{\ge 0}$ is clearly a computable structure, hence it is arithmetically definable, 
and we can fix a (parameter-free) interpretation $\mathsf{Z}$ 
of a copy of $((\mathbb{Z}[X])_{\ge 0}, X)$ in the standard model $\N$. 
To be more specific, we represent polynomials from
$((\mathbb{Z}[X])_{\ge 0}, X)$ as (natural numbers coding) finite sequences of integers. 

This provides us with one interpretation needed for the bi-interpretability. To define the other one, observe that $\mathbb{N}$, the standard cut, is parameter-free definable in $(\mathbb{Z}[X])_{\ge 0}$.
Namely, let $\delta(x)$ say that all numbers smaller or equal to $x$ are either even or odd.
Since elements of the form $X+k$, where $k\in\mathbb{Z}$,
are downwards cofinal over $\mathbb{N}$ in $(\mathbb{Z}[X])_{\ge 0}$, and no such element is divisible by $2$, only the standard integers satisfy $\delta(x)$ in $(\mathbb{Z}[X])_{\ge 0}$.
This 
gives rise to a parameter-free interpretation of $\mathbb{N}$ in $((\mathbb{Z}[X])_{\ge 0}, X)$ 
(in fact, in $(\mathbb{Z}[X])_{\ge 0}$), 
which we will denote by $\mathsf{N}$.

Now, we show that there is a parameter-free definable isomorphism between the identity interpretation
on $((\mathbb{Z}[X])_{\ge 0}, X)$ and $\mathsf{NZ}$. This crucially depends on the fact that $\PAm$ is sequential,
so internally in $\PAm$ we have a notion of finite sequence
that is well-behaved for sequences of standard 
length. Consequently, by mimicking the usual recursive definitions, we can define such notions as (standard) finite sums and finite products. In particular, there is a $((\mathbb{Z}[X])_{\ge 0}, X)$-definable function with domain $\mathsf{N}\mathsf{Z}$ which, given a coded finite sequence $a=(a_0,\ldots, a_n)\in \mathsf{NZ}$, 
returns $a_nX^n + a_{n-1}X^{n-1}+\ldots + a_1X + a_0$; the definition of the function needs no parameters beyond $X$ itself, 
which is named by a constant in $((\mathbb{Z}[X])_{\ge 0}, X)$. The inverse of this function is our $((\mathbb{Z}[X])_{\ge 0}, X)$-definable isomorphism between $\mathsf{id}$ and $\mathsf{N}\mathsf{Z}$. The $\mathbb{N}$-definable isomorphism $j$ between $\mathsf{id}$ and $\mathsf{Z}\mathsf{N}$ is the usual map from 
Section \ref{subsec:formalized-categoricity}.

Thus, $\N$ is parameter-free bi-interpretable with $((\mathbb{Z}[X])_{\ge 0}, X)$, 
which means that it is also bi-interpretable with $(\mathbb{Z}[X])_{\ge 0}$. To prove that the bi-interpretability with $(\mathbb{Z}[X])_{\ge 0}$
requires parameters, it is enough to observe that $(\mathbb{Z}[X])_{\ge 0}$ carries a non-trivial automorphism:
namely, the semiring homomorphism generated by $X\mapsto X+1$, whose inverse is given by $X\mapsto X-1$.
On the other hand, $\N$ has no automorphisms other than the 
identity, and as mentioned in Section \ref{sec:prelim} (see the final remark in the paragraph ``Retractions and bi-interpretations of structures''),
structures that are parameter-free bi-interpretable have isomorphic automorphism groups. 
\end{proof}

In the remainder of this subsection, we will continue
to use the notation $\delta(x)$, $j$, $\mathsf{N}$, $\mathsf{Z}$ for the formulas resp.~interpretations thus denoted in the proof of Lemma 
\ref{lem:visser_pam}. 
We let $\iota_z$ stand for 
the $(\mathbb{Z}[X])_{\ge 0}$-definable map that, given
a parameter $z$ and a sequence $(a_0,\ldots,a_n)$ in $\mathsf{NZ}$, 
outputs $a_nz^n + a_{n-1}z^{n-1}+\ldots + a_1z + a_0$.
Thus, $(\iota_X)^{-1}$ is an isomorphism
between $((\mathbb{Z}[X])_{\ge 0}, X)$ and $((\mathbb{Z}[X])_{\ge 0}, X)^{\mathsf{NZ}}$.

We also let $h_z$ be the $(\mathbb{Z}[X])_{\ge 0}$-definable operation that maps the parameter $z$ to $z+1$ and extends to values of polynomials in $z$ in the obvious way. More precisely: $h_z$ takes $p\in(\mathbb{Z}[X])_{\ge 0}$ and searches for $a\in \mathsf{NZ}$ such that $p = \iota_z(a)$. If such an $a$ does not exist, the function is undefined. If it does, then $h_z(p):= \iota_{z+1}(a)$. 
Note that in general, $h_z$ is only a partial function -- for instance, the domain of $h_{X^2}$ is $(\mathbb{Z}[X^2])_{\ge 0}$ rather than all of $(\mathbb{Z}[X])_{\ge 0}$ --
but $h_X$ is an automorphism of $(\mathbb{Z}[X])_{\ge 0}$.

Let $U$ be the following theory, axiomatizing some properties of $(\mathbb{Z}[X])_{\ge 0}$ in the spirit of the theories $\tn{n}$ and $\sn{n}$ of Sections \ref{sec:proper-solid} resp.~\ref{subsec:tight-not-neat}:

\begin{enumerate}[(i)]
    \item\label{it:u-1} $\PA^-$,
    \item\label{it:u-2} ``$\delta$ defines a cut which is the smallest definable cut'',
    \item\label{it:u-3} ``there exists $x$ such that $h_x$ is a nontrivial automorphism of $\mathsf{id}$'',
    \item\label{it:u-4} ``there exists $x$ such that $(\iota_x)^{-1} \colon \mathsf{id} \to \mathsf{N}\mathsf{Z}$ is an isomorphism'',
    \item\label{it:u-5} $\mathsf{N }\vDash ``j \colon \mathsf{id} \to \mathsf{Z}\mathsf{N}$ is an isomorphism''.
\end{enumerate}

We observe that \ref{it:u-2} is an axiom scheme, while the
other axioms of $U$ are single statements.

Our goal is to prove the following theorem:
\begin{theorem}\label{thm:tight-not-neat-not-sem-tight}
    There is a sequential r.e.~subtheory of $\PA$ which is tight but is neither neat nor semantically tight.
\end{theorem}

The theory in question is 
${U} \oplus_{\IS{1}}{\PA}$. 
As usual, to prove a tightness-like property of the theory
(here, specifically tightness only), we need to show
the property for $U$ and to obtain a result that rules out some ``mixed cases'' of interpretations.

\begin{lemma}\label{lem:u-tight}
The theory $U$ is solid.
\end{lemma}

\begin{proof}
 Let $\MM_1 \rhd \MM_2 \rhd \MM_3$ be models of $U$
 such that there is an $\MM_1$-definable isomorphism from $\MM_1$ onto $\MM_3$. For each $i \in \{1,2,3\}$, consider  $\NN_i:= \MM_i^{\mathsf{N}}$ and $\MM_{i}^*:= \NN_i^{\mathsf{Z}}$. Since each $\MM_i$ is a model of $U$, it follows that $\MM_1, \MM_{2}, \MM_{3}$ and $\NN_1, \NN_{2}, \NN_{3}$ satisfy the assumptions of Lemma \ref{lem:indfact} for $m=0$, so $\NN_1$ is $\MM_1$-definably isomorphic to $\NN_{2}$. Hence, $\MM_1^*$ is $\MM_1$-definably isomorphic to $\MM_2^*$. 
 
 By axiom \ref{it:u-4} of $U$, each $\MM_i$ is $\MM_i$-definably, and thus also $\MM_1$-definably,
 isomorphic to $\MM_i^*$. Composing these 
 isomorphisms with the one between $\MM_1^*$ and $\MM_2^*$,
 we obtain an $\MM_1$-definable isomorphism between
 $\MM_1$ and $\MM_2$.
\end{proof}

\begin{remark}
Note that in the proof of Lemma \ref{lem:u-tight}, the isomorphism between $\MM_1$ and $\MM_2$ needed to witness 
solidity is defined using parameters from $\MM_1$ that might not be involved in defining the interpretations between $\MM_1$, $\MM_2$, $\MM_3$ and the isomorphism between
$\MM_1$ and $\MM_3$. In any case, we will only need 
the tightness of $U$ in the remainder of our argument.
\end{remark}

\begin{lemma}\label{lem:u-pa-not-biint}
If $\MM_1 \vDash \PA$ and $\MM_2 \vDash U$, then
$\MM_1$ and $\MM_2$ are not parameter-free bi-interpretable.
\end{lemma}
\begin{proof}
Suppose the contrary and let $(\mathsf{M}_2,\mathsf{M}_1)$ be a parameter-free bi-interpretation in $\MM_1\vDash \PA$ such that $\MM_1^{\mathsf{M}_2}\vDash U$. 
Let $\KK$ be the prime substructure of $\MM_1$. Then $\KK \preccurlyeq \MM_1$,
so $\mathsf{M}_1$ and $\mathsf{M}_2$ witness that $\mathcal{K}$ is parameter-free bi-interpretable with a model of $U$. This cannot be the case, because the automorphism group of $\mathcal{K}$ is trivial, while every model of $U$, specifically of axiom \ref{it:u-3}, carries a nontrivial automorphism.
\end{proof}

\begin{proof}[Proof of Theorem \ref{thm:tight-not-neat-not-sem-tight}]
Clearly, ${U} \oplus_{\IS{1}}{\PA}$ is an r.e.~subtheory of $\PA$, and it is sequential because it implies $\PA^-$.

We now show that it is tight. Let $V_1$, $V_2$ be bi-interpretable extensions of ${U} \oplus_{\IS{1}}{\PA}$, and let $\mathsf{V}_2$, $\mathsf{V}_1$ be interpretations witnessing the bi-interpretability. 
We claim that applying $\mathsf{V}_{3-i}$ in a model of 
$V_i + \PA$ gives rise to a model of $V_{3-i} + \PA$,
and analogously for models of $U$ instead of $\PA$.
To prove the claim, note that if $\MM \vDash V_i + \PA$, then by the choice of $\mathsf{V}_1$, $\mathsf{V}_2$
the structures $\MM$ and $\MM^{\mathsf{V}_{3-i}}$ are
in fact parameter-free bi-interpretable. 
So, by Lemma \ref{lem:u-pa-not-biint}, it must
be the case that $\MM^{\mathsf{V}_{3-i}} \vDash \PA$.
The proof for models of $U$ is similar.

By the claim, $V_1 + \PA$ is bi-interpretable with $V_2 + \PA$, and $V_1 + U$ is bi-interpretable with $V_2 + U$. 
Thus, the solidity of $\PA$ implies that $V_1 + \PA \equiv V_2 + \PA$, and Lemma \ref{lem:u-tight} implies that $V_1 + U \equiv V_2 + U$. So, $V_1 \equiv V_2$, proving tightness
of ${U} \oplus_{\IS{1}}{\PA}$.

The lack of semantical tightness is witnessed by 
the structures $\N$ and $(\mathbb{Z}[X])_{\ge 0}$,
which are both models of ${U} \oplus_{\IS{1}}{\PA}$ and are bi-interpretable by Lemma \ref{lem:visser_pam}
but are not isomorphic.
The lack of neatness is witnessed by the theories
$\Th{\N}$ and $\Th{(\mathbb{Z}[X])_{\ge 0}}$.
The former is a retract of the latter, because not
only the interpretations $\mathsf{N}$ and $\mathsf{Z}$,
but also the isomorphism $j$ between $\N$
and $(\N)^{\mathsf{Z}\mathsf{N}}$ are defined
without parameters.
\end{proof}

\subsection{Neat but not semantically tight}\label{subsec:neat-not-sem-tight}
Our final separation result takes the following form.

\begin{theorem}\label{thm:neat-not-sem-tight}
There is a sequential r.e.~theory which is neat but not semantically tight.
\end{theorem}

This time, the theory in question will be a strengthening of
$\PA$ formulated in the language extending $\LPA$ by a fresh constant symbol $c$. It will take the form ${\PA} + {p(c)}$,
where $p(x)$ is a partial type with some particular properties.

\begin{lemma}\label{lem:type}
There exists a computable partial type $p$ over $\PA$ such that:
\begin{enumerate}[(i)]
    \item\label{lem-type-i} for all $\MM, \NN \vDash \PA$ and elements $a \in \MM$, $b \in \NN$ realizing $p$,
    \begin{equation*}
    \text{if $\MM \equiv \NN$, then $(\MM,a) \equiv (\NN,b)$;}
    \end{equation*}
    \item\label{lem-type-ii} the theory $\PA \cup \{a \neq b\} \cup p(a) \cup p(b)$ is consistent.
\end{enumerate}
\end{lemma}

The lemma can be obtained from known constructions of indiscernible types (see the proof of Theorem 3.1.2 in \cite{kossak-schmerl} and the Remark following it),
but in order to make the paper more self-contained,
we give a relatively simple proof based on flexible formulas.

\begin{proof}[Proof of Lemma \ref{lem:type}]
Let $\varphi_0(x),\varphi_1(x),\ldots$ be a computable enumeration of unary $\LPA$ formulas. We set
\begin{equation*}
    p(a)\coloneq\{\varphi_n(a)\leftrightarrow\forall y\,\xi_n(y)\colon\,n\in\omega\}\text{,}
\end{equation*}
for a suitably chosen computable sequence of flexible arithmetical formulas $\xi_n$. This ensures that Condition~\ref{lem-type-i} and the computability requirement are met. To guarantee the fulfillment of \ref{lem-type-ii}, we directly design the construction so that every finite fragment of $p$ is consistently realizable by infinitely many elements.

Let
\begin{equation*}
    \alpha_n=\bigwedge_{k<n}\big(\varphi_k(x)\leftrightarrow\forall y\,\xi_k(y)\big)\text{\qquad and\qquad}A_n=\left\{\exists^{\geqslant m}x\,\alpha_n(x)\colon\,m\in\omega\right\}\cup\PA\text{,}
\end{equation*}
where $\xi_n$ is $\Sigma_{r(n)}$-flexible over $A_n$, with $r(n)$ being the least $\ell$ such that the formula $\alpha_n(x)\land\varphi_n(x)$ is $\Sigma_\ell$. By Theorem~\ref{tw_flex_main}, the sequence $(\xi_n)_{n\in\omega}$ can be chosen to be computable. 
To establish the lemma, it suffices to show that the consistency of $A_n$ implies the consistency of $A_{n+1}$ for every $n\in\omega$, since $A_0=\PA$, which is assumed to be consistent.

For the sake of contradiction, suppose that $A_n$ is consistent but $A_{n+1}$ is inconsistent. Then both of the following theories:
\begin{gather*}
    \PA\;\cup\;\left\{\exists^{\geqslant m}x\,(\alpha_n(x)\land\varphi_n(x))\land\forall y\,\xi_n(y)\colon\,m\in\omega\right\}\text{\quad and}\\
    \PA\;\cup\;\left\{\exists^{\geqslant m}x\,(\alpha_n(x)\land\lnot\varphi_n(x))\land\lnot\forall y\,\xi_n(y)\colon\,m\in\omega\right\}\text{,}
\end{gather*}
are inconsistent, and thus there are $m$ and $m'$ such that:
\begin{equation}\label{eq:indscr_type}
\begin{aligned}
    \PA\;&\vdash\;\Big(\exists^{\geqslant m}x\,(\alpha_n(x)\land\varphi_n(x))\Big)\,\rightarrow\,\lnot\forall y\,\xi_n(y)\text{,}\\
    \PA\;&\vdash\;\Big(\exists^{\geqslant m'}x\,(\alpha_n(x)\land\lnot\varphi_n(x))\Big)\,\rightarrow\,\forall y\,\xi_n(y)\text{.}
\end{aligned}
\end{equation}
Since $\PA\subseteq A_n$ and
\begin{equation*}
    A_n\;\vdash\;\lnot\Big(\exists^{\geqslant m}x\,(\alpha_n(x)\land\varphi_n(x))\Big)\,\rightarrow\,\Big(\exists^{\geqslant m'}x\,(\alpha_n(x)\land\lnot\varphi_n(x))\Big)\text{,}
\end{equation*}
from (\ref{eq:indscr_type}), we obtain that
\begin{equation*}
    A_n\;\vdash\;\Big(\exists^{\geqslant m}x\,(\alpha_n(x)\land\varphi_n(x))\Big)\,\leftrightarrow\,\lnot\forall y\,\xi_n(y)\text{.}
\end{equation*}
It follows that
\begin{equation*}
    A_n\;\vdash\;\lnot\Big(\Big(\exists^{\geqslant m}x\,(\alpha_n(x)\land\varphi_n(x))\Big)\,\leftrightarrow\,\forall y\,\xi_n(y)\Big)\text{,}
\end{equation*}
and thus that
\begin{equation*}
    A_n\;\vdash\;\lnot\forall y\,\Big(\Big(\exists^{\geqslant m}x\,(\alpha_n(x)\land\varphi_n(x))\Big)\,\leftrightarrow\,\xi_n(y)\Big)\text{.}
\end{equation*}
This contradicts the $\Sigma_{r(n)}$-flexibility of $\xi_n$ (following from the consistency of $A_n$).
\end{proof}

\begin{proof}[Proof of Theorem \ref{thm:neat-not-sem-tight}]
Let $T$ be ${\PA} + {p(a)}$, where $p(x)$ is the type
provided by Lemma \ref{lem:type}. Clearly, $T$ is an r.e.~theory.

We claim that $T$ is not semantically tight, even in the weak sense of \cite{fh:bi-interpretation-set} 
(see one of the remarks following Definition \ref{def:sem-tight}). Indeed, let $(\MM,a,b)$ be a model in which $a$ and $b$ are distinct elements realizing $p$, and let $\mathcal{N}:= \mathcal{K}(\mathcal{M},a,b)$, i.e. the submodel of $\mathcal{M}$ with the universe consisting of the elements which are definable in $\MM$ with the parameter $\langle a,b\rangle$. Since $\mathcal{N}$ is an elementary submodel of $\mathcal{M}$, it follows that both $(\mathcal{N},a)$, $(\mathcal{N},b)$ are models of $T$. Clearly, $(\NN,a)$ and $(\NN,b)$ are bi-interpretable (using $a$ and $b$ as parameters). However, it follows from Ehrenfeucht's Lemma (see Section \ref{sec:prelim}) that $\NN$ admits no non-trivial automorphisms. Hence $(\NN,a)$ and $(\NN,b)$ witness that $T$ is not semantically tight.

We now argue that $T$ is neat. Take any two extensions $U$ and $V$ of $T$ and assume that $\mathsf{V}:U\rhd V$ and $\mathsf{U}: V\rhd U$ witness that $U$ is a retract of $V$. In particular, $\mathsf{V}\mathsf{U}$ is $U$-provably isomorphic to $\mathsf{id}_U$. Take any $(\MM,a)\vDash U$. 
We claim that $(\MM,a) \simeq (\MM,a)^{\mathsf{V}}$,
so in particular $(\MM,a)\vDash V$, which will suffice to prove neatness. 

Consider $(\NN,b):= (\MM,a)^{\mathsf{V}}$ and $(\MM^*, a^*):= (\NN,b)^{\mathsf{U}}$. Since $\mathsf{V}$ and $\mathsf{U}$ witness the retraction between $U$ and $V$, it follows that $(\MM^*,a^*)$ is $(\MM,a)$-definably isomorphic to $(\MM,a)$. By the solidity of $\PA$, there is also an $\MM$-definable
isomorphism $\iota$ between $\MM$ and $\NN$.
Moreover, since $\iota$ is in fact the map from Lemma \ref{ultimate_dedekind}, and the interpretation $\mathsf{V}$
of $(\NN,b)$ uses no parameters from $\MM$ other than $a$, the definition of $\iota$ also has $a$ as its unique parameter. This means that the element $\iota^{-1}(b)$ of $\MM$ is definable in $\MM$ from $a$.
 
Both $(\MM,a)$ and $(\NN,b)$ are models of $T$, so the properties of the type $p(x)$ imply that the $\LPA$-types of $a$ and $b$ are determined by the theories of $\MM$ and $\NN$, which
are the same because $\MM$ and $\NN$ are isomorphic. 
Hence, $\iota^{-1}(b)$ is not only definable from $a$ in $\MM$, but also has the same arithmetical type as $a$. 
By Ehrenfeucht's Lemma (see Section \ref{sec:prelim}), we obtain $\iota^{-1}(b)=a$, which means that $\iota$ is an isomorphism between $(\MM,a)$ and $(\NN,b)$; hence $(\MM,a) \simeq (\NN,b)$ as claimed.
\end{proof}

\section{Solid subtheories of stronger systems}\label{sec:soa}

The focus of this paper is on subtheories of first-order arithmetic. However, the question on the existence of solid proper subtheories asked in \cite{enayat:variations} 
concerned not only $\PA$, but also other foundationally relevant axiom {systems}, like second-order arithmetic $\Zt$ and ZF set theory. 

As in the case of $\PA$, we feel that, in order to avoid trivial examples, in the case of more powerful systems the question should also be not about proper subtheories as such, but proper subtheories containing a sufficiently strong characteristic fragment of a given axiom scheme, or even better about arbitrarily strong subtheories.
In the present paper, we do not take up that problem at length. However, we observe that our methods, combined 
in one case with some recent results from the literature, provide some
proper solid subsystems of set theory
and second-order arithmetic. 

As a first example, we sketch the proof of the following proposition (below, $\ZF_{\Pi_n}$ denotes the set of $\Pi_n$ consequences of $\ZF$). 
We expect
that similar methods would work in the context of restricted fragments of $\Zt$, although we have not verified the details.
\begin{prop}\label{prop:set-theory}
Assume that $\ZF$ is consistent. For every $n$, there is a proper solid subtheory of $\ZF$ extending $\ZF_{\Pi_{n+1}}$.
\end{prop}
\begin{proof}
Clearly we can assume that $n \ge 3$. We adjust the main construction from Section \ref{sec:solidity-below}. The structure $\mathcal{H}$ now becomes an arithmetically defined model of ${\ZF} + {\mathnormal{V\!=\!L}}$, obtained as previously by an arithmetized Henkin-style completeness proof. Unlike in the construction from Section \ref{sec:proper-solid}, 
this time we do not need to code any additional subsets of $\mathbb{N}$ in $\mathcal{H}$. As previously, $\mathcal{K}_{n}(\mathcal{H})$ stands for the submodel of $\mathcal{H}$ consisting of $\Sigma_{n}$-definable elements of $\mathcal{H}$ (where we consider $\Sigma_n$ formulas in the language of set theory). As shown in \cite[Theorem~3]{enalel_categoricity_corr}, for $n\geq 3$ it holds that $\mathcal{K}_{n}(\mathcal{H})\preceq_{n}\mathcal{H}$ but $\mathcal{K}_{n}(\mathcal{H})\vDash \neg\mathrm{Coll}(\Sigma_{n+1})$, where $\mathrm{Coll}(\Sigma_{n+1})$ is a $\Pi_{n+4}$ sentence expressing the collection principle for $\Sigma_{n+1}$-definable relations. By the partial elementarity, $\mathcal{K}_{n}(\mathcal{H})\vDash \ZF_{\Pi_{n+1}}$. The proof  of \cite[Theorem 2a)]{enalel_categoricity_corr}  shows that $\mathcal{K}_{n}(\mathcal{H})$ is bi-interpretable with the standard model of arithmetic $\mathbb{N}$. Let $\mathsf{N}_n^s$, $\mathsf{K}_n^s$, $i_n^s$ and $j_n^s$ (where $s$ stands for ``sets'') be defined analogously to $\mathsf{N}_{n-1}$, $\mathsf{K}_{n-1}$, $i_{n-1}$, $j_{n-1}$, respectively (of course, $\mathsf{N}_n^s$ is now applied in a set-theoretic universe, where it defines a cut in
the natural numbers of that universe). 
We can construct a theory $\mathrm{IA}(n)$
which is defined like $\mathrm{IT}(n-1)$ with the exception that $\mathsf{N}_{n-1}$, $\mathsf{K}_{n-1}$, $i_{n-1}$, $j_{n-1}$ get replaced with their $s$-variants and
\begin{itemize}
    \item in (i) we take $\ZF_{\Pi_{n+1}}$ instead of $\IS{n}+\exp$ and $\neg\mathrm{Coll}(\Sigma_{n+1})$, instead of $\neg \BS{n+1}$;
    \item in (iii) we say that $\mathsf{N}_n^s\vDash {\PA} + {\Con{\ZF}}$.
\end{itemize}
Then, as in the proof of Proposition \ref{prop:biint_truth}, we can show that $\mathrm{IA}(n)$ is bi-interpretable with ${\PA}+{\Con{\ZF}}$. Hence, $\mathrm{IA}(n)$ is solid. Now, our target theory can be defined as 

\[{(\ZF_{\Pi_{n+1}}}\cup {\mathrm{IA}(n)) \oplus_{\mathrm{Coll}(\Sigma_{n+1})}{\ZF}}.\]

Both  $\mathrm{IA}(n)$ and $\ZF$ are solid. 
So, by (an obvious variant of) Proposition \ref{pres_solid_disj_sum}, it remains to show that the family $\{ \mathrm{IA}(n), \ZF\}$ is retract-disjoint. This however, by Corollary \ref{cor_biint_pres_retract}, is equivalent to the retract-disjointness of the family $\{{\PA} + {\Con{\ZF}}, \ZF\}$. The latter can be proved as follows: assume first that $\MM\vDash \ZF$ and $(\mathsf{N},\mathsf{M})$ is a retraction in $\MM$ such that $\MM^{\mathsf{N}}\vDash \PA$. Arguing as in the proof of Lemma \ref{lem:indfact}, we see that $\omega^{\MM}$ and ${\MM^{\mathsf{N}}}$
are $\MM$-definably isomorphic. However then it follows that $\MM$ sees a definable bijection between its universe and a subset of its $\omega$, which contradicts replacement. So assume now that $\NN\vDash {\PA} + {\Con{\ZF}}$, that $(\mathsf{M},\mathsf{N})$ is a retraction in $\NN$ 
and that $\NN^{\mathsf{M}}\vDash \ZF$. Then, as above, 
${\NN}$ and $\omega^{\NN^{\mathsf{M}}}$ 
are $\NN$-definably isomorphic. However, $\NN^{\mathsf{M}}$ carries a satisfaction
predicate for $\omega^{\NN^{\mathsf{M}}}$, so $\NN$ would be able to pull it back through the isomorphism, contradicting Tarski's undefinability of truth theorem.
\end{proof}

Let us note that the theories constructed 
in Proposition \ref{prop:set-theory} 
do not prove the full induction scheme for the natural numbers. In general it seems much harder to construct proper solid subtheories of a solid theory $T$ that would contain some prescribed $T$-provable scheme  (such as e.g.~full induction in the context of subsystems of $\Zt$ or full comprehension in the context of $\ZF$). Nevertheless, to illustrate what can be done using our
methods in a rather straightforward manner, 
we provide a simple example of a proper solid subtheory of $\Zt$ containing arithmetical comprehension and full induction.

Recall that $\ACA_0'$ is the theory that extends $\ACA_0$ by the axiom ``for every set $X$ and every number $k$, the set $X^{(k)}$ exists'', and $\ACA'$ is $\ACA_0'$ plus the full induction scheme for the language of second-order arithmetic. 

\begin{prop}
There exists an r.e.~solid proper subtheory of $\Zt$
extending $\ACA'$.
\end{prop}

\begin{proof}
Let $U$ be the theory of those models of $\ACA'$ that consider themselves to consist of the arithmetical sets. In other words, $U$ is $\ACA'$ plus the axiom \[\forall X\, \exists k\, (X \text{ is Turing-reducible to } 0^{(k)}).\]
We claim that ${U} \oplus_{\Pi^1_1\text{-}\mathrm{CA}_0} \Zt$ is solid.

By \cite{enayat:variations}, $\Zt$ is solid. It is also easy to show that $U$ is solid: if $\MM_1 \rhd \MM_2 \rhd \MM_3$ are models of $U$ and we are given an $\MM_1$-definable isomorphism between $\MM_1$ and $\MM_3$, then an argument just like the one for $\PA$ shows that there is an $\MM_1$-definable isomorphism  between the first-order parts of $\MM_1$ and $\MM_2$ (the argument makes use of the full induction scheme available in $U$, because the interpretations between the models might not be first-order definable). This extends to the second-order parts in the natural way: 
if $e_1, k_1 \in \MM_1$ 
are mapped by the first-order isomorphism to $e_2, k_2 \in \MM_2$, respectively,
then map the set computed by $e_1$ with the oracle $0^{(k_1)}$ (i.e. $\{e_1\}^{0^{(k_1)}}$) to the set computed by $e_2$ with the oracle $0^{(k_2)}$ (i.e. $\{e_2\}^{0^{(k_2)}}$). 
We can prove by induction on $k_1$ that this is an embedding of the second-order universes, and it is surjective because $\MM_2$ satisfies $U$.

It remains to show that the family $\{U, \Zt\}$ is retract-disjoint. This is similar to the proofs of Lemma \ref{lem:no-mix-ct-pa} and Proposition \ref{prop:set-theory}. If say $\MM \vDash U$
is a retract of $\NN \vDash \Zt$, then again by the usual argument the first-order universes of $\MM$ and $\NN$ are definably isomorphic. But the second-order universe of $\NN$ contains a set that is a definition of satisfaction for second-order formulas in $\MM$, because all sets in $\MM$ are internally arithmetical. We could use the isomorphism between the first-order universes to transfer this definition to $\MM$, contradicting Tarski's theorem. The argument for the case when a model of $\Zt$ is a retract of a model of $U$ is analogous.
\end{proof}

By a somewhat more involved argument in a similar spirit (using finite iterations of hyperjump in the place of jump),
we can prove the solidity of a proper fragment of $\Zt$
containing $\Pi^1_1$-comprehension (and full induction).
Solid proper subtheories of $\Zt$ containing fragments at the level of $\Pi^1_2$-comprehension and beyond are left as a possible topic for future work.

\section{Conclusions and open problems}\label{sec:conclusion}

The work presented in this paper provides significant new insight into the behaviour of solidity and similar properties for subtheories of first-order arithmetic, as well as into the precise relations between the properties.

When it comes to solid subtheories of $\PA$, we were able to not only show the existence of relatively strong solid proper subtheories of $\PA$, but also to provide examples that are strictly below $\PA$ in terms of interpretability rather than just provability. Still, it seems that a piece of the picture remains missing.

Recall the Remark at the end of Section \ref{subsec:solid-non-interpret}, pointing out that the reason why the theories $TD_n$ defined in that section fail to interpret $\PA$ is that they are unable to make an infinite case distinction, but each model of each $TD_n$ actually interprets a model of $\PA$. In fact, among the solid subtheories of $\PA$ that we are able to come up with, all the ones containing a reasonable amount of arithmetic (we leave aside trivial counterexamples like ``either $\PA$ holds or the universe has one element'') have the property that each of their models interprets a model of $\PA$. Thus, the following question seems to be of interest.

\begin{question}\label{q:model-not-interpret}
Can a solid subtheory of $\PA$ containing $\EA$ have a model that does not interpret any model of $\PA$?
\end{question}

A potential negative answer to Question \ref{q:model-not-interpret} could be interpreted as meaning
that, in an appropriately weakened sense, $\PA$ is a minimal solid theory after all.

We also mention two questions in a similar spirit originally raised by other authors. One of the questions was already asked by Enayat in \cite{enayat:variations}:

\begin{question}
Is there a consistent finitely axiomatizable solid sequential theory?
\end{question}

The other was suggested by Fedor Pakhomov (private communication): 

\begin{question}\label{q:solid-interpreted-in-isn}
Is there a solid sequential theory that is interpretable in $\IS{n}$, for some $n$?
\end{question}

Note that a positive answer to Question \ref{q:solid-interpreted-in-isn} implies a positive answer to Question \ref{q:model-not-interpret}, because (for G\"odel-style reasons) $\IS{n}$ has models that do not interpret any model of $\PA$.

Turning now to the topic of relations between tightness, solidity and the other notions, since solidity is the strongest of the four categoricity-like properties considered and tightness the weakest, \emph{a priori} there could be up to six combinations of the properties. These combinations correspond to rows of Table \ref{table:separations} below and are listed in the first column. 

Before we started our work, all theories that had been classified were either solid or not even tight. We were able to come up with examples witnessing some separations between the properties, including a theory that is tight but has none of the stronger properties and a theory that is neat but has neither of the semantical properties. All the separations we obtained can be witnessed by theories that are at least sequential, although separating examples that we managed to classify exactly do not have arbitrary arithmetical strength.
Table \ref{table:separations} lists the combinations of properties known to occur based on results up to and including the present paper, with references to sections of this paper or to earlier work, as appropriate.

\begin{table}[htbp]
\centering
 \begin{tabular}{|c | c | c|} 
  \cline{2-3}
   \multicolumn{1}{c|}{} & sequential & arbitrarily strong below $\PA$ \\  
  \hline
  not tight & \cite{enayat:variations} & \cite{enayat:variations} \\ 
  \hline
  tight only & Sec.~\ref{subsec:pam} & ? \\
  \hline
  neat but not sem.~tight & Sec.~\ref{subsec:neat-not-sem-tight} & ? \\
  \hline
  sem.~tight but not neat & ? &  ? \\
  \hline
  sem.~tight and neat only & ? & ? \\  
  \hline
  solid & \cite{enayat:variations} & Sec.~\ref{sec:proper-solid} \\
  \hline
 \end{tabular}
 
\caption{Possible combinations of categoricity-like properties discovered up to and including the present paper.}

\label{table:separations}

\end{table}

Recently, the first author developed a new method that makes it possible to show that in fact all six combinations corresponding to rows of Table \ref{table:separations}
are possible, as witnessed by variants of finite set theory.
These advances will be reported in a separate work \cite{gruza:separations}. Many separations between tightness-like properties, including some not considered in this paper, were also recently obtained by Elliot Glazer (private communication), mostly in the context of (infinite) set theory.

One example considered in the present paper that is not represented in Table \ref{table:separations} is the
family of tight but not neat theories $S_n$ 
studied in Section \ref{subsec:tight-not-neat}.
These theories could fill either the second or the fourth row of Table \ref{table:separations}, depending on whether they are semantically tight. Thus, we  ask:

\begin{question}
Are the theories $S_n$ from Section \ref{subsec:tight-not-neat} semantically tight?
\end{question}

\section*{Acknowledgements}

The authors would like to thank Ali Enayat for insightful discussions, for his comments on earlier versions of this paper, and for asking the question that led to Theorem \ref{main_thm_solid_inf-below_PA}. We are also indebted to Konstantinos Papafilippou for fruitful conversations regarding the proof of Lemma \ref{lem:type}, and to Elliot Glazer for a careful explanation of his recent work on tightness-like properties.

The work of the first and third authors was supported by  grant no.~2022/46/E/HS1/00452 of the National Science Centre, Poland. The work of the second author was supported by grant no.~2023/49/B/ST1/02627 of the National Science Centre, Poland.


\bibliographystyle{plain}
\bibliography{solidity}

\end{document}

%% file: retraction.tex
\tikzstyle{ipe stylesheet} = [
  ipe import,
  even odd rule,
  line join=round,
  line cap=butt,
  ipe pen normal/.style={line width=0.4},
  ipe pen heavier/.style={line width=0.8},
  ipe pen fat/.style={line width=1.2},
  ipe pen ultrafat/.style={line width=2},
  ipe pen normal,
  ipe mark normal/.style={ipe mark scale=3},
  ipe mark large/.style={ipe mark scale=5},
  ipe mark small/.style={ipe mark scale=2},
  ipe mark tiny/.style={ipe mark scale=1.1},
  ipe mark normal,
  /pgf/arrow keys/.cd,
  ipe arrow normal/.style={scale=7},
  ipe arrow large/.style={scale=10},
  ipe arrow small/.style={scale=5},
  ipe arrow tiny/.style={scale=3},
  ipe arrow normal,
  /tikz/.cd,
  ipe arrows, 
  <->/.tip = ipe normal,
  ipe dash normal/.style={dash pattern=},
  ipe dash dotted/.style={dash pattern=on 1bp off 3bp},
  ipe dash dashed/.style={dash pattern=on 4bp off 4bp},
  ipe dash dash dotted/.style={dash pattern=on 4bp off 2bp on 1bp off 2bp},
  ipe dash dash dot dotted/.style={dash pattern=on 4bp off 2bp on 1bp off 2bp on 1bp off 2bp},
  ipe dash normal,
  ipe node/.append style={font=\normalsize},
  ipe stretch normal/.style={ipe node stretch=1},
  ipe stretch normal,
  ipe opacity 10/.style={opacity=0.1},
  ipe opacity 30/.style={opacity=0.3},
  ipe opacity 50/.style={opacity=0.5},
  ipe opacity 75/.style={opacity=0.75},
  ipe opacity opaque/.style={opacity=1},
  ipe opacity opaque,
]
\definecolor{red}{rgb}{1,0,0}
\definecolor{blue}{rgb}{0,0,1}
\definecolor{green}{rgb}{0,1,0}
\definecolor{yellow}{rgb}{1,1,0}
\definecolor{orange}{rgb}{1,0.647,0}
\definecolor{gold}{rgb}{1,0.843,0}
\definecolor{purple}{rgb}{0.627,0.125,0.941}
\definecolor{gray}{rgb}{0.745,0.745,0.745}
\definecolor{brown}{rgb}{0.647,0.165,0.165}
\definecolor{navy}{rgb}{0,0,0.502}
\definecolor{pink}{rgb}{1,0.753,0.796}
\definecolor{seagreen}{rgb}{0.18,0.545,0.341}
\definecolor{turquoise}{rgb}{0.251,0.878,0.816}
\definecolor{violet}{rgb}{0.933,0.51,0.933}
\definecolor{darkblue}{rgb}{0,0,0.545}
\definecolor{darkcyan}{rgb}{0,0.545,0.545}
\definecolor{darkgray}{rgb}{0.663,0.663,0.663}
\definecolor{darkgreen}{rgb}{0,0.392,0}
\definecolor{darkmagenta}{rgb}{0.545,0,0.545}
\definecolor{darkorange}{rgb}{1,0.549,0}
\definecolor{darkred}{rgb}{0.545,0,0}
\definecolor{lightblue}{rgb}{0.678,0.847,0.902}
\definecolor{lightcyan}{rgb}{0.878,1,1}
\definecolor{lightgray}{rgb}{0.827,0.827,0.827}
\definecolor{lightgreen}{rgb}{0.565,0.933,0.565}
\definecolor{lightyellow}{rgb}{1,1,0.878}
\definecolor{black}{rgb}{0,0,0}
\definecolor{white}{rgb}{1,1,1}
\begin{tikzpicture}[ipe stylesheet]
  \filldraw[shift={(159.999, 767.997)}, scale=1.0588, rgb color={draw=0.702 0.804 0.878}, rgb color={fill=0.702 0.804 0.878}]
    (0, 0) rectangle (272, -272);
  \filldraw[rgb color={draw=0.843 0.655 0.275}, rgb color={fill=0.843 0.655 0.275}]
    (304, 624) circle[radius=118.929];
  \draw[rgb color={draw=0.702 0.804 0.878}, ipe pen fat]
    (176, 768)
     -- (176, 480);
  \draw[rgb color={draw=0.702 0.804 0.878}, ipe pen fat]
    (192, 768)
     -- (192, 480);
  \draw[rgb color={draw=0.702 0.804 0.878}, ipe pen fat]
    (208, 768)
     -- (208, 480);
  \draw[shift={(160, 688)}, rotate=90, rgb color={draw=0.702 0.804 0.878}, ipe pen fat]
    (0, 0)
     -- (0, -288);
  \draw[rgb color={draw=0.702 0.804 0.878}, ipe pen fat]
    (384, 768)
     -- (384, 480);
  \draw[rgb color={draw=0.702 0.804 0.878}, ipe pen fat]
    (400, 768)
     -- (400, 480);
  \draw[rgb color={draw=0.702 0.804 0.878}, ipe pen fat]
    (416, 768)
     -- (416, 480);
  \draw[rgb color={draw=0.702 0.804 0.878}, ipe pen fat]
    (432, 768)
     -- (432, 480);
  \draw[shift={(160, 496)}, rotate=90, rgb color={draw=0.702 0.804 0.878}, ipe pen fat]
    (0, 0)
     -- (0, -288);
  \draw[shift={(160, 512)}, rotate=90, rgb color={draw=0.702 0.804 0.878}, ipe pen fat]
    (0, 0)
     -- (0, -288);
  \draw[shift={(160, 528)}, rotate=90, rgb color={draw=0.702 0.804 0.878}, ipe pen fat]
    (0, 0)
     -- (0, -288);
  \draw[shift={(160, 560)}, rotate=90, rgb color={draw=0.702 0.804 0.878}, ipe pen fat]
    (0, 0)
     -- (0, -288);
  \draw[shift={(160, 592)}, rotate=90, rgb color={draw=0.702 0.804 0.878}, ipe pen fat]
    (0, 0)
     -- (0, -288);
  \draw[shift={(160, 624)}, rotate=90, rgb color={draw=0.702 0.804 0.878}, ipe pen fat]
    (0, 0)
     -- (0, -288);
  \draw[shift={(160, 656)}, rotate=90, rgb color={draw=0.702 0.804 0.878}, ipe pen fat]
    (0, 0)
     -- (0, -288);
  \draw[shift={(160, 704)}, rotate=90, rgb color={draw=0.702 0.804 0.878}, ipe pen fat]
    (0, 0)
     -- (0, -288);
  \draw[shift={(160, 720)}, rotate=90, rgb color={draw=0.702 0.804 0.878}, ipe pen fat]
    (0, 0)
     -- (0, -288);
  \draw[shift={(160, 736)}, rotate=90, rgb color={draw=0.702 0.804 0.878}, ipe pen fat]
    (0, 0)
     -- (0, -288);
  \draw[rgb color={draw=0.702 0.804 0.878}, ipe pen fat]
    (256, 768)
     -- (256, 480);
  \draw[rgb color={draw=0.702 0.804 0.878}, ipe pen fat]
    (288, 768)
     -- (288, 480);
  \draw[rgb color={draw=0.702 0.804 0.878}, ipe pen fat]
    (320, 768)
     -- (320, 480);
  \draw[rgb color={draw=0.702 0.804 0.878}, ipe pen fat]
    (352, 768)
     -- (352, 480);
  \draw[shift={(160, 752)}, rotate=90, rgb color={draw=0.702 0.804 0.878}, ipe pen fat]
    (0, 0)
     -- (0, -288);
  \filldraw[rgb color={draw=0.702 0.804 0.878}, rgb color={fill=0.302 0.514 0.671}]
    (240, 672) rectangle (368, 544);
  \draw[rgb color={draw=0.702 0.804 0.878}]
    (224, 528) rectangle (224, 528);
  \draw[rgb color={draw=0.702 0.804 0.878}, ipe pen fat]
    (224, 768)
     -- (224, 480);
  \draw[shift={(160, 544)}, rotate=90, rgb color={draw=0.702 0.804 0.878}, ipe pen fat]
    (0, 0)
     -- (0, -288);
  \draw[shift={(160, 576)}, rotate=90, rgb color={draw=0.702 0.804 0.878}, ipe pen fat]
    (0, 0)
     -- (0, -288);
  \draw[shift={(160, 608)}, rotate=90, rgb color={draw=0.702 0.804 0.878}, ipe pen fat]
    (0, 0)
     -- (0, -288);
  \draw[shift={(160, 640)}, rotate=90, rgb color={draw=0.702 0.804 0.878}, ipe pen fat]
    (0, 0)
     -- (0, -288);
  \draw[shift={(160, 672)}, rotate=90, rgb color={draw=0.702 0.804 0.878}, ipe pen fat]
    (0, 0)
     -- (0, -288);
  \filldraw[rgb color={draw=0.702 0.804 0.878}, rgb color={fill=0.702 0.804 0.878}]
    (224, 736) rectangle (240, 720);
  \filldraw[rgb color={draw=0.702 0.804 0.878}, rgb color={fill=0.702 0.804 0.878}]
    (208, 720) rectangle (224, 704);
  \filldraw[rgb color={draw=0.702 0.804 0.878}, rgb color={fill=0.702 0.804 0.878}]
    (192, 704) rectangle (208, 688);
  \filldraw[rgb color={draw=0.702 0.804 0.878}, rgb color={fill=0.702 0.804 0.878}]
    (208, 704) rectangle (224, 688);
  \filldraw[rgb color={draw=0.702 0.804 0.878}, rgb color={fill=0.702 0.804 0.878}]
    (224, 720) rectangle (240, 704);
  \filldraw[rgb color={draw=0.702 0.804 0.878}, rgb color={fill=0.702 0.804 0.878}]
    (240, 736) rectangle (256, 720);
  \filldraw[rgb color={draw=0.702 0.804 0.878}, rgb color={fill=0.702 0.804 0.878}]
    (256, 736) rectangle (272, 720);
  \filldraw[rgb color={draw=0.702 0.804 0.878}, rgb color={fill=0.702 0.804 0.878}]
    (256, 752) rectangle (272, 736);
  \filldraw[rgb color={draw=0.702 0.804 0.878}, rgb color={fill=0.702 0.804 0.878}]
    (272, 752) rectangle (288, 736);
  \filldraw[rgb color={draw=0.702 0.804 0.878}, rgb color={fill=0.702 0.804 0.878}]
    (288, 752) rectangle (304, 736);
  \filldraw[rgb color={draw=0.702 0.804 0.878}, rgb color={fill=0.702 0.804 0.878}]
    (304, 752) rectangle (320, 736);
  \filldraw[rgb color={draw=0.702 0.804 0.878}, rgb color={fill=0.702 0.804 0.878}]
    (320, 752) rectangle (336, 736);
  \filldraw[rgb color={draw=0.702 0.804 0.878}, rgb color={fill=0.702 0.804 0.878}]
    (336, 752) rectangle (352, 736);
  \filldraw[rgb color={draw=0.702 0.804 0.878}, rgb color={fill=0.702 0.804 0.878}]
    (336, 736) rectangle (352, 720);
  \filldraw[rgb color={draw=0.702 0.804 0.878}, rgb color={fill=0.702 0.804 0.878}]
    (352, 736) rectangle (368, 720);
  \filldraw[rgb color={draw=0.702 0.804 0.878}, rgb color={fill=0.702 0.804 0.878}]
    (368, 736) rectangle (384, 720);
  \filldraw[rgb color={draw=0.702 0.804 0.878}, rgb color={fill=0.702 0.804 0.878}]
    (368, 720) rectangle (384, 704);
  \filldraw[rgb color={draw=0.702 0.804 0.878}, rgb color={fill=0.702 0.804 0.878}]
    (384, 720) rectangle (400, 704);
  \filldraw[rgb color={draw=0.702 0.804 0.878}, rgb color={fill=0.702 0.804 0.878}]
    (384, 704) rectangle (400, 688);
  \filldraw[rgb color={draw=0.702 0.804 0.878}, rgb color={fill=0.702 0.804 0.878}]
    (400, 704) rectangle (416, 688);
  \filldraw[rgb color={draw=0.702 0.804 0.878}, rgb color={fill=0.702 0.804 0.878}]
    (400, 688) rectangle (416, 672);
  \filldraw[rgb color={draw=0.702 0.804 0.878}, rgb color={fill=0.702 0.804 0.878}]
    (400, 672) rectangle (416, 656);
  \filldraw[rgb color={draw=0.702 0.804 0.878}, rgb color={fill=0.702 0.804 0.878}]
    (416, 672) rectangle (432, 656);
  \filldraw[rgb color={draw=0.702 0.804 0.878}, rgb color={fill=0.702 0.804 0.878}]
    (416, 656) rectangle (432, 640);
  \filldraw[rgb color={draw=0.702 0.804 0.878}, rgb color={fill=0.702 0.804 0.878}]
    (416, 640) rectangle (432, 624);
  \filldraw[rgb color={draw=0.702 0.804 0.878}, rgb color={fill=0.702 0.804 0.878}]
    (416, 624) rectangle (432, 608);
  \filldraw[rgb color={draw=0.702 0.804 0.878}, rgb color={fill=0.702 0.804 0.878}]
    (416, 608) rectangle (432, 592);
  \filldraw[rgb color={draw=0.702 0.804 0.878}, rgb color={fill=0.702 0.804 0.878}]
    (416, 592) rectangle (432, 576);
  \filldraw[rgb color={draw=0.702 0.804 0.878}, rgb color={fill=0.702 0.804 0.878}]
    (400, 592) rectangle (416, 576);
  \filldraw[rgb color={draw=0.702 0.804 0.878}, rgb color={fill=0.702 0.804 0.878}]
    (400, 576) rectangle (416, 560);
  \filldraw[rgb color={draw=0.702 0.804 0.878}, rgb color={fill=0.702 0.804 0.878}]
    (400, 560) rectangle (416, 544);
  \filldraw[rgb color={draw=0.702 0.804 0.878}, rgb color={fill=0.702 0.804 0.878}]
    (384, 560) rectangle (400, 544);
  \filldraw[rgb color={draw=0.702 0.804 0.878}, rgb color={fill=0.702 0.804 0.878}]
    (384, 544) rectangle (400, 528);
  \filldraw[rgb color={draw=0.702 0.804 0.878}, rgb color={fill=0.702 0.804 0.878}]
    (368, 544) rectangle (384, 528);
  \filldraw[rgb color={draw=0.702 0.804 0.878}, rgb color={fill=0.702 0.804 0.878}]
    (368, 528) rectangle (384, 512);
  \filldraw[rgb color={draw=0.702 0.804 0.878}, rgb color={fill=0.702 0.804 0.878}]
    (352, 528) rectangle (368, 512);
  \filldraw[rgb color={draw=0.702 0.804 0.878}, rgb color={fill=0.702 0.804 0.878}]
    (336, 528) rectangle (352, 512);
  \filldraw[rgb color={draw=0.702 0.804 0.878}, rgb color={fill=0.702 0.804 0.878}]
    (336, 512) rectangle (352, 496);
  \filldraw[rgb color={draw=0.702 0.804 0.878}, rgb color={fill=0.702 0.804 0.878}]
    (320, 512) rectangle (336, 496);
  \filldraw[rgb color={draw=0.702 0.804 0.878}, rgb color={fill=0.702 0.804 0.878}]
    (304, 512) rectangle (320, 496);
  \filldraw[rgb color={draw=0.702 0.804 0.878}, rgb color={fill=0.702 0.804 0.878}]
    (288, 512) rectangle (304, 496);
  \filldraw[rgb color={draw=0.702 0.804 0.878}, rgb color={fill=0.702 0.804 0.878}]
    (272, 512) rectangle (288, 496);
  \filldraw[rgb color={draw=0.702 0.804 0.878}, rgb color={fill=0.702 0.804 0.878}]
    (256, 512) rectangle (272, 496);
  \filldraw[rgb color={draw=0.702 0.804 0.878}, rgb color={fill=0.702 0.804 0.878}]
    (256, 528) rectangle (272, 512);
  \filldraw[rgb color={draw=0.702 0.804 0.878}, rgb color={fill=0.702 0.804 0.878}]
    (240, 528) rectangle (256, 512);
  \filldraw[rgb color={draw=0.702 0.804 0.878}, rgb color={fill=0.702 0.804 0.878}]
    (224, 528) rectangle (240, 512);
  \filldraw[rgb color={draw=0.702 0.804 0.878}, rgb color={fill=0.702 0.804 0.878}]
    (224, 544) rectangle (240, 528);
  \filldraw[rgb color={draw=0.702 0.804 0.878}, rgb color={fill=0.702 0.804 0.878}]
    (208, 544) rectangle (224, 528);
  \filldraw[rgb color={draw=0.702 0.804 0.878}, rgb color={fill=0.702 0.804 0.878}]
    (208, 560) rectangle (224, 544);
  \filldraw[rgb color={draw=0.702 0.804 0.878}, rgb color={fill=0.702 0.804 0.878}]
    (192, 560) rectangle (208, 544);
  \filldraw[rgb color={draw=0.702 0.804 0.878}, rgb color={fill=0.702 0.804 0.878}]
    (192, 576) rectangle (208, 560);
  \filldraw[rgb color={draw=0.702 0.804 0.878}, rgb color={fill=0.702 0.804 0.878}]
    (192, 592) rectangle (208, 576);
  \filldraw[rgb color={draw=0.702 0.804 0.878}, rgb color={fill=0.702 0.804 0.878}]
    (176, 592) rectangle (192, 576);
  \filldraw[rgb color={draw=0.702 0.804 0.878}, rgb color={fill=0.702 0.804 0.878}]
    (176, 608) rectangle (192, 592);
  \filldraw[rgb color={draw=0.702 0.804 0.878}, rgb color={fill=0.702 0.804 0.878}]
    (176, 624) rectangle (192, 608);
  \filldraw[rgb color={draw=0.702 0.804 0.878}, rgb color={fill=0.702 0.804 0.878}]
    (176, 640) rectangle (192, 624);
  \filldraw[rgb color={draw=0.702 0.804 0.878}, rgb color={fill=0.702 0.804 0.878}]
    (176, 656) rectangle (192, 640);
  \filldraw[rgb color={draw=0.702 0.804 0.878}, rgb color={fill=0.702 0.804 0.878}]
    (176, 672) rectangle (192, 656);
  \filldraw[rgb color={draw=0.702 0.804 0.878}, rgb color={fill=0.702 0.804 0.878}]
    (192, 672) rectangle (208, 656);
  \filldraw[rgb color={draw=0.702 0.804 0.878}, rgb color={fill=0.702 0.804 0.878}]
    (192, 688) rectangle (208, 672);
  \draw[rgb color={draw=0.702 0.804 0.878}, ipe pen fat]
    (240, 768)
     -- (240, 480);
  \draw[rgb color={draw=0.702 0.804 0.878}, ipe pen fat]
    (368, 768)
     -- (368, 480);
  \draw[rgb color={draw=0.702 0.804 0.878}, ipe pen fat]
    (272, 768)
     -- (272, 480);
  \draw[rgb color={draw=0.702 0.804 0.878}, ipe pen fat]
    (304, 768)
     -- (304, 480);
  \draw[rgb color={draw=0.702 0.804 0.878}, ipe pen fat]
    (336, 768)
     -- (336, 480);
  \node[ipe node, font=\huge]
     at (318.576, 686.904) {$\mathcal{M}^{\mathsf{N}}$};
  \node[ipe node, font=\huge]
     at (279.399, 615.511) {$\mathcal{M}^{\mathsf{NM}}$};
  \node[ipe node, font=\huge]
     at (356.558, 740.705) {$\mathcal{M}$};
  \node[ipe node, font=\huge]
     at (401.827, 695.275) {$\simeq$};
  \draw[ipe pen heavier, ->]
    (383.2899, 738.6427)
     arc[start angle=-33.2056, end angle=83.3783, x radius=75.8651, y radius=-75.8651];
\end{tikzpicture}

%% file: three-models.tex
\definecolor{red}{rgb}{1,0,0}
\definecolor{blue}{rgb}{0,0,1}
\definecolor{green}{rgb}{0,1,0}
\definecolor{yellow}{rgb}{1,1,0}
\definecolor{orange}{rgb}{1,0.5,0}
\definecolor{purple}{rgb}{0.75,0,0.25}
\definecolor{gray}{rgb}{0.5,0.5,0.5}
\definecolor{brown}{rgb}{0.75,0.5,0.25}
\definecolor{pink}{rgb}{1,0.75,0.75}
\definecolor{violet}{rgb}{0.5,0,0.5}
\definecolor{darkgray}{rgb}{0.25,0.25,0.25}
\definecolor{lightgray}{rgb}{0.75,0.75,0.75}
\definecolor{lime}{rgb}{0.75,1,0}
\definecolor{teal}{rgb}{0,0.5,0.5}
\definecolor{cyan}{rgb}{0,1,1}
\definecolor{magenta}{rgb}{1,0,1}
\definecolor{olive}{rgb}{0.5,0.5,0}
\definecolor{gold}{rgb}{1,0.843,0}
\definecolor{navy}{rgb}{0,0,0.502}
\definecolor{seagreen}{rgb}{0.18,0.545,0.341}
\definecolor{turquoise}{rgb}{0.251,0.878,0.816}
\definecolor{darkblue}{rgb}{0,0,0.545}
\definecolor{darkcyan}{rgb}{0,0.545,0.545}
\definecolor{darkgreen}{rgb}{0,0.392,0}
\definecolor{darkmagenta}{rgb}{0.545,0,0.545}
\definecolor{darkorange}{rgb}{1,0.549,0}
\definecolor{darkred}{rgb}{0.545,0,0}
\definecolor{lightblue}{rgb}{0.678,0.847,0.902}
\definecolor{lightcyan}{rgb}{0.878,1,1}
\definecolor{lightgreen}{rgb}{0.565,0.933,0.565}
\definecolor{lightyellow}{rgb}{1,1,0.878}
\definecolor{black}{rgb}{0,0,0}
\definecolor{white}{rgb}{1,1,1}
\begin{tikzpicture}[ipe import]
  \node[ipe node]
     at (184, 708) {$\mathcal{M}_1$};
  \draw
    (192, 704)
     -- (192, 608);
  \draw
    (188, 700)
     arc[start angle=180, end angle=360, x radius=4, y radius=-4];
  \draw[rgb color={draw=0.271 0.459 0.706}, semithick]
    (188, 680)
     arc[start angle=180, end angle=360, x radius=4, y radius=-4];
  \node[ipe node]
     at (248, 708) {$\mathcal{M}_2$};
  \draw
    (256, 704)
     -- (256, 608);
  \draw
    (252, 700)
     arc[start angle=180, end angle=360, x radius=4, y radius=-4];
  \draw[rgb color={draw=0.957 0.427 0.263}, semithick]
    (252, 680)
     arc[start angle=180, end angle=360, x radius=4, y radius=-4];
  \node[ipe node]
     at (312, 708) {$\mathcal{M}_3$};
  \draw
    (320, 704)
     -- (320, 608);
  \draw
    (316, 700)
     arc[start angle=180, end angle=360, x radius=4, y radius=-4];
  \draw[rgb color={draw=0.671 0.851 0.914}, semithick]
    (316, 680)
     arc[start angle=180, end angle=360, x radius=4, y radius=-4];
  \draw[rgb color={draw=0.271 0.459 0.706}, semithick]
    (192, 684)
     -- (192, 608);
  \draw[rgb color={draw=0.957 0.427 0.263}, semithick]
    (256, 684)
     -- (256, 608);
  \draw[rgb color={draw=0.671 0.851 0.914}, semithick]
    (320, 684)
     -- (320, 608);
  \draw[rgb color={draw=0.271 0.459 0.706}, semithick]
    (256, 668)
     -- (256, 608);
  \draw[rgb color={draw=0.271 0.459 0.706}, semithick]
    (252, 664)
     arc[start angle=180, end angle=360, x radius=4, y radius=-4];
  \draw[rgb color={draw=0.957 0.427 0.263}, semithick]
    (316, 664)
     arc[start angle=180, end angle=360, x radius=4, y radius=-4];
  \draw[rgb color={draw=0.957 0.427 0.263}, semithick]
    (320, 668)
     -- (320, 608);
  \draw[rgb color={draw=0.271 0.459 0.706}, semithick]
    (316, 648)
     arc[start angle=180, end angle=360, x radius=4, y radius=-4];
  \draw[rgb color={draw=0.271 0.459 0.706}, semithick]
    (320, 652)
     -- (320, 608);
  \draw[rgb color={draw=0.192 0.212 0.584}, semithick]
    (188, 648)
     arc[start angle=180, end angle=360, x radius=4, y radius=-4];
  \draw[rgb color={draw=0.192 0.212 0.584}, semithick]
    (192, 652)
     -- (192, 608);
  \draw[dashed, ->]
    (200, 680)
     -- (248, 664);
  \draw[dashed, ->]
    (264, 680)
     -- (312, 664);
  \draw[<-]
    (316, 684)
     .. controls (276, 748) and (236, 748) .. (196, 684);
  \node[ipe node]
     at (172, 680) {$\mathsf{N}_1$};
  \node[ipe node]
     at (328, 680) {$\mathsf{N}_3$};
  \node[ipe node]
     at (236, 680) {$\mathsf{N}_2$};
  \node[ipe node]
     at (253.782, 735.931) {$j$};
  \node[ipe node]
     at (217.191, 677.124) {$h_1$};
  \node[ipe node]
     at (285.191, 677.124) {$h_2$};
  \draw[dotted, ->]
    (196, 676)
     .. controls (214.6667, 668) and (214.6667, 660) .. (196, 652);
  \node[ipe node, font=\tiny]
     at (204, 648) {$j^{\mbox{-}1}{\circ}h_2{\circ}h_1$};
  \draw
    (188, 608)
     -- (196, 608);
  \draw
    (252, 608)
     -- (260, 608);
  \draw
    (316, 608)
     -- (324, 608);
\end{tikzpicture}